\documentclass[a4paper, reqno, 12pt]{amsart}

\usepackage[usenames,dvipsnames]{color}
\usepackage{amsthm,amsfonts,amssymb,amsmath,amsxtra}
\usepackage[all]{xy}
\SelectTips{cm}{}
\usepackage{xr-hyper}
\usepackage[colorlinks=
   citecolor=Black,
   linkcolor=Red,
   urlcolor=Blue]{hyperref}
\usepackage{verbatim}

\usepackage[margin=1.25in]{geometry}
\usepackage{mathrsfs}

\newcommand{\remind}[1]{{\bf ** #1 **}}

\RequirePackage{xspace}
\RequirePackage{etoolbox}
\RequirePackage{varwidth}
\RequirePackage{enumitem}
\RequirePackage{tensor}
\RequirePackage{mathtools}
\RequirePackage{longtable}
\RequirePackage{multirow}

\usepackage[url=false,doi=false,isbn=false,sorting=nyt,giveninits=true,maxbibnames=4]{biblatex}
\def\le{\leqslant}

\def\L{\Lambda}

\def\o{\omega}

\def\l{\lambda}

\def\<{\langle}
\def\>{\rangle}

\newcommand{{\BG}}{\ensuremath{\mathbb {G}}\xspace}

\newcommand{{\BK}}{\ensuremath{\mathbb {K}}\xspace}

\newtheorem{theorem}{Theorem}
\newtheorem{proposition}[theorem]{Proposition}
\newtheorem{lemma}[theorem]{Lemma}

\newtheorem{corollary}[theorem]{Corollary}

\theoremstyle{definition}

\newtheorem{remark}[theorem]{Remark}

\numberwithin{equation}{section}
\numberwithin{theorem}{section}

\setitemize[0]{leftmargin=*,itemsep=\the\smallskipamount}
\setenumerate[0]{leftmargin=*,itemsep=\the\smallskipamount}

\renewcommand{\to}{%
   \ifbool{@display}{\longrightarrow}{\rightarrow}%
   }
\let\shortmapsto\mapsto
\renewcommand{\mapsto}{%
   \ifbool{@display}{\longmapsto}{\shortmapsto}%
   }
\newlength{\olen}
\newlength{\ulen}
\newlength{\xlen}
\newcommand{\xra}[2][]{%
   \ifbool{@display}%
      {\settowidth{\olen}{$\overset{#2}{\longrightarrow}$}%
       \settowidth{\ulen}{$\underset{#1}{\longrightarrow}$}%
       \settowidth{\xlen}{$\xrightarrow[#1]{#2}$}%
       \ifdimgreater{\olen}{\xlen}%
          {\underset{#1}{\overset{#2}{\longrightarrow}}}%
          {\ifdimgreater{\ulen}{\xlen}%
             {\underset{#1}{\overset{#2}{\longrightarrow}}}
             {\xrightarrow[#1]{#2}}}}%
      {\xrightarrow[#1]{#2}}
   }
\makeatother
\newcommand{\xyra}[2][]{%
   \settowidth{\xlen}{$\xrightarrow[#1]{#2}$}%
   \ifbool{@display}%
      {\settowidth{\olen}{$\overset{#2}{\longrightarrow}$}%
       \settowidth{\ulen}{$\underset{#1}{\longrightarrow}$}%
       \ifdimgreater{\olen}{\xlen}%
          {\mathrel{\xymatrix@M=.12ex@C=3.2ex{\ar[r]^-{#2}_-{#1} &}}}%
          {\ifdimgreater{\ulen}{\xlen}%
             {\mathrel{\xymatrix@M=.12ex@C=3.2ex{\ar[r]^-{#2}_-{#1} &}}}
             {\mathrel{\xymatrix@M=.12ex@C=\the\xlen{\ar[r]^-{#2}_-{#1} &}}}}}%
      {\mathrel{\xymatrix@M=.12ex@C=\the\xlen{\ar[r]^-{#2}_-{#1} &}}}%
   }
\makeatletter
\newcommand{\xla}[2][]{%
   \ifbool{@display}%
      {\settowidth{\olen}{$\overset{#2}{\longleftarrow}$}%
       \settowidth{\ulen}{$\underset{#1}{\longleftarrow}$}%
       \settowidth{\xlen}{$\xleftarrow[#1]{#2}$}%
       \ifdimgreater{\olen}{\xlen}%
          {\underset{#1}{\overset{#2}{\longleftarrow}}}%
          {\ifdimgreater{\ulen}{\xlen}%
             {\underset{#1}{\overset{#2}{\longleftarrow}}}
             {\xleftarrow[#1]{#2}}}}%
      {\xleftarrow[#1]{#2}}
   }
\newcommand{\isoarrow}{%
   \ifbool{@display}{\overset{\sim}{\longrightarrow}}{\xrightarrow\sim}%
   }

\def\blambda{\boldsymbol{\lambda}}
   
\begin{document}

\title[]{Canonical bases of tensor products and positivity properties}

\author[Jiepeng Fang]{Jiepeng Fang}
\address{Department of Mathematics and New Cornerstone Science Laboratory, The University of Hong Kong, Pokfulam, Hong Kong, Hong Kong SAR, China}
\email{fangjp@hku.hk}

\author[Xuhua He]{Xuhua He}
\address{Department of Mathematics and New Cornerstone Science Laboratory, The University of Hong Kong, Pokfulam, Hong Kong, Hong Kong SAR, China}
\email{xuhuahe@hku.hk}

\thanks{}




\begin{abstract}
Let $\mathbf{U}$ be a quantum group of symmetric type. We introduce the {\it thickening realization} to realize (a suitable approximation of) the tensor product ${^{\omega}\Lambda_{\lambda_1}}\otimes \Lambda_{\lambda_2}$ of a simple integrable lowest weight module and a highest weight module as a subquotient of the Verma module of a larger quantum group $\tilde{\mathbf{U}}$. For the canonical basis of the tensor product, we show that the entries of the transition matrix from the pure tensor basis to it, and the structure constants of the action by spherical parabolic subalgebras of the modified quantum group $\dot{\mathbf{U}}$ are given by the structure constants of the comultiplication and multiplication in the negative part $\tilde{\mathbf{U}}^-$ of $\tilde{\mathbf{U}}$ with respect to its canonical basis respectively. Thus, we deduce the positivity property of the canonical basis of the tensor product. 

In particular, we obtain the positivity property of the canonical bases for the action of $\dot{\mathbf{U}}$ on simple integrable highest weight modules, generalizing Lusztig's theorem from Chevalley generators to any canonical basis elements of $\dot{\mathbf{U}}$; for the action of Chevalley generators on ${^{\omega}\Lambda_{\lambda_1}}\otimes \Lambda_{\lambda_2}$; and for multiplication in $\dot{\mathbf{U}}$, as well as the actions on arbitrary tensor products. At $v=1$, these results connect to geometric total positivity on double flag varieties, explored in \cite{HX}.
\end{abstract}

\maketitle


\section*{Introduction}

\subsection{Canonical bases}

The quantum group $\mathbf{U}$, also called the quantized enveloping algebra, was independently introduced by Drinfeld and Jimbo. It is a Hopf algebra over $\mathbb{Q}(v)$, as a non-trivial deformation of the universal enveloping algebra of the Kac-Moody Lie algebra. 

The theory of canonical basis first established by Lusztig \cite{Lusztig-1990, Lusztig-1991} and further studied by Kashiwara \cite{Kashiwara-1991}, reveals profound and elegant structures on the quantum group and its modules. A cornerstone of this theory is the \textbf{positivity property} which is established by Lusztig \cite{Lusztig-1993} for the canonical basis of the positive (resp. negative) parts $\mathbf{U}^+$ (resp. $\mathbf{U}^-$) of $\mathbf{U}$ with respect to the multiplication and the comultiplication, and for the canonical bases of the simple integrable highest (resp. lowest) weight module $\Lambda_{\lambda}$ (resp. ${^{\omega}\Lambda_{\lambda}}$) with respect to the action by Chevalley generators of $\mathbf{U}$. In symmetric cases, these structure constants lie in $\mathbb{N}[v,v^{-1}]$, not merely in $\mathbb{Z}[v,v^{-1}]$. As Lusztig notes \cite{Lusztig-1993}, this positivity is a signature feature, indicating ``a new, extremely rigid structure'' inherent to quantum groups.

The canonical basis of the tensor ${^{\omega}\Lambda_{\lambda_1}}\otimes \Lambda_{\lambda_2}$ of a simple integrable lowest weight module and a highest weight module was constructed by Lusztig \cite{Lusztig-1992}, and later of the tensor product of arbitrary numbers of simple lowest weight and highest weight modules by Bao-Wang \cite{Bao-Wang-2016}. Crucially, the canonical basis of the tensor product is \emph{not} the tensor product of the individual canonical bases. The former is a new basis defined via Lusztig’s quasi-$\mathcal{R}$-matrix and a variant of Kazhdan-Lusztig algorithm. 

A fundamental question is whether the canonical basis of the tensor product also possesses the same positivity property as the individual canonical bases. 

\subsection{Previous works}
A powerful method for establishing the positivity property of canonical bases is their geometrization or categorification. This method realizes the given algebraic objects and their bases in terms of geometric or algebraic categories, where the positivity property often emerges from fundamental geometric or categorical constraints. 

We discuss the quantum group $\mathbf{U}$ associated with symmetric Cartan datum.

Lusztig \cite{Lusztig-1990,Lusztig-1991} built the fundamental framework of the categorification for $\mathbf{U}^-$ by perverse sheaves on the varieties of quiver representations and realized the canonical basis $\mathbf{B}$ by simple perverse sheaves. Khovanov-Lauda \cite{Khovanov-Lauda-2009} and Rouquier \cite{Rouquier-2008} categorified $\mathbf{U}^-$ via quiver Hecke algebras, and later Varagnolo-Vasserot \cite{Varagnolo-Vasserot-2011} and Rouquier \cite{Rouquier-2012} showed that the indecomposable projective modules realized $\mathbf{B}$. These categorifications established the positivity property of $\mathbf{B}$ with respect to the multiplication and the comultiplication in $\mathbf{U}^-$.

Kang-Kashiwara \cite{Kang-Kashiwara-2012} gave a categorification for the simple integrable highest weight module $\Lambda_{\lambda}$ via cyclotomic quiver Hecke algebras, while Zheng \cite{Zheng-2014} and later Fang-Lan-Xiao \cite{Fang-Lan-Xiao-2023} gave geometric realizations by perverse sheaves. These results reconfirmed the positivity property of the canonical basis $\mathbf{B}(\Lambda_{\lambda})$ with respect to the action by the Chevalley generators which had been established by Lusztig \cite{Lusztig-1993}.

For the tensor product of simple integrable highest weight modules, Zheng \cite{Zheng-2014} gave a categorification by perverse sheaves, and Li \cite{Li-2013,Li-2014} showed that simple objects realized the canonical basis. Webster \cite{Webster-2015} gave another categorification building on works of \cite{Khovanov-Lauda-2009,Rouquier-2012,Varagnolo-Vasserot-2011} and established the positivity property of the canonical basis. Fang-Lan \cite{Fang-Lan-2023, Fang-Lan-2025} provided another proof for the positivity property based on works of \cite{Zheng-2014,Li-2014,Fang-Lan-Xiao-2023,Lusztig-1993}. 

The more challenging—and arguably more significant—case involves the tensor product ${^{\omega}\Lambda_{\lambda_1}}\otimes \Lambda_{\lambda_2}$ containing simple lowest weight and highest weight modules at the same time. By a stable limiting process $\lambda\rightarrow \infty$, the canonical bases $\mathbf{B}({^{\omega}\Lambda_{\lambda_1+\lambda}} \otimes \Lambda_{\lambda_2+\lambda})$ glue together to produce the canonical basis $\dot{\mathbf{B}}$ of the modified quantum group $\dot{\mathbf{U}}$. The positivity property of $\mathbf{B}({^{\omega}\Lambda_{\lambda_1}}\otimes \Lambda_{\lambda_2})$ is intimately related to the positivity property of $\dot{\mathbf{B}}$. 

For finite type case, where the distinction between lowest and highest weight modules disappears, the results of \cite{Webster-2015, Fang-Lan-2023, Fang-Lan-2025} already implies the positivity property of $\mathbf{B}({^{\omega}\Lambda_{\lambda_1}}\otimes \Lambda_{\lambda_2})$. For affine type $A$ case, a series of works by Beilinson-Lusztig-MacPherson \cite{Beilinson-Lusztig-MacPherson-1990}, Ginzburg-Vasserot \cite{Ginzburg-Vasserot-1993} and Lusztig \cite{Lusztig-1999,Lusztig-2000} gave a geometric realization of $\dot{\mathbf{U}}$ by perverse sheaves, and Schiffmann-Vasserot \cite{Schiffmann-Vasserot-2000}, McGerty \cite{McGerty-2012} showed that the simple perverse sheaves realized the canonical basis, while Fu-Shoji \cite{Fu-Shoji-2018} gave a Ringel-Hall algebra approach to these works. These results established the positivity property of $\dot{\mathbf{B}}$ with respect to the multiplication in $\dot{\mathbf{U}}$, and implied the positivity of $\mathbf{B}({^{\omega}\Lambda_{\lambda_1}}\otimes \Lambda_{\lambda_2})$. Fan-Li \cite{Fan-Li-2021} and Fu \cite{Fu-2021} also established the positivity property of $\dot{\mathbf{B}}$ with respect to the comultiplication in $\dot{\mathbf{U}}$.

Beyond finite type and affine type $A$ cases, the question of whether the canonical bases of ${^{\omega}\Lambda_{\lambda_1}}\otimes \Lambda_{\lambda_2}$ and $\dot{\mathbf{U}}$ possess the positivity property remained open.  

\subsection{Main result}
In this paper, we consider the quantum group associated with symmetric Cartan datum.

Our method is different from previous categorification approach. We introduce the {\it thickening realization}: enlarge the quantum group $\mathbf{U}$ to a larger one $\tilde{\mathbf{U}}$, then realize (a suitable approximation of) the tensor product module ${^{\omega}\Lambda_{\lambda_1}}\otimes \Lambda_{\lambda_2}$ of $\mathbf{U}$ as a subquotient of a Verma module of $\tilde{\mathbf{U}}$. This realization is compatible with the canonical bases, and provides us with a relation between the canonical basis of the tensor product and the canonical basis of the negative part $\tilde{\mathbf{U}}^-$ of $\tilde{\mathbf{U}}$.

Our first main result expresses the transition matrix between two bases of the tensor product ${^{\omega}\Lambda_{\lambda_1}}\otimes \Lambda_{\lambda_2}$ by the structure constant of the comultiplication in $\tilde{\mathbf{U}}^-$.

\begin{theorem}[\textbf{Theorem \ref{thm:structure-trans}}]\label{first main result}
For any dominant weights $\lambda_1,\lambda_2$, the entries of the transition matrix from the pure tensor basis $\mathbf{B}({^{\omega}\Lambda_{\lambda_1}})\otimes \mathbf{B}(\Lambda_{\lambda_2})$ to the canonical basis $\mathbf{B}({^{\omega}\Lambda_{\lambda_1}}\otimes \Lambda_{\lambda_2})$ of the tensor product ${^{\omega}\Lambda_{\lambda_1}}\otimes \Lambda_{\lambda_2}$ are given by the structure constants of the comultiplication in $\tilde{\mathbf{U}}^-$ with respect to its canonical basis $\tilde{\mathbf{B}}$.  
\end{theorem}

Choosing the approximation of the tensor product ${^{\omega}\Lambda_{\lambda_1}}\otimes \Lambda_{\lambda_2}$ appropriately, the thickening realization is compatible with the action by the spherical parabolic canonical basis elements of $\dot{\mathbf{U}}$. Further enlarging the quantum group $\tilde{\mathbf{U}}$ to $\tilde{\tilde{\mathbf{U}}}$, our second main result relates the multiplication in $\dot{\mathbf{U}}$ with the multiplication in $\tilde{\tilde{\mathbf{U}}}^-$.

\begin{theorem}[\textbf{Theorem \ref{thm:structure-multiplication}}]\label{second main result}
For any $\dot{b},\dot{b}'\in \dot{\mathbf{B}}$, if $\dot{b}$ is spherical parabolic, then the structure constants of the multiplication $\dot{b}\dot{b}'$ in $\dot{\mathbf{U}}$ with respect to $\dot{\mathbf{B}}$ are given by the structure constants of the multiplication in $\tilde{\tilde{\mathbf{U}}}^-$ with respect to its canonical basis $\tilde{\tilde{\mathbf{B}}}$.
\end{theorem}

Combining with the known positivity property of the canonical bases of $\tilde{\mathbf{B}}$ and $\tilde{\tilde{\mathbf{B}}}$ in \cite[Theorem 14.4.13]{Lusztig-1993}, above two results give the following positivity property:

\begin{theorem}[\textbf{Theorem \ref{thm:trans-pos}, Theorem \ref{thm:mult}}]
The canonical bases of the tensor product and the modified quantum group have the following positivity property:
\begin{enumerate}
\item (Positivity of transition matrix) 
for any dominant weights $\l_1$ and $\l_2$, we have 
$$\mathbf{B}({^{\omega}\Lambda_{\lambda_1}}\otimes \Lambda_{\lambda_2})\subset \mathbb{N}[v^{-1}][\mathbf{B}({^{\omega}\Lambda_{\lambda_1}})\otimes \mathbf{B}(\Lambda_{\lambda_2})];$$
\item (Positivity of multiplication) for any $\dot{b},\dot{b}'\in \dot{\mathbf{B}}$, if one of them is spherical parabolic, then 
$$\dot{b}\dot{b}'\in \mathbb{N}[v,v^{-1}][\dot{\mathbf{B}}].$$
\end{enumerate}
\end{theorem}

The positivity of multiplication in above theorem implies the positivity property of $\mathbf{B}({^{\omega}\Lambda_{\lambda_1}}\otimes \Lambda_{\lambda_2})$ with respect to the action by any spherical parabolic $\dot{b}\in \dot{\mathbf{B}}$, and the positivity property of $\mathbf{B}(\Lambda_{\lambda_1}\otimes \Lambda_{\lambda_2})$ with respect to the action by any $\dot{b}\in \dot{\mathbf{B}}$ (see Theorem \ref{Positivity in general case}). As special cases, we obtain the following important consequences:

\begin{corollary}
For any dominant weights $\l_1,\l_2$ and $i \in I$, we have
$$F_i(\mathbf{B}({^{\omega}\Lambda_{\lambda_1}}\otimes \Lambda_{\lambda_2})),E_i(\mathbf{B}({^{\omega}\Lambda_{\lambda_1}}\otimes \Lambda_{\lambda_2}))\subset \mathbb{N}[v,v^{-1}][\mathbf{B}({^{\omega}\Lambda_{\lambda_1}}\otimes \Lambda_{\lambda_2})].$$
\end{corollary}

As an application (at $v=1$), the joint work of Xie and the second-named author \cite{HX} obtains interesting connections between the geometric total positivity on the double flag varieties with the positivity on the tensor product of simple modules.

\begin{corollary}
For any dominant weight $\lambda$ and $\dot{b}\in \dot{\mathbf{B}}$, we have 
$$\dot{b}(\mathbf{B}(\Lambda_{\lambda}))\subset \mathbb{N}[v,v^{-1}][\mathbf{B}(\Lambda_{\lambda})].$$
\end{corollary}

This result establishes the positivity property of the canonical basis of a single simple highest weight module with respect to the action by the canonical basis of the modified quantum group, which generalizes Lusztig's theorem \cite[Theorem 22.1.7]{Lusztig-1993}.

\begin{corollary}
For any dominant weight $\l_1,\l_2$ and $\dot{b} \in \dot{\mathbf{B}}$, we have $$\dot{b}(\mathbf{B}(\Lambda_{\lambda_1}\otimes \Lambda_{\lambda_2}))\subset \mathbb{N}[v,v^{-1}][\mathbf{B}(\Lambda_{\lambda_1}\otimes \Lambda_{\lambda_2})].$$
\end{corollary}

The case that $\dot{b}=E_i1_{\lambda_2-\lambda_1}$ or $F_i1_{\lambda_2-\lambda_1}$ for any $i\in I$ was established in the joint work of Lan and the first-named author in \cite{Fang-Lan-2023,Fang-Lan-2025}. 

In finite type, all canonical basis elements of $\dot{\mathbf{B}}$ are spherical parabolic, and so 

\begin{corollary}
If $\dot{\mathbf{U}}$ is of finite type, then
\begin{enumerate}
\item (Positivity of multiplication) for any $\dot{b},\dot{b}'\in \dot{\mathbf{B}}$, we have  
$$\dot{b}\dot{b}'\in \mathbb{N}[v,v^{-1}][\dot{\mathbf{B}}];$$
\item (Positivity of action) for any  dominant weights $\lambda_1,\l_2$ and $\dot{b}\in \dot{\mathbf{B}}$, we have 
$$\dot{b}(\mathbf{B}({^{\omega}\Lambda_{\lambda_1}}\otimes \Lambda_{\lambda_2}))\subset \mathbb{N}[v,v^{-1}][\mathbf{B}({^{\omega}\Lambda_{\lambda_1}}\otimes \Lambda_{\lambda_2})].$$
\end{enumerate}
\end{corollary}

This result verifies Lusztig's conjecture \cite[Conjecture 25.4.2]{Lusztig-1993} on the positivity of the multiplication in $\dot{\mathbf{U}}$ in finite type case. In this case, the conjecture has been established by Webster in \cite{Webster-2015} via the categorification approach.

Above positivity property of the canonical bases of the tensor products of two simple modules ${^{\omega}\Lambda_{\lambda_1}}\otimes \Lambda_{\lambda_2}$ or $\Lambda_{\lambda_1}\otimes \Lambda_{\lambda_2}$ are generalized to the arbitrary tensor product ${^{\omega}\Lambda_{\lambda_1}}\otimes \cdots \otimes {^{\omega}\Lambda_{\lambda_m}}\otimes \Lambda_{\lambda_{m+1}}\otimes \cdots \otimes \Lambda_{\lambda_{m+n}}$ in Theorem \ref{Positivity in general case}.

\subsection{Motivation and Strategy: the thickening philosophy}
Our approach is inspired by the theory of total positivity. Introduced by Lusztig for reductive groups and their flag varieties in \cite{Lusztig-1994} and later generalized to the Kac-Moody setting in \cite{Lusztig-2019}, \cite{Bao-He-2021}, this theory reveals a profound connection between geometric positivity and its algebraic counterpart. A highlight of this connection is that the totally nonnegative part of a flag variety consists precisely of those points whose corresponding projective lines in $\mathbb{P}(\Lambda_\lambda)$ are spanned by nonnegative linear combinations of the canonical basis $\mathbf{B}(\Lambda_{\lambda})$ (at $v=1$).

We posit an analogous connection between tensor products of modules and twisted products of flag varieties. The study of the totally nonnegative part of such twisted products was initiated by Webster–Yakimov \cite{WY-2007} for $GL_2$, with the general case resolved recently by Bao and the second-named author \cite{Bao-He-2022} using a thickening method: embedding the twisted product of flag varieties for $G$ into a single flag variety for a larger Kac–Moody group $\tilde{G}$. The key philosophy is that the enhanced symmetry of $\tilde{G}$ provides a powerful tool to unravel the intricate geometry of the twisted product (with, a priori, only the symmetry from $G$). 

In this paper, we apply this philosophy to the quantum group $\mathbf{U}$. The larger quantum group $\tilde{\mathbf{U}}$ arises from a variant of Nakajima's framed quiver construction \cite{Nakajima-1998,Nakajima-2001}, which is also considered in \cite{Li-2014,Fang-Lan-2025}. The thickening realization realizes the tensor product ${^\omega V_w(\lambda_1)}\otimes\Lambda_{\lambda_2}\subset {^\omega \Lambda_{\lambda_1}}\otimes\Lambda_{\lambda_2}$ as a subquotient of a Verma module of $\tilde{\mathbf{U}}$ for any Demazure submodule ${^\omega V_w(\lambda_1)}\subset \Lambda_{\lambda_1}$. This realization is compatible with the canonical basis, and provides the canonical basis of ${^\omega V_w(\lambda_1)}\otimes\Lambda_{\lambda_2}$ with the positivity property from the canonical basis of $\tilde{\mathbf{U}}^-$. By the fact that  ${}^\omega\Lambda_{\lambda_1}$ is the direct limit of its Demazure submodules, we obtain the desired positivity property of the canonical basis of ${^\omega \Lambda_{\lambda_1}}\otimes\Lambda_{\lambda_2}$.

\subsection{Structure of the paper}
The logical structure of our approach is illustrated by the following blueprint:
\begin{equation}\label{*}
\tag{$*$}\xymatrix@C=1cm@R=1cm{
\mathbf{B}((\mathbf{f}\theta_{\lambda_2}\mathbf{f})_{-w\lambda_1\odot\lambda_2})  \ar@<.5ex>[r] \ar@{^{(}->}[d] &\mathbf{B}(M_{-w\lambda_1}\otimes \Lambda_{\lambda_2}) \ar@<.5ex>[l] \ar@{^{(}->}[d] \ar@{->>}[r] &\mathbf{B}({^{\omega}V_w(\lambda_1)}\otimes \Lambda_{\lambda_2})\sqcup\{0\} \ar@{^{(}->}[d]\\
(\mathbf{f}\theta_{\lambda_2}\mathbf{f})_{-w\lambda_1\odot\lambda_2} \ar@<.5ex>[r]^-{\varphi_{-w\lambda_1,\lambda_2}} &M_{-w\lambda_1}\otimes \Lambda_{\lambda_2} \ar@<.5ex>[l]^-{\psi_{-w\lambda_1,\lambda_2}} \ar@{->>}[r]^-{a\otimes \mathrm{Id}} &{^{\omega}V_w(\lambda_1)}\otimes \Lambda_{\lambda_2}.}
\end{equation}
\begin{itemize}
\item \S\ref{sec:thick-module} establishes the key $\mathbf{U}$-module isomorphisms $\varphi_{-w\lambda_1,\lambda_2}$ and $\psi_{-w\lambda_1,\lambda_2}$ in the thickening realization for the tensor product ${^{\omega}V_w(\lambda_1)}\otimes \Lambda_{\lambda_2}\subset {^{\omega}\Lambda_{\lambda_1}}\otimes \Lambda_{\lambda_2}$.
\item \S\ref{Thickening realization of canonical basis} studies the compatibility of the thickening realization with the canonical basis and proves Theorem \ref{first main result} (Theorem \ref{thm:structure-trans}).
\item \S \ref{Structure constant of multiplication} applies the thickening realization to compare the multiplication in $\dot{\mathbf{U}}$ and the multiplication in $\tilde{\mathbf{U}}^-$, and proves Theorem \ref{second main result} (Theorem \ref{thm:structure-multiplication}).
\item \S \ref{Positivity of canonical basis of modified quantum group} verifies that the involution $\omega$ preserves the canonical basis $\dot{\mathbf{B}}$ which was a conjecture by Lusztig, and then applies Theorem \ref{thm:structure-multiplication} to establishes the positivity property of the multiplication in $\dot{\mathbf{U}}$ with respect to $\dot{\mathbf{B}}$ (Theorem \ref{thm:mult}).
\item \S\ref{Positivity of canonical basis of tensor product} uses the positivity property of multiplication in $\dot{\mathbf{U}}$ to establish the positivity property of the canonical bases of the tensor product (Theorem \ref{Positivity in general case}).
\end{itemize}

{\bf Acknowledgements} JF is partially supported by the New Cornerstone Science Foundation through the New Cornerstone Investigator Program awarded to XH and National Natural Science Foundation of China (Grant No. 12471030). XH is partially supported by the New Cornerstone Science Foundation through the New Cornerstone Investigator Program. We thank Yiqiang Li, George Lusztig and Weiqiang Wang for useful comments. JF would like to thank Jie Xiao and Yixin Lan for their long-term discussion and collaboration.

\section{Preliminary}\label{Preliminary}

\subsection{Cartan datum and root datum}

A symmetric Cartan datum $(I,\cdot)$ consists of a finite set $I$ and a symmetric bilinear form $\mathbb{Z}[I]\times \mathbb{Z}[I]\rightarrow \mathbb{Z}$ denoted by $(\nu,\nu')\mapsto \nu\cdot \nu'$ such that $i\cdot i=2$ for any $i\in I$, and $i\cdot j\leqslant 0$ for any $i\not=j$ in $I$.

A $Y$-regular root datum $(Y,X,\langle\,,\,\rangle,\ldots)$ of type $(I,\cdot)$ consists of two finitely generated free abelian groups $Y,X$, a perfect bilinear pairing $\langle\,,\,\rangle:Y\times X\rightarrow \mathbb{Z}$ and two embeddings $I\hookrightarrow Y, I\hookrightarrow X$ such that $\langle i,j\rangle=i\cdot j$ for any $i,j\in I$, and the image of the embedding $I\hookrightarrow Y$ is linearly independent in $Y$.

The embeddings $I\hookrightarrow Y$ and $I\hookrightarrow X$ induce the group homomorphisms $\mathbb{Z}[I]\rightarrow Y$ and $\mathbb{Z}[I]\rightarrow X$ respectively. We denote the images of $\nu\in \mathbb{Z}[I]$ under either of these homomorphisms by the same symbol $\nu$.

Throughout this paper, we fix a symmetric Cartan datum $(I,\cdot)$ and a $Y$-regular root datum $(Y,X,\langle\,,\,\rangle,\ldots)$ of type $(I,\cdot)$.

\subsection{The algebra $\mathbf{f}$}\label{The algebra f}

We follow \cite[Chapter 1]{Lusztig-1993}.

Let $\mathbb{Q}(v)$ be the field of rational functions in $v$ with coefficients in $\mathbb{Q}$. We denote
$$[n]=\frac{v^n-v^{-n}}{v-v^{-1}},\ [m]!=\prod^m_{k=1}[k]\in \mathbb{Q}(v)\ \textrm{for any}\ n\in\mathbb{Z},m\in \mathbb{N}.$$

For the symmetric Cartan datum $(I,\cdot)$, the algebra $\mathbf{f}$ is the $\mathbb{Q}(v)$-algebra with unit generated by $\theta_i$ for any $i\in I$, subject to the quantum Serre relations:
$$\sum_{n=0}^{1-i\cdot j}(-1)^n\theta_i^{(n)}\theta_j\theta_i^{(1-i\cdot j-n)}=0\ \textrm{for any $i\not=j$ in $I$},$$
where $\theta_i^{(n)}=\theta^n_i/[n]!$ for any $i\in I$ and $n\in \mathbb{N}$.

The algebra $\mathbf{f}=\bigoplus_{\nu\in \mathbb{N}[I]}\mathbf{f}_\nu$ is $\mathbb{N}[I]$-graded such that $\theta_i\in \mathbf{f}_i$ for any $i\in I$. For any homogeneous $x\in \mathbf{f}_\nu$, we denote $|x|=\nu$. 

The tensor product $\mathbf{f}\otimes \mathbf{f}$ over $\mathbb{Q}(v)$ is an algebra with the multiplication 
$$(x_1\otimes x_2)(y_1\otimes y_2)=v^{|x_2|\cdot |y_1|}x_1y_1\otimes x_2y_2\ \textrm{for any homogeneous}\ x_1,x_2,y_1,y_2\in \mathbf{f}.$$

The comultiplication of $\mathbf{f}$ is the algebra homomorphism $r:\mathbf{f}\rightarrow \mathbf{f}\otimes \mathbf{f}$ defined by $r(\theta_i)=\theta_i\otimes 1+1\otimes \theta_i$ for any $i\in I$.

\subsection{Quantum group}\label{Quantized enveloping algebra}

We refer to \cite[Chapter 3]{Lusztig-1993} for details.

For the $Y$-regular root datum $(Y,X,\langle\,,\,\rangle,\ldots)$ of type $(I,\cdot)$, the quantum group $\mathbf{U}$ is the $\mathbb{Q}(v)$-algebra with unit generated by $E_i,F_i$ and $K_\mu$ for any $i\in I$ and $\mu\in Y$, subject to the relations in \cite[\S 3.1.1]{Lusztig-1993}.

Let $\mathbf{U}^+$ and $\mathbf{U}^-$ be the subalgebras of $\mathbf{U}$ generated by $E_i$ and $F_i$ for any $i\in I$ respectively. Then there is an algebra automorphism $\omega:\mathbf{U}\rightarrow \mathbf{U}$ defined by $\omega(E_i)=F_i, \omega(F_i)=E_i, \omega(K_{\mu})=K_{-\mu}$ for any $i\in I$ and $\mu\in Y$.
There are algebra isomorphisms $\mathbf{f}\rightarrow \mathbf{U}^+, x\mapsto x^+$ and $\mathbf{f}\rightarrow \mathbf{U}^-, x\mapsto x^-$
such that $E_i=\theta_i^+$ and $F_i=\theta_i^-$ for any $i\in I$. Moreover, $\omega(x^+)=x^-$ and $\omega(x^-)=x^+$ for any $x\in \mathbf{f}$.

Let $\mathbf{U}^0$ be the subalgebra of $\mathbf{U}$ generated by $K_{\mu}$ for any $\mu\in Y$. We denote $\mathbf{U}(\mathfrak{b}^+)=\mathbf{U}^0\otimes \mathbf{U}^+$ and $\mathbf{U}(\mathfrak{b}^-)=\mathbf{U}^0\otimes \mathbf{U}^-$. They are subalgebras of $\mathbf{U}$.

The tensor product $\mathbf{U}\otimes \mathbf{U}$ over $\mathbb{Q}(v)$ is an algebra with the multiplication
$$(u_1\otimes u_2)(u'_1\otimes u'_2)=u_1u'_1\otimes u_2u'_2\ \textrm{for any}\ u_1,u_2,u'_1,u'_2\in \mathbf{U}.$$

The comultiplication of $\mathbf{U}$ is the algebra homomorphism $\Delta:\mathbf{U}\rightarrow \mathbf{U}\otimes \mathbf{U}$ defined by $\Delta(E_i)=E_i\otimes 1+K_i\otimes E_i, \Delta(F_i)=1\otimes F_i+F_i\otimes K_{-i}, \Delta(K_\mu)=K_\mu\otimes K_\mu$ for any $i\in I$ and $\mu\in Y$.

\subsection{Weight module}\label{U-module}

We refer to \cite[\S 3.4]{Lusztig-1993} for details.

A weight $\mathbf{U}$-module over $\mathbb{Q}(v)$ is a $\mathbf{U}$-module whose underlying vector space has a direct sum decomposition into weight spaces. All $\mathbf{U}$-modules considered in this paper are weight modules.

For any $\mathbf{U}$-module $M$, we define a new $\mathbf{U}$-module ${^{\omega}M}$. Its underlying vector space is $M$. For any $m \in M$, we denote by $^{\omega} m\in {^{\omega}M}$ the corresponding element. The module structure is defined by $u ({}^{\omega} m)={}^{\omega}(\omega(u) m)$ for any $u\in \mathbf{U}$ and $m\in M$. 

For any $\zeta\in X$, we denote by $M_{\zeta}$ the Verma module of $\mathbf{U}$ with the highest weight $\zeta$. Its underlying vector space is $\mathbf{f}$, and the module structure is defined by $E_i1=0, F_ix=\theta_ix, K_{\mu}x=v^{\langle\mu,\zeta-|x|\rangle}x$ for any $i\in I,\mu\in Y$ and homogeneous $x\in \mathbf{f}$. We denote by $v_{\zeta}\in M_{\zeta}$ the highest weight vector $1\in \mathbf{f}$.

Let $X^+=\{\l \in X; \langle i,\lambda\rangle \in \mathbb{N} \text{ for any } i\in I\}$ be the set of dominant weights. For any $\lambda\in X^+$, the simple integrable $\mathbf{U}$-module with the highest weight $\lambda$ is  
$$\Lambda_\lambda=M_{\lambda}/\sum_{i\in I}\mathbf{U}^-F_i^{\langle i,\lambda\rangle+1}v_{\lambda}.$$
We denote by $\pi_{\lambda}:M_{\lambda}\rightarrow \Lambda_{\lambda}$ the natural projection, and denote by $\eta_{\lambda}=\pi_{\lambda}(v_{\lambda})$ the highest weight vector.

The simple integrable $\mathbf{U}$-module with the lowest weight $-\lambda$ is ${^{\omega}\Lambda_{\lambda}}$. We denote by $\xi_{-\lambda}={}^{\omega}(\eta_\l)$ the lowest weight vector.

\subsection{Modified quantum group}
We refer to \cite[Chapter 23]{Lusztig-1993} for details.

The modified quantum group $\dot{\mathbf{U}}$ is a $\mathbb{Q}(v)$-algebra defined by Lusztig in \cite[\S 23.1.1]{Lusztig-1993}. We have 
$$\dot{\mathbf{U}}=\bigoplus_{\zeta_1,\zeta_2\in X}{_{\zeta_1}\mathbf{U}_{\zeta_2}},\ \textrm{where}\ {_{\zeta_1}\mathbf{U}_{\zeta_2}}=\mathbf{U}/(\sum_{\mu\in Y}(K_{\mu}-v^{\langle \mu,\zeta_1\rangle})\mathbf{U}+\sum_{\mu\in Y}\mathbf{U}(K_{\mu}-v^{\langle \mu,\zeta_2\rangle})).$$
In the meanwhile, $\dot{\mathbf{U}}$ is a free $(\mathbf{f}\otimes \mathbf{f}^{\mathrm{opp}})$-module with a basis $\{1_{\zeta};\zeta\in X\}$, where $1_\zeta$ is the image of $1\in \mathbf{U}$ under the natural projection $\mathbf{U}\rightarrow{_{\zeta}\mathbf{U}_{\zeta}}$ for any $\zeta\in X$. Comparing with the triangular decomposition $\mathbf{U}=\mathbf{U}^+\otimes \mathbf{U}^0\otimes \mathbf{U}^-$ of the quantum group, $\dot{\mathbf{U}}$ is a modified form of $\mathbf{U}$ in which $\mathbf{U}^0$ is replaced by $\bigoplus_{\zeta\in X}\mathbb{Q}(v)1_{\zeta}$.

A $\dot{\mathbf{U}}$-module $M$ over $\mathbb{Q}(v)$ is called {\it unital}, if for any $m\in M$, we have $1_{\zeta}m=0$ for all but finitely many $\zeta\in X$ and $\sum_{\zeta\in X}1_{\zeta}m=m$. Then giving a unital $\dot{\mathbf{U}}$-module is the same as giving a weight $\mathbf{U}$-module (see \cite[\S  23.1.4]{Lusztig-1993}). For this reason, $\dot{\mathbf{U}}$ is an algebra more appropriate than $\mathbf{U}$ for the study of weight modules.

\subsection{Canonical basis}\label{Canonical basis}

We refer to \cite[Chapter 14]{Lusztig-1993} for details.

Let $\mathcal{A}=\mathbb{Z}[v,v^{-1}]\subset \mathbb{Q}(v)$ be the ring of Laurent polynomials with coefficients in $\mathbb{Z}$. The integral form ${_{\mathcal{A}}\mathbf{f}}$ of $\mathbf{f}$ is the $\mathcal{A}$-subalgebra generated by $\theta_i^{(n)}$ for any $i\in I$ and $n\in \mathbb{N}$.

The bar-involution on $\mathbf{f}$ is the $\mathbb{Q}$-algebra involution $\mathbf{f}\rightarrow \mathbf{f},x\mapsto \overline{x}$ defined by $\overline{\theta_i}=\theta_i, \overline{v^n}=v^{-n}$ for any $i\in I$ and $n\in \mathbb{Z}$. 

By \cite[Proposition 1.2.3]{Lusztig-1993}, there is a unique non-degenerate symmetric bilinear form $(\,,\,)_{\mathbf{f}}:\mathbf{f}\times \mathbf{f}\rightarrow \mathbb{Q}(v)$ such that $(1,1)_{\mathbf{f}}=1,(\theta_i,\theta_j)_{\mathbf{f}}=\delta_{i,j}(1-v^{-2})^{-1}$ for any $i,j\in I$, and $(x_1x_2,x)_{\mathbf{f}}=(x_1\otimes x_2,r(x))_{\mathbf{f}\otimes \mathbf{f}}$ for any $x_1,x_2,x\in \mathbf{f}$, where the bilinear form on $\mathbf{f}\otimes \mathbf{f}$ is given by $(x_1\otimes x_2,x'_1\otimes x'_2)_{\mathbf{f}\otimes \mathbf{f}}=(x_1,x'_1)_{\mathbf{f}}(x_2,x'_2)_{\mathbf{f}}$ for any $x_1,x_2,x'_1,x'_2\in \mathbf{f}$.

The canonical basis $\mathbf{B}$ of $\mathbf{f}$ is defined in \cite[Definition 14.4.6]{Lusztig-1993}. It is a $\mathcal{A}$-basis of ${_{\mathcal{A}}\mathbf{f}}$ such that  $\overline{b}=b$ for any $b\in \mathbf{B}$, and $(b,b')_{\mathbf{f}}\in \delta_{b,b'}+v^{-1}\mathbb{Z}[[v^{-1}]]$ for any $b,b'\in \mathbf{B}$. The canonical basis $\mathbf{B}$ has the following positivity property.

\begin{theorem}[{\cite[Theorem 14.4.13]{Lusztig-1993}}]\label{14.4.13}
Let $b_1,b_2,b\in \mathbf{B}$. Then
$$b_1b_2\in \mathbb{N}[v,v^{-1}][\mathbf{B}],\ r(b)\in \mathbb{N}[v,v^{-1}][\mathbf{B}\otimes \mathbf{B}].$$
\end{theorem}

Let $\lambda\in X^+$. The integral form of $\Lambda_{\lambda}$ is ${_{\mathcal{A}}\Lambda_{\lambda}}=\{x^-\eta_{\lambda};x\in {_{\mathcal{A}}\mathbf{f}}\}$. 

The bar-involution on $\mathbf{U}$ is the $\mathbb{Q}$-algebra involution $\mathbf{U}\rightarrow \mathbf{U},u\mapsto \overline{u}$ defined by $\overline{E_i}=E_i, \overline{F_i}=F_i, \overline{K_{\mu}}=K_{-\mu}, \overline{v^n}=v^{-n}$ for any $i\in I, \mu\in Y$ and $n\in \mathbb{Z}$. The bar-involution on $\Lambda_{\lambda}$ is the $\mathbb{Q}$-linear involution $\Lambda_{\lambda}\rightarrow \Lambda_{\lambda},m\mapsto \overline{m}$ defined by $\overline{u\eta_{\lambda}}=\overline{u}\eta_{\lambda}$ for any $u\in \mathbf{U}$. Then $\overline{um}=\bar{u}\bar{m}$ for any $u\in \mathbf{U}$ and $m\in \Lambda_{\lambda}$.

By \cite[Proposition 19.1.2]{Lusztig-1993}, there is a
unique symmetric bilinear form $(\,,\,)_{\Lambda_{\lambda}}:\Lambda_{\lambda}\times \Lambda_{\lambda}\rightarrow \mathbb{Q}(v)$ such that $(\eta_{\lambda},\eta_{\lambda})_{\Lambda_{\lambda}}=1$, and $(um,m')_{\Lambda_{\lambda}}\!\!=\!(m,\rho(u)m')_{\Lambda_{\lambda}}\!$ for any $u\in \mathbf{U}$ and $m,m'\in\Lambda_{\lambda}$, where $\rho:\mathbf{U}\rightarrow \mathbf{U}^{\mathrm{opp}}$ is the algebra isomorphism defined in \cite[\S 19.1.1]{Lusztig-1993}.

The canonical basis $\mathbf{B}(\Lambda_{\lambda})$ of $\Lambda_{\lambda}$ is defined in \cite[Definition 14.4.12]{Lusztig-1993}. By \cite[Theorem 14.4.11]{Lusztig-1993}, $\mathbf{B}(\Lambda_{\lambda})=\{b^-\eta_{\lambda}; b\in \mathbf{B}\}\setminus \{0\}$. It is a $\mathcal{A}$-basis of ${_{\mathcal{A}}\Lambda_{\lambda}}$ such that $\overline{b}=b$ for any $b\in \mathbf{B}(\Lambda_{\lambda})$, and $(b,b')_{\Lambda_{\lambda}}\in \delta_{b,b'}+v^{-1}\mathbb{Z}[v^{-1}]$ for any $b,b'\in \mathbf{B}(\Lambda_{\lambda})$. The canonical basis $\mathbf{B}(\Lambda_{\lambda})$ has the following positivity property.

\begin{theorem}[{\cite[Theorem 22.1.7]{Lusztig-1993}}]\label{22.1.7}
Let $\lambda\in X^+$ and $i\in I$. Then 
$$F_i(\mathbf{B}(\Lambda_{\lambda})), E_i(\mathbf{B}(\Lambda_{\lambda}))\subset \mathbb{N}[v,v^{-1}][\mathbf{B}(\Lambda_{\lambda})].$$
\end{theorem}

Similarly, the simple lowest weight module ${^{\omega}\Lambda_{\lambda}}$ has the integral form, the bar-involution, and the symmetric bilinear form which satisfy similar property. The canonical basis $\mathbf{B}({^{\omega}\Lambda_{\lambda}})=\{b^+\xi_{-\lambda}; b\in \mathbf{B}\}\setminus \{0\}$ of ${^{\omega}\Lambda_{\lambda}}$ satisfies similar property.

\subsection{Based $\mathbf{U}$-module}\label{Based module}

A based $\mathbf{U}$-module $(M,B)$ is an integrable $\mathbf{U}$-module $M$ with a basis $B$ satisfying the conditions in \cite[\S 2.1]{Bao-Wang-2016} which is a generalization of \cite[\S 27.1.2]{Lusztig-1993}. One condition is that the $\mathbb{Q}$-linear involution $M\rightarrow M,m\mapsto \overline{m}$ defined by $\overline{b}=b,\overline{v^n}=v^{-n}$ for any $b\in B$ and $n\in \mathbb{Z}$, is compatible with the bar-involution and the action of $\mathbf{U}$, that is, 
\begin{equation}\label{27.1.2c}
\overline{um}=\bar{u}\bar{m}\ \textrm{for any}\ u\in \mathbf{U}, m\in M.   
\end{equation}
For any based $\mathbf{U}$-modules $(M,B)$ and $(M',B')$, a based $\mathbf{U}$-module homomorphism from $(M,B)$ to $(M',B')$ is a $\mathbf{U}$-module homomorphism $f:M\rightarrow M'$ such that $f(B)\subset B'\sqcup\{0\}$, and $B\cap \mathrm{ker}\,f$ is a basis of $\mathrm{ker}\,f$.

For any $\lambda\in X^+$, there are based $\mathbf{U}$-modules $(\Lambda_{\lambda},\mathbf{B}(\Lambda_{\lambda}))$ and $({^{\omega}\Lambda_{\lambda}},\mathbf{B}({^{\omega}\Lambda_{\lambda}}))$ (see \cite[\S  27.1.4]{Lusztig-1993}).

\subsection{Canonical basis of tensor product}\label{Canonical basis of tensor product}
For any $\mathbf{U}$-modules $M_1$ and $M_2$, the tensor product $M_1\otimes M_2$ over $\mathbb{Q}(v)$ is naturally a $\mathbf{U}\otimes \mathbf{U}$-module. We regard it as a $\mathbf{U}$-module via the comultiplication $\Delta:\mathbf{U}\rightarrow \mathbf{U}\otimes \mathbf{U}$.

For subsets $B_1\subset M_1,B_2\subset M_2$, we denote $B_1 \otimes B_2=\{b_1 \otimes b_2; b_1 \in B_1, b_2 \in B_2\}$. If the elements $b_1\diamondsuit b_2\in M_1\otimes M_2$ have been defined for any $b_1\in B_1$ and $b_2\in B_2$, then we simply denote $B_1\diamondsuit B_2=\{b_1\diamondsuit b_2;b_1\in B_1,b_2\in B_2\}$.
 
Let $\lambda_1,\lambda_2\in X^+$. The tensor products ${^{\omega}\Lambda_{\lambda_1}}\otimes \Lambda_{\lambda_2}$ and $\Lambda_{\lambda_1}\otimes \Lambda_{\lambda_2}$ have bases $\mathbf{B}({^{\omega}\Lambda_{\lambda_1}})\otimes \mathbf{B}(\Lambda_{\lambda_2})$ and $\mathbf{B}(\Lambda_{\lambda_1})\otimes \mathbf{B}(\Lambda_{\lambda_2})$ respectively. The bar-involutions on them are defined by
$\overline{m_1\otimes m_2}=\overline{m_1}\otimes \overline{m_2}$ for any $m_1\in {^{\omega}\Lambda_{\lambda_1}}$ or $\Lambda_{\lambda_1}$ and $m_2\in \Lambda_{\lambda_2}$. Note that in general $\Delta(\overline{u})\not=\overline{\Delta(u)}$ for $u \in \mathbf{U}$, and so these bar-involutions do not satisfy \eqref{27.1.2c}. Hence $({^{\omega}\Lambda_{\lambda_1}}\otimes \Lambda_{\lambda_2},\mathbf{B}({^{\omega}\Lambda_{\lambda_1}})\otimes \mathbf{B}(\Lambda_{\lambda_2}))$ and $(\Lambda_{\lambda_1}\otimes \Lambda_{\lambda_2},\mathbf{B}(\Lambda_{\lambda_1})\otimes \mathbf{B}(\Lambda_{\lambda_2}))$
are not based $\mathbf{U}$-modules.

Lusztig defined a new $\mathbb{Q}$-linear involution on the tensor products satisfying \eqref{27.1.2c}. By \cite[\S  14.2.5]{Lusztig-1993}, we have $\mathbf{B}=\bigsqcup_{\nu\in \mathbb{N}[I]}\mathbf{B}_{\nu}$, where $\mathbf{B}_{\nu}=\mathbf{B}\cap \mathbf{f}_{\nu}$ is a basis of $\mathbf{f}_{\nu}$. For any $\nu=\sum_{i\in I}\nu_ii 
\in \mathbb{N}[I]$, we denote $\mathrm{tr}\,\nu=\sum_{i\in I}\nu_i\in \mathbb{N}$. Let $\{b^*; b\in \mathbf{B}_{\nu}\}$ be the basis of $\mathbf{f}_{\nu}$ dual to the basis $\mathbf{B}_{\nu}$ with respect to $(\,,\,)_{\mathbf{f}}$. The quasi-$\mathcal{R}$-matrices on ${^{\omega}\Lambda_{\lambda_1}}\otimes \Lambda_{\lambda_2}$ and $\Lambda_{\lambda_1}\otimes \Lambda_{\lambda_2}$ are the linear transformations on them defined by 
$$\Theta(m_1\otimes m_2)=\sum_{\nu\in \mathbb{N}[I]}(-v)^{\mathrm{tr}\,\nu}\sum_{b\in \mathbf{B}_{\nu}}b^-m_1\otimes b^{*+}m_2\ \textrm{for any}\ m_1\in {^{\omega}\Lambda_{\lambda_1}}\ \textrm{or}\ \Lambda_{\lambda_1}, m_2\in \Lambda_{\lambda_2}.$$
Note that there are always only finitely many non-zero terms in the summation. By \cite[Lemma 24.1.2]{Lusztig-1993}, $\Delta(\overline{u})\Theta(m_1\otimes m_2)=\Theta\overline{\Delta(u)}(m_1\otimes m_2)$ for any $u\in \mathbf{U}$, $m_1\in {^{\omega}\Lambda_{\lambda_1}}$ or $\Lambda_{\lambda_1}$ and $m_2\in \Lambda_{\lambda_2}$. The quasi-$\mathcal{R}$-matrix is an intertwiner between the composition of the comultiplication and the bar-involution taken in different orders, which is similar to the universal $\mathcal{R}$-matrix introduced by Drinfeld as an intertwiner between the comultiplication and its transpose. The $\Psi$-involutions on ${^{\omega}\Lambda_{\lambda_1}}\otimes \Lambda_{\lambda_2}$ and $\Lambda_{\lambda_1}\otimes \Lambda_{\lambda_2}$ are the $\mathbb{Q}$-linear transformations on them defined by
\begin{equation}\label{Phi-involution}
\Psi(m_1\otimes m_2)=\Theta(\overline{m_1}\otimes \overline{m_2})\ \textrm{for any}\ m_1\in {^{\omega}\Lambda_{\lambda_1}}\ \textrm{or}\ \Lambda_{\lambda_1}, m_2\in \Lambda_{\lambda_2}.
\end{equation}
Then $\Psi(u(m_1\otimes m_2))=\overline{u}\Psi(m_1\otimes m_2)$ for any $u\in \mathbf{U}$, $m_1\in {^{\omega}\Lambda_{\lambda_1}}$ or $\Lambda_{\lambda_1}, m_2\in \Lambda_{\lambda_2}$, that is, the $\Psi$-involutions satisfy \eqref{27.1.2c}.

The canonical basis of the tensor product ${^{\omega}\Lambda_{\lambda_1}}\otimes \Lambda_{\lambda_2}$ was established by Lusztig in \cite[Theorem 24.3.3]{Lusztig-1993}.

\begin{theorem}\label{24.3.3} 
Let $\l_1, \l_2 \in X^+$. Then for any $b_1\in \mathbf{B}({^{\omega}\Lambda_{\lambda_1}})$ and $b_2\in \mathbf{B}(\Lambda_{\lambda_2})$, there exists a unique $b_1 \diamondsuit b_2 \in \mathbb{Z}[v^{-1}][\mathbf{B}({^{\omega}\Lambda_{\lambda_1}})\otimes \mathbf{B}(\Lambda_{\lambda_2})]$ such that
\begin{align*}
\Psi(b_1 \diamondsuit b_2)=b_1 \diamondsuit b_2 \textrm{ and } b_1 \diamondsuit b_2-b_1 \otimes b_2 \in v^{-1}\mathbb{Z}[v^{-1}][\mathbf{B}({^{\omega}\Lambda_{\lambda_1}})\otimes \mathbf{B}(\Lambda_{\lambda_2})].
\end{align*}
Moreover, the set $\mathbf{B}({^{\omega}\Lambda_{\lambda_1}}\otimes \Lambda_{\lambda_2})=:\mathbf{B}({^{\omega}\Lambda_{\lambda_1}}) \diamondsuit \mathbf{B}(\Lambda_{\lambda_2})$ is a basis of ${^{\omega}\Lambda_{\lambda_1}}\otimes \Lambda_{\lambda_2}$. The transition matrix from $\mathbf{B}(^{\omega}\Lambda_{\lambda_1})\otimes  \mathbf{B}(\Lambda_{\lambda_2})$ to $\mathbf{B}({^{\omega}\Lambda_{\lambda_1}}\otimes \Lambda_{\lambda_2})$ is unitriangular with off-diagonal entries in $v^{-1}\mathbb{Z}[v^{-1}]$.
\end{theorem}

The basis $\mathbf{B}({^{\omega}\Lambda_{\lambda_1}}\otimes \Lambda_{\lambda_2})$ is the canonical basis of the tensor product ${^{\omega}\Lambda_{\lambda_1}}\otimes \Lambda_{\lambda_2}$. Moreover, $({^{\omega}\Lambda_{\lambda_1}}\otimes \Lambda_{\lambda_2},\mathbf{B}({^{\omega}\Lambda_{\lambda_1}}\otimes \Lambda_{\lambda_2}))$ is a based $\mathbf{U}$-module.

The canonical basis of the tensor product $\Lambda_{\lambda_1}\otimes \Lambda_{\lambda_2}$ was established by Bao and Wang in \cite[Theorem 2.7]{Bao-Wang-2016}.
\begin{theorem}\label{Bao-Wang-2.7}
Let $\l_1, \l_2 \in X^+$. Then for any $b_1\in \mathbf{B}(\Lambda_{\lambda_1})$ and $b_2\in \mathbf{B}(\Lambda_{\lambda_2})$, there exists a unique $b_1\diamondsuit b_2\in \mathbb{Z}[v^{-1}][\mathbf{B}(\Lambda_{\lambda_1})\otimes \mathbf{B}(\Lambda_{\lambda_2})]$ such that
\begin{align*}
\Psi(b_1 \diamondsuit b_2)=b_1 \diamondsuit b_2 \textrm{ and } b_1 \diamondsuit b_2-b_1 \otimes b_2 \in v^{-1}\mathbb{Z}[v^{-1}][\mathbf{B}(\Lambda_{\lambda_1})\otimes \mathbf{B}(\Lambda_{\lambda_2})].
\end{align*}
Moreover, the set $\mathbf{B}(\Lambda_{\lambda_1}\otimes \Lambda_{\lambda_2}):=\mathbf{B}({\Lambda_{\lambda_1}}) \diamondsuit \mathbf{B}(\Lambda_{\lambda_2})$ is a basis of $\Lambda_{\lambda_1}\otimes \Lambda_{\lambda_2}$. The transition matrix from $\mathbf{B}(\Lambda_{\lambda_1})\otimes  \mathbf{B}(\Lambda_{\lambda_2})$ to $\mathbf{B}(\Lambda_{\lambda_1}\otimes \Lambda_{\lambda_2})$ is unitriangular with off-diagonal entries in $v^{-1}\mathbb{Z}[v^{-1}]$
\end{theorem}

The basis $\mathbf{B}(\Lambda_{\lambda_1}\otimes \Lambda_{\lambda_2})$ is the canonical basis of the tensor product ${\Lambda_{\lambda_1}}\otimes \Lambda_{\lambda_2}$. Moreover, $({\Lambda_{\lambda_1}}\otimes \Lambda_{\lambda_2},\mathbf{B}({\Lambda_{\lambda_1}}\otimes \Lambda_{\lambda_2}))$ is a based $\mathbf{U}$-module.

\subsection{Canonical basis of modified quantum group}

The integral form ${_{\mathcal{A}}\dot{\mathbf{U}}}$ of $\dot{\mathbf{U}}$ is the $\mathcal{A}$-subalgebra generated by $E_i^{(n)}1_{\zeta},F_i^{(n)}1_{\zeta}$ for any $i\in I,n\in \mathbb{N}$ and $\zeta\in X$.
 
The canonical basis $\mathbf{B}({^{\omega}\Lambda_{\lambda_1}}\otimes \Lambda_{\lambda_2})$ of the tensor product ${^{\omega}\Lambda_{\lambda_1}}\otimes \Lambda_{\lambda_2}$ for any $\lambda_1,\lambda_2\in X^+$ has very nice stability property in the sense of \cite[\S  25.1]{Lusztig-1993}. The canonical basis of $\dot{\mathbf{U}}$ was established by Lusztig in \cite[Theorem 25.2.1]{Lusztig-1993} by gluing $\mathbf{B}({^{\omega}\Lambda_{\lambda_1}}\otimes \Lambda_{\lambda_2})$ for any $\lambda_1,\lambda_2\in X^+$.

\begin{theorem}\label{25.2.1}
Let $\zeta\in X$ and $b_1,b_2\in \mathbf{B}$. Then 
\begin{enumerate}
\item there exists a unique $b_1\diamondsuit_{\zeta}b_2\in {_{\mathcal{A}}\dot{\mathbf{U}}1_{\zeta}}$ such that 
$$(b_1\diamondsuit_{\zeta}b_2)(\xi_{-\lambda_1}\otimes \eta_{\lambda_2})=b_1^+\xi_{-\lambda_1}\diamondsuit b_2^-\eta_{\lambda_2}\in \mathbf{B}({^{\omega}\Lambda_{\lambda_1}}\otimes \Lambda_{\lambda_2})$$
for any $\lambda_1,\lambda_2\in X^+$ such that $\zeta=\lambda_2-\lambda_1$ and $b_1^+\xi_{-\lambda_1}\neq0, b_2^-\eta_{\lambda_2}\neq0$;
\item we have $(b_1\diamondsuit_{\zeta}b_2)(\xi_{-\lambda_1}\otimes \eta_{\lambda_2})=0$
for any $\lambda_1,\lambda_2\in X^+$ such that $\zeta=\lambda_2-\lambda_1$ and either $b_1^+\xi_{-\lambda_1}=0$ or $b_2^-\eta_{\lambda_2}=0$;
\item the set $\{b_1\diamondsuit_{\zeta}b_2;\zeta\in X,b_1,b_2\in \mathbf{B}\}$ is a basis of $\dot{\mathbf{U}}$.
\end{enumerate}
\end{theorem}

The set $\dot{\mathbf{B}}:=\{b_1\diamondsuit_{\zeta}b_2;\zeta\in X,b_1,b_2\in \mathbf{B}\}$ is the canonical basis of $\dot{\mathbf{U}}$. 

By \cite[Proposition 25.2.6]{Lusztig-1993}, $b^-1_{\zeta}=1\diamondsuit_{\zeta}b,\ b^+1_{\zeta}=b\diamondsuit_{\zeta}1\in \dot{\mathbf{B}}$ for any $\zeta\in X$ and $b\in \mathbf{B}$. Hence $\dot{\mathbf{B}}$ is a generalization of $\mathbf{B}$.

\subsection{Canonical basis of Demazure module}

For the symmetric Cartan datum $(I,\cdot)$, the Weyl group $W$ is generated by $s_i$ for any $i\in I$, subject to the relations in \cite[\S  2.1.1]{Lusztig-1993}. The group $W$ acts on $Y$ and $X$ (see \cite[\S  2.2.6]{Lusztig-1993}).

Let $\lambda\in X^+$ and $w\in W$. By \cite[Proposition 5.2.7]{Lusztig-1993}, the $w\lambda$-weight subspace $\Lambda_{\lambda}$ is one-dimensional. Let $\eta_{w\lambda}\in \mathbf{B}(\Lambda_{\lambda})$ be the unique canonical basis element of $\Lambda_{\lambda}$ of weight $w\lambda$. The {\it Demazure module} $V_{w}(\lambda)$ is the $\mathbf{U}(\mathfrak{b}^+)$-submodule of $\Lambda_{\lambda}$ generated by $\eta_{w\lambda}$. 

By \cite[Proposition 3.2.3]{Kashiwara-1993}, the set $\mathbf{B}(V_{w}(\lambda)):=\mathbf{B}(\L_\l)\cap V_w(\l)$ is a basis of $V_{w}(\lambda)$. We call it the {\it canonical basis} of $V_w(\l)$. 

By \cite[Corollary 3.2.2]{Kashiwara-1993}, we have $V_{w}(\lambda)\subset V_{w'}(\lambda)$ for any $w,w'\in W$ with $w\leqslant w'$ under the Bruhat order. Let $\ast$ be the Demazure product on $W$ (see \cite[\S 1.2]{He-Nie-2024}). Then $w, w' \le w \ast w'$ for any $w, w'\in W$. Thus $(W, \le)$ is a directed set, and $\{V_w(\lambda);w\in W\}$ forms a direct system. We have $\varinjlim_{w\in W} V_{w}(\lambda)=\Lambda_{\lambda}$, and so $\mathbf{B}(\Lambda_{\lambda})=\bigcup_{w\in W}\mathbf{B}(V_{w}(\lambda))$. This enables us to approximate the canonical basis of simple modules by their Demazure submodules.

Similarly, let $\xi_{-w\lambda}={^{\omega}\eta_{w\lambda}}\in \mathbf{B}({^{\omega}\Lambda_{\lambda}})$ be the unique canonical basis element of ${^{\omega}\Lambda_{\lambda}}$ of weight $-w\lambda$. The {\it Demazure module} ${^{\omega}V_{w}(\lambda)}$ is the $\mathbf{U}(\mathfrak{b}^-)$-submodule of ${^{\omega}\Lambda_{\lambda}}$ generated by $\xi_{-w\lambda}$. It has the {\it canonical basis} $\mathbf{B}({^{\omega}V_{w}(\lambda)}):=\mathbf{B}({}^\o \L_\l)\cap {}^\o V_w(\l)$. We have $\varinjlim_{w\in W}{^{\omega}V_{w}(\lambda)}={^{\omega}\Lambda_{\lambda}}$ and $\mathbf{B}({^{\omega}\Lambda_{\lambda}})=\bigcup_{w\in W}\mathbf{B}({^{\omega}V_{w}(\lambda)})$.

Recall that $M_{-w\lambda}$ is the Verma module with the highest weight $-w\lambda$. There is a surjective $\mathbf{U}(\mathfrak{b}^-)$-module homomorphism $a: M_{-w \l} \to {^{\omega}V_{w}(\lambda)},x^-v_{-w\lambda}\mapsto x^-\xi_{-w\lambda}$. We call it the {\it projection map}.

The underlying vector space of $M_{-w\lambda}$ is $\mathbf{f}$, thus the canonical basis $\mathbf{B}$ of $\mathbf{f}$ can be carried over to $M_{-w\lambda}$. We denote it by $\mathbf{B}(M_{-w\lambda}):=\{b^-v_{-w\lambda_1};b\in \mathbf{B}\}$. The following proposition follows from \cite[Lemma 8.2.1]{Kashiwara-1994}. 

\begin{proposition}\label{action map is a based map}
Let $\l \in X^+,w \in W$. Then $a(\mathbf{B}(M_{-w\lambda}))\subset \mathbf{B}({^{\omega}V_{w}(\lambda)})\sqcup \{0\}$, and the set $\mathbf{B}(M_{-w\lambda})\cap \mathrm{ker}\,a$ is a basis of $\mathrm{ker}\,a$.
\end{proposition}

\begin{proposition}\label{canonical basis of demazure otimes highest weight}
Let $\lambda_1,\lambda_2\in X^+,w\in W$. Then the subset $\mathbf{B}({^{\omega}V_{w}(\lambda_1)}\otimes \Lambda_{\lambda_2}):=\mathbf{B}(^\o \L_{\l_1}) \diamondsuit \mathbf{B}(\Lambda_{\lambda_2})$ of $\mathbf{B}({^{\omega}\Lambda_{\lambda_1}}\otimes \Lambda_{\lambda_2})$ is a basis of ${^{\omega}V_{w}(\lambda_1)}\otimes \Lambda_{\lambda_2}$. 
\end{proposition}
\begin{proof}
Note that the $\Psi$-involution on the tensor product ${^{\omega}\Lambda_{\lambda_1}}\otimes \Lambda_{\lambda_2}$ preserves the subspace ${^{\omega}V_{w}(\lambda_1)}\otimes \Lambda_{\lambda_2}$. Indeed, for any $b_1\in \mathbf{B}({^{\omega}V_{w}(\lambda_1)})$ and $b_2\in \mathbf{B}(\Lambda_{\lambda_2})$, by \eqref{Phi-involution}, we have 
\begin{align*}
\Psi(b_1\otimes b_2)=\sum_{\nu\in \mathbb{N}[I]}(-v)^{\mathrm{tr}\,\nu}\sum_{b\in \mathbf{B}_{\nu}}b^-b_1\otimes {b^*}^+ b_2\in {^{\omega}V_{w}(\lambda_1)}\otimes \Lambda_{\lambda_2}.
\end{align*}
By the proof of Theorem \ref{24.3.3}, $b_1\diamondsuit b_2\in {^{\omega}V_{w}(\lambda_1)}\otimes \Lambda_{\lambda_2}$, and the transition matrix from the basis $\mathbf{B}({^{\omega}V_{w}(\lambda_1)})\otimes \mathbf{B}(\Lambda_{\lambda_2})$ to the set $\mathbf{B}({^{\omega}V_{w}(\lambda_1)}\otimes \Lambda_{\lambda_2})$ is unitriangular. So $\mathbf{B}({^{\omega}V_{w}(\lambda_1)}\otimes \Lambda_{\lambda_2})$ is also a basis of ${}^\o V_w(\l_1))\otimes \L_{\l_2}$.
\end{proof}

We call $\mathbf{B}({^{\omega}V_{w}(\lambda_1)}\otimes \Lambda_{\lambda_2})$ the {\it canonical bases} of the tensor product ${^{\omega}V_{w}(\lambda_1)}\otimes \Lambda_{\lambda_2}$. By definition, we have $\mathbf{B}({^{\omega}V_{w}(\lambda_1)}\otimes \Lambda_{\lambda_2})=\mathbf{B}({^{\omega}\Lambda_{\lambda_1}}\otimes \Lambda_{\lambda_2})\cap ({^{\omega}V_{w}(\lambda_1)}\otimes \Lambda_{\lambda_2})$.

\section{Thickening realization of tensor product}\label{sec:thick-module}

\subsection{Thickening product}\label{The thickening product}

For the symmetric Cartan datum $(I,\cdot)$, we define its {\it thickening Cartan datum} to be the symmetric Cartan datum $(\tilde{I},\cdot)$, where $\tilde{I}=I\sqcup I'$ is obtained from $I$ by adding a copy $I'=\{i';i\in I\}$, and the symmetric bilinear form $\mathbb{Z}[\tilde{I}]\times \mathbb{Z}[\tilde{I}]\rightarrow \mathbb{Z}$ is extended from that in $(I,\cdot)$ such that 
$$i\cdot j'=-\delta_{i,j},\ i'\cdot j'=2\delta_{i,j}\ \textrm{for any}\ i,j\in I.$$

For the $Y$-regular root datum $(Y,X,\langle\,,\,\rangle,\ldots)$ of type $(I,\cdot)$, we define its {\it thickening root datum} to be the $\tilde{Y}$-regular root datum $(\tilde{Y},\tilde{X},\langle\,,\,\rangle,\ldots)$ of type $(\tilde{I},\cdot)$, where $\tilde{Y}=Y\oplus \mathbb{Z}[I'],\tilde{X}=X\oplus \mathrm{Hom}_{\mathbb{Z}}(\mathbb{Z}[I'],\mathbb{Z})$, the bilinear paring $\tilde{Y}\times \tilde{X}\rightarrow \mathbb{Z}$ and the embeddings $\tilde{I}\hookrightarrow \tilde{Y},\tilde{I}\hookrightarrow \tilde{X}$ are extended from those in $(Y,X,\langle\,,\,\rangle,\ldots)$ respectively (see \cite[\S 3.1]{Fang-Lan-2025}).

Let $\tilde{\mathbf{f}}$ and $\tilde{\mathbf{U}}$ be the algebra and the quantum group defined in \S \ref{The algebra f} and \S \ref{Quantized enveloping algebra} associated with the thickening data $(\tilde{I},\cdot)$ and $(\tilde{Y},\tilde{X},\langle\,,\,\rangle,\ldots)$ respectively. Then the natural embedding $I\hookrightarrow \tilde{I}$ induces subalgebra embeddings $\mathbf{f}\hookrightarrow\tilde{\mathbf{f}}$ and $\mathbf{U}\hookrightarrow\tilde{\mathbf{U}}$.

Let $\lambda\in X^+$. Since $i'\cdot j'=0,\theta_{i'}\theta_{j'}=\theta_{j'}\theta_{i'}$ for any $i\not=j$ in $I$, the element 
$$\theta_{\lambda}=\prod_{i\in I}\theta_{i'}^{(\langle i,\lambda\rangle)}\in \tilde{\mathbf{f}}$$
is independent of the order in the product. 

For any $\zeta\in X$ and $\lambda\in X^+$, we regard them as elements of $\tilde{X}$ via the natural embedding $X\hookrightarrow \tilde{X}$, and define $\zeta\odot\lambda\in \tilde{X}$ such that
\begin{equation}\label{definition of odot}
\begin{aligned}
&\langle \mu,\zeta\odot\lambda\rangle=\langle\mu,\zeta+\lambda+|\theta_{\lambda}|\rangle\ \textrm{for any}\ \mu\in Y;\\
&\langle i',\zeta\odot\lambda\rangle=\langle i,\lambda\rangle\ \textrm{for any}\ i\in I.
\end{aligned}
\end{equation}
We call $\zeta\odot\lambda$ the {\it thickening product} of $\zeta$ and $\lambda$. This product ``thickens'' the weight $\zeta$ by incorporating the data of $\lambda$ into the new $I'$-components. By definition, $\langle i,\lambda+|\theta_{\lambda}|\rangle=0,\langle i,\zeta\odot\lambda\rangle=\langle i,\zeta\rangle$ for any $i\in I$. So $\zeta\odot\lambda\in \tilde{X}^+$ when $\zeta,\lambda\in X^+$.

\subsection{Realization of tensor product}\label{ftheta_lambdaf}

In this subsection, we establish the thickening realization for the tensor product of a Verma module with a simple highest weight module. In particular, we have the following $\mathbf{U}$-module isomorphisms 
$$\xymatrix@C=2cm{(\mathbf{f}\theta_{\lambda_2}\mathbf{f})_{-w\lambda_1\odot\lambda_2} \ar@<.5ex>[r]^{\varphi_{-w\lambda_1,\lambda_2}} &M_{-w\lambda_1}\otimes \Lambda_{\lambda_2} \ar@<.5ex>[l]^{\psi_{-w\lambda_1,\lambda_2}}}$$
in the diagram \eqref{*}.

Let $\lambda\in X^+$. We define $\mathbf{f}\theta_{\lambda}\mathbf{f}$ to be the subspace of $\tilde{\mathbf{f}}$ spanned by $x\theta_{\lambda}y$ for any $x,y\in \mathbf{f}$. Let $\tilde{\zeta}\in \tilde{X}$, and $M_{\tilde{\zeta}}$ be the Verma module of $\tilde{\mathbf{U}}$ with the highest weight $\tilde{\zeta}$. Then $M_{\tilde{\zeta}}$ is $\mathbf{U}$-module via the embedding $\mathbf{U}\hookrightarrow \tilde{\mathbf{U}}$. We define $$(\mathbf{f}\theta_{\lambda}\mathbf{f})_{\tilde{\zeta}}=\{x^-v_{\tilde{\zeta}};x\in \mathbf{f}\theta_{\lambda}\mathbf{f}\}.$$ 
Similar to \cite[Lemma 3.1]{Fang-Lan-2025}, $(\mathbf{f}\theta_{\lambda}\mathbf{f})_{\tilde{\zeta}}$ is a $\mathbf{U}$-submodule of $M_{\tilde{\zeta}}$. In particular, for $\zeta\in X$ and $\zeta\odot\lambda\in \tilde{X}$, we have the $\mathbf{U}$-module $(\mathbf{f}\theta_{\lambda}\mathbf{f})_{\zeta\odot\lambda}$ of $M_{\zeta\odot\lambda}$.

Analogously to \cite[\S  3.2, (18)]{Li-2014}, we construct the $\mathbf{U}$-module homomorphism $\psi_{\zeta,\lambda}:M_{\zeta}\otimes \Lambda_{\lambda}\rightarrow (\mathbf{f}\theta_{\lambda}\mathbf{f})_{\zeta\odot\lambda}$. Let $\mathbb{Q}(v)_{\lambda}$ be the one-dimensional $\mathbf{U}(\mathfrak{b^+})$-module with the basis $\{\rho_{\lambda}\}$ and the module structure defined by $E_i\rho_{\lambda}=0,K_{\mu}\rho_{\lambda}=v^{\langle \mu,\lambda\rangle}\rho_{\lambda}$ for any $i\in I$ and $\mu\in Y$. Then $M_{\lambda}=\mathbf{U}\otimes_{\mathbf{U}(\mathfrak{b^+})}\mathbb{Q}(v)_{\lambda}$. We can check that the map $M_{\zeta}\otimes \mathbb{Q}(v)_{\lambda}\rightarrow (\mathbf{f}\theta_{\lambda}\mathbf{f})_{\zeta\odot\lambda},x^-v_{\zeta}\otimes \rho_{\lambda}\mapsto \rho_{\lambda}\theta_{\lambda}^-x^-v_{\zeta\odot\lambda}$ for any $x\in \mathbf{f}$, is a $\mathbf{U}(\mathfrak{b^+})$-module homomorphism. By the tensor identity $M_{\zeta}\otimes M_{\lambda}\cong \mathbf{U}\otimes_{\mathbf{U}(\mathfrak{b}^+)}(M_{\zeta}\otimes \mathbb{Q}(v)_{\lambda})$ and the Frobenius reciprocity 
$$\mathrm{Hom}_{\mathbf{U}}(M_{\zeta}\otimes M_{\lambda},(\mathbf{f}\theta_{\lambda}\mathbf{f})_{\zeta\odot\lambda})\cong \mathrm{Hom}_{\mathbf{U}(\mathfrak{b^+})}(M_{\zeta}\otimes \mathbb{Q}(v)_{\lambda},(\mathbf{f}\theta_{\lambda}\mathbf{f})_{\zeta\odot\lambda}),$$
we obtain a $\mathbf{U}$-module homomorphism $\psi_0:M_{\zeta}\otimes M_{\lambda}\rightarrow (\mathbf{f}\theta_{\lambda}\mathbf{f})_{\zeta\odot\lambda}$ such that 
$$\psi_0(x^-v_{\zeta}\otimes v_{\lambda})=\theta_{\lambda}^-x^-v_{\zeta\odot\lambda}\ \textrm{for any}\ x\in \mathbf{f}.$$
For any $i\in I$, we can prove by induction on $n\in \mathbb{N}$ that $\psi_0(x^-v_{\zeta}\otimes F_i^{(n)}v_{\lambda})=\sum_{k=0}^{n} (-1)^k v^{-k(\langle i,\lambda\rangle+1-n)}F_i^{(n-k)}\theta_{\lambda}^-F_i^{(k)}x^-v_{\zeta\odot\lambda}$ for any $x\in \mathbf{f}$. Then we have 
$$\psi_0(x^-v_{\zeta}\otimes F_i^{(\langle i,\lambda\rangle+1)}v_{\lambda})=\!\!\!\sum_{k=0}^{\langle i,\lambda\rangle+1}\!\! (-1)^k F_i^{(\langle i,\lambda\rangle+1-k)}F_{i'}^{(\langle i,\lambda\rangle)}F_i^{(k)}\!\prod_{j\in I;j\not=i}F_{j'}^{(\langle j
,\lambda\rangle)}x^-v_{\zeta\odot\lambda}=0$$
by the higher order quantum Serre relations in \cite[Proposition 7.1.5]{Lusztig-1993}. Moreover, we can to prove by induction on $\mathrm{tr}\,|y|\in \mathbb{N}$ that $\psi_0(x^-v_{\zeta}\otimes y^-F_i^{(\langle i,\lambda\rangle+1)}v_{\lambda})=0$ for any homogeneous $x,y\in \mathbf{f}$. As a result, $\psi_0(M_{\zeta}\otimes \sum_{i\in I}\mathbf{U}^-F_i^{\langle i,\lambda\rangle+1}v_{\lambda})=0$, and so $\psi_0$ induces a $\mathbf{U}$-module homomorphism $\psi_{\zeta,\lambda}:M_{\zeta}\otimes \Lambda_{\lambda}\rightarrow (\mathbf{f}\theta_{\lambda}\mathbf{f})_{\zeta\odot\lambda}$ such that 
\begin{equation}\label{im psi}
\psi_{\zeta,\lambda}(x^-v_{\zeta}\otimes \eta_{\lambda})=\theta_{\lambda}^-x^-v_{\zeta\odot\lambda}\ \textrm{for any}\ x\in \mathbf{f} 
\end{equation}

Next we introduce the map $\varphi_{\zeta,\lambda}:(\mathbf{f}\theta_{\lambda}\mathbf{f})_{\zeta\odot\lambda}\rightarrow M_{\zeta}\otimes \Lambda_{\lambda}$.

\begin{lemma}\label{linear map varphi}
There exists a linear map $\varphi_{\zeta,\lambda}:(\mathbf{f}\theta_{\lambda}\mathbf{f})_{\zeta\odot\lambda}\rightarrow M_{\zeta}\otimes \Lambda_{\lambda}$ such that for any $x,y\in \mathbf{f}$, write $r(x)=\sum x_1\otimes x_2$ with homogeneous $x_1,x_2\in \mathbf{f}$, then 
\begin{equation}\label{im varphi}
\varphi_{\zeta,\lambda}(x^-\theta_{\lambda}^-y^-v_{\zeta\odot\lambda})=\sum v^{|x_2|\cdot |\theta_{\lambda}|}x_2^-y^-v_{\zeta}\otimes x_1^-\eta_{\lambda}.
\end{equation}
\end{lemma}
\begin{proof}
Consider the composition of linear maps
\begin{align*}
\varphi_0:\mathbf{f}\theta_{\lambda}\mathbf{f}\hookrightarrow \tilde{\mathbf{f}}\xrightarrow{r}\tilde{\mathbf{f}}\otimes \tilde{\mathbf{f}}\xrightarrow{p\otimes \mathrm{Id}}\bigoplus_{\nu\in \mathbb{N}[I]}\tilde{\mathbf{f}}_{\nu+|\theta_{\lambda}|}\otimes \tilde{\mathbf{f}}\hookrightarrow \tilde{\mathbf{f}}\otimes \tilde{\mathbf{f}}\xrightarrow{\tau}\tilde{\mathbf{f}}\otimes \tilde{\mathbf{f}},
\end{align*}
where $r$ is the comultiplication of $\tilde{\mathbf{f}}$, $p$ is the natural projection from $\tilde{\mathbf{f}}$ to the direct sum of its certain homogeneous components, $\tau$ is the natural isomorphism defined by $\tau(x\otimes y)=y\otimes x$. For any $x,y\in \mathbf{f}$, write $r(x)=\sum x_1\otimes x_2,r(y)=\sum y_1\otimes y_2$ with homogeneous $x_1,x_2,y_1,y_2\in \mathbf{f}$, by definitions, we have
$$\varphi_0(x\theta_{\lambda}y)=\sum v^{|x_2|\cdot|\theta_{\lambda}|+|x_2|\cdot|y_1|}x_2y_2\otimes x_1\theta_{\lambda}y_1\in \mathbf{f}\otimes \mathbf{f}\theta_{\lambda}\mathbf{f},$$
and so $\mathrm{im}\,\varphi_0\subset \mathbf{f}\otimes\mathbf{f}\theta_{\lambda}\mathbf{f}$.

Recall that $0\odot\lambda\in \tilde{X}^+$ is dominant, and we have the Verma module $M_{0\odot\lambda}$, the simple highest weight module $\Lambda_{0\odot\lambda}$ of $\tilde{\mathbf{U}}$ and the natural projection $\pi_{0\odot\lambda}:M_{0\odot\lambda}\rightarrow \Lambda_{0\odot\lambda}$. Identifying the vector spaces $\mathbf{f}\theta_{\lambda}\mathbf{f}\cong (\mathbf{f}\theta_{\lambda}\mathbf{f})_{\zeta\odot\lambda}$ and $\mathbf{f}\otimes\mathbf{f}\theta_{\lambda}\mathbf{f}\cong M_{\zeta}\otimes (\mathbf{f}\theta_{\lambda}\mathbf{f})_{0\odot\lambda}$, we obtain a linear map 
$$(\mathbf{f}\theta_{\lambda}\mathbf{f})_{\zeta\odot\lambda}\xrightarrow{\varphi_0}M_{\zeta}\otimes (\mathbf{f}\theta_{\lambda}\mathbf{f})_{0\odot\lambda}\xrightarrow{\mathrm{Id}\otimes \pi_{0\odot\lambda}} M_{\zeta}\otimes \pi_{0\odot\lambda}((\mathbf{f}\theta_{\lambda}\mathbf{f})_{0\odot\lambda})\subset M_{\zeta}\otimes \Lambda_{0\odot\lambda}$$
such that $(\mathrm{Id}\otimes \pi_{0\odot\lambda})\varphi_0(x^-\theta_{\lambda}^-y^-v_{\zeta\odot\lambda})=\sum v^{|x_2|\cdot|\theta_{\lambda}|+|x_2|\cdot|y_1|}x_2^-y_2^-v_{\zeta}\otimes x_1^-\theta_{\lambda}^-y_1^-\eta_{0\odot\lambda}$.
Since $\langle i,0\odot\lambda\rangle=0$ for any $i\in I$ and $y_1^-\eta_{0\odot\lambda}=0$ whenever $|y_1|\not=0$, we have
$$(\mathrm{Id}\otimes \pi_{0\odot\lambda})\varphi_0(x\theta_{\lambda}y)=\sum v^{|x_2|\cdot |\theta_{\lambda}|}x_2^-y^-v_{\zeta}\otimes x_1^-\theta_{\lambda}^-\eta_{0\odot\lambda}.$$
By \cite[Proposition 3.4]{Fang-Lan-2025}, the right multiplication by $\theta_{\lambda}$ induces a $\mathbf{U}$-module isomorphism $\phi_{\lambda}:\Lambda_{\lambda}\rightarrow \pi_{0\odot\lambda}((\mathbf{f}\theta_{\lambda}\mathbf{f})_{0\odot\lambda})$ such that $x^-\eta_{\lambda}\mapsto x^-\theta_{\lambda}^-\eta_{0\odot\lambda}$ for any $x\in \mathbf{f}$. Then the composition of linear maps 
$$\varphi_{\zeta,\lambda}:(\mathbf{f}\theta_{\lambda}\mathbf{f})_{\zeta\odot\lambda}\xrightarrow{(\mathrm{Id}\otimes \pi_{0\odot\lambda})\varphi_0}M_{\zeta}\otimes\pi_{0\odot\lambda}((\mathbf{f}\theta_{\lambda}\mathbf{f})_{0\odot\lambda})\xrightarrow{\mathrm{Id}\otimes \phi_{\lambda}^{-1}}M_{\zeta}\otimes \Lambda_{\lambda}$$ 
satisfies the condition.
\end{proof}


\begin{lemma}\label{isomorphism varphi}
Let $\zeta\in X$ and $\lambda\in X^+$. Then both $\varphi_{\zeta, \l}$ and $\psi_{\zeta,\l}$ are $\mathbf{U}$-module isomorphisms, and they inverse to each other.
\end{lemma}
\begin{proof}
We simply denote $\varphi=\varphi_{\zeta,\lambda}$ and $\psi=\psi_{\zeta,\lambda}$. We prove that $\varphi$ commutes with the actions of $K_{\mu}$ and $F_i$ for $\mu\in Y$ and $i\in I$. For any homogeneous $x, y\in \mathbf{f}$, write $r(x)=\sum x_1\otimes x_2$ with homogeneous $x_1,x_2\in \mathbf{f}$. Then $|x|=|x_1|+|x_2|$. By \eqref{definition of odot} and \eqref{im varphi}, we have
\begin{align*}
\varphi(K_{\mu}x^-\theta_{\lambda}^-y^-v_{\zeta\odot\lambda})=&\varphi(v^{\langle\mu,\zeta\odot\lambda-|x\theta_{\lambda}y|\rangle}x^-\theta_{\lambda}^-y^-v_{\zeta\odot\lambda})\\
=&v^{\langle\mu,\zeta+\lambda+|\theta_{\lambda}|-|x\theta_{\lambda}y|\rangle}\sum v^{|x_2|\cdot |\theta_{\lambda}|}x_2^-y^-v_{\zeta}\otimes x_1^-\eta_{\lambda}\\
=&\sum v^{|x_2|\cdot |\theta_{\lambda}|} v^{\langle\mu,\zeta-|x_2y|\rangle}x_2^-y^-v_{\zeta}\otimes v^{\langle\mu,\lambda-|x_1|\rangle}x_1^-\eta_{\lambda}\\
=&\sum v^{|x_2|\cdot |\theta_{\lambda}|} K_{\mu}x_2^-y^-v_{\zeta}\otimes K_{\mu}x_1^-\eta_{\lambda}=K_{\mu}\varphi(x^-\theta_{\lambda}^-y^-v_{\zeta\odot\lambda}).
\end{align*}
By $r(\theta_ix)=\sum \theta_i x_1\otimes x_2+\sum v^{i\cdot |x_1|}x_1\otimes \theta_ix_2$ and \eqref{im varphi}, we have
\begin{align*}
&\varphi(F_i x^-\theta_{\lambda}^-y^-v_{\zeta\odot\lambda})=\varphi((\theta_ix)^-\theta_{\lambda}^-y^-v_{\zeta\odot\lambda})\\
=&\sum v^{|x_2|\cdot |\theta_{\lambda}|}x_2^-y^-v_{\zeta}\otimes (\theta_ix_1)^-\eta_{\lambda}+\sum v^{i\cdot |x_1|} v^{|\theta_i x_2|\cdot |\theta_{\lambda}|}(\theta_ix_2)^-y^-v_{\zeta}\otimes x_1^-\eta_{\lambda}\\
=&\sum v^{|x_2|\cdot |\theta_{\lambda}|}x_2^-y^-v_{\zeta}\otimes F_ix_1^-\eta_{\lambda}+\sum v^{|x_2|\cdot |\theta_{\lambda}|}F_ix_2^-y^-v_{\zeta}\otimes v^{\langle-i,\lambda-|x_1|\rangle}x_1^-\eta_{\lambda}\\
=&\sum v^{|x_2|\cdot |\theta_{\lambda}|}x_2^-y^-v_{\zeta}\otimes F_ix_1^-\eta_{\lambda}+\sum v^{|x_2|\cdot |\theta_{\lambda}|}F_ix_2^-y^-v_{\zeta}\otimes K_{-i}x_1^-\eta_{\lambda}\\
=&F_i\varphi(x^-\theta_{\lambda}^-y^-v_{\zeta\odot\lambda}).
\end{align*}

We prove that $\psi\varphi=\mathrm{Id}:(\mathbf{f}\theta_{\lambda}\mathbf{f})_{\zeta\odot\lambda}\rightarrow (\mathbf{f}\theta_{\lambda}\mathbf{f})_{\zeta\odot\lambda}$. Notice that $(\mathbf{f}\theta_{\lambda}\mathbf{f})_{\zeta\odot\lambda}$ is spanned by $x^-\theta_{\lambda}^-y^-v_{\zeta\odot\lambda}$ for any homogeneous $x,y\in \mathbf{f}$. We prove by induction on $\mathrm{tr}\,|x|\in \mathbb{N}$ that $\psi\varphi(x^-\theta_{\lambda}^-y^-v_{\zeta\odot\lambda})=x^-\theta_{\lambda}^-y^-v_{\zeta\odot\lambda}$. If $\mathrm{tr}\,|x|=0$, then $x\in \mathbb{Q}(v)$. By \eqref{im varphi} and \eqref{im psi}, $\psi\varphi(\theta_{\lambda}^-y^-v_{\zeta\odot\lambda})=\psi(y^-v_{\zeta}\otimes \eta_{\lambda})=\theta_{\lambda}^-y^-v_{\zeta\odot\lambda}$. If $\mathrm{tr}\,|x|>0$, then it is enough to prove the statement for any $x=\theta_ix'$ with $i\in I$ and homogeneous $x'\in \mathbf{f}$. Notice that $\mathrm{tr}\,|x'|<\mathrm{tr}\,|x|$. Since both $\varphi$ and $\psi$ commute with the actions of $F_i$, by the inductive hypothesis, we have 
\begin{align*}
\psi\varphi(x^-\theta_{\lambda}^-y^-v_{\zeta\odot\lambda})=&\psi\varphi(F_i{x'}^-\theta_{\lambda}^-y^-v_{\zeta\odot\lambda})=F_i\psi\varphi({x'}^-\theta_{\lambda}^-y^-v_{\zeta\odot\lambda})\\
=&F_i{x'}^-\theta_{\lambda}^-y^-v_{\zeta\odot\lambda}=x^-\theta_{\lambda}^-y^-v_{\zeta\odot\lambda}.
\end{align*}

We prove that $\varphi\psi=\mathrm{Id}:M_{\zeta}\otimes \Lambda_{\lambda}\rightarrow M_{\zeta}\otimes \Lambda_{\lambda}$. Notice that $M_{\zeta}\otimes \Lambda_{\lambda}$ is spanned by $x^-v_{\zeta}\otimes y^-\eta_{\lambda}$ for any homogeneous $x,y\in \mathbf{f}$. We prove by induction on $\mathrm{tr}\,|y|\in \mathbb{N}$ that $\varphi\psi(x^-v_{\zeta}\otimes y^-\eta_{\lambda})=x^-v_{\zeta}\otimes y^-\eta_{\lambda}$. If $\mathrm{tr}\,|y|=0$, then $y\in \mathbb{Q}(v)$. By \eqref{im psi} and \eqref{im varphi}, $\varphi\psi(x^-v_{\zeta}\otimes \eta_{\lambda})=\varphi(\theta_{\lambda}^-x^-v_{\zeta\odot\lambda})=x^-v_{\zeta}\otimes \eta_{\lambda}$. If $\mathrm{tr}\,|y|>0$, then it is enough to prove the statement for any $y=\theta_iy'$ with $i\in I$ and homogeneous $y'\in \mathbf{f}$. Notice that $\mathrm{tr}\,|y'|<\mathrm{tr}\,|y|$. Since both $\varphi$ and $\psi$ commute with the actions of $F_i$, by 
\begin{align*}
F_i(x^-v_{\zeta}\otimes y'^-\eta_{\lambda})=&x^-v_{\zeta}\otimes F_iy'^-\eta_{\lambda}+F_ix^-v_{\zeta}\otimes K_{-i}y'^-\eta_{\lambda}\\
=&x^-v_{\zeta}\otimes y^-\eta_{\lambda}+(\theta_ix)^-v_{\zeta}\otimes v^{\langle -i,\lambda-|y'|\rangle}y'^-\eta_{\lambda}
\end{align*}
and the inductive hypothesis, we have 
\begin{align*}
\varphi\psi(x^-v_{\zeta}\otimes y^-\eta_{\lambda})=&\varphi\psi(F_i(x^-v_{\zeta}\otimes y'^-\eta_{\lambda}))-\varphi\psi((\theta_ix)^-v_{\zeta}\otimes v^{\langle -i,\lambda-|y'|\rangle}y'^-\eta_{\lambda})\\
=&F_i(x^-v_{\zeta}\otimes y'^-\eta_{\lambda})-(\theta_ix)^-v_{\zeta}\otimes v^{\langle -i,\lambda-|y'|\rangle}y'^-\eta_{\lambda}=x^-v_{\zeta}\otimes y^-\eta_{\lambda}.
\end{align*}
Hence $\varphi$ and $\psi$ are inverse to each other. Since $\psi$ is a $\mathbf{U}$-module homomorphism, both $\varphi$ and $\psi$ are $\mathbf{U}$-module isomorphisms.
\end{proof}


\section{Thickening realization of canonical basis of tensor product}\label{Thickening realization of canonical basis}

\subsection{Canonical basis of tensor product of Verma module and simple module}\label{Canonical basis of Verma otimes highest weight}

Let $\zeta\in X$ and $\lambda\in X^+$. The underlying vector space of the Verma module $M_{\zeta}$ is $\mathbf{f}$, thus the canonical basis of $\mathbf{f}$ can be carried over to $M_{\zeta}$. We denote it by $\mathbf{B}(M_\zeta)=\{b^-v_{\zeta};b\in \mathbf{B}\}$. Note that $M_{\zeta}$ is not integrable, and $(M_{\zeta},\mathbf{B}(M_\zeta))$ does not satisfy the definition of the based module in \cite[\S  27.1.2]{Lusztig-1993} or \cite[\S  2.1]{Bao-Wang-2016}. Nevertheless, a canonical basis for $M_{\zeta}\otimes \Lambda_{\lambda}$ still exists.

We use the same formula in \eqref{Phi-involution} to define the $\Psi$-involution on $M_{\zeta}\otimes \Lambda_{\lambda}$:
\begin{equation}\label{Theta and Psi}
\Psi(m_1\otimes m_2)=\sum_{\nu\in \mathbb{N}[I]}(-v)^{\mathrm{tr}\,\nu}\sum_{b\in \mathbf{B}_{\nu}}b^-\overline{m_1}\otimes b^{*+}\overline{m_2}\ \textrm{for any} \ m_1\in M_{\zeta},m_2\in \Lambda_{\lambda}.
\end{equation}
Note that there are still only finitely many non-zero terms in the summation. 

\begin{lemma}\label{canonical basis of Verma otime highest weight}
Let $\zeta\in X$ and $\lambda\in X^+$. Then for any $b_1\in \mathbf{B}(M_\zeta)$ and $b_2\in \mathbf{B}(\L_\lambda)$, there exists a unique $b_1\diamondsuit b_2\in \mathbb{Z}[v^{-1}][\mathbf{B}(M_\zeta)\otimes \mathbf{B}(\Lambda_{\lambda})]$ such that
$$\Psi(b_1\diamondsuit b_2)=b_1\diamondsuit b_2 \text{ and } b_1\diamondsuit b_2-b_1\otimes b_2 \in v^{-1}\mathbb{Z}[v^{-1}][\mathbf{B}(M_\zeta)\otimes \mathbf{B}(\Lambda_{\lambda})].$$
Moreover, the set $\mathbf{B}(M_\zeta \otimes \L_\l):=\mathbf{B}(M_\zeta) \diamondsuit \mathbf{B}(\L_\l)$ is a basis of $M_{\zeta}\otimes \Lambda_{\lambda}$. The transition matrix from $\mathbf{B}(M_\zeta)\otimes \mathbf{B}(\Lambda_{\lambda})$ to $\mathbf{B}(M_\zeta \otimes \L_\l)$ is unitriangular with off-diagonal entries in $v^{-1}\mathbb{Z}[v^{-1}]$.
\end{lemma}
\begin{proof} 
By the same proof as \cite[Proposition 2.4]{Bao-Wang-2016}, both $\Theta$ and $\Psi$ preserve the $\mathcal{A}$-submodule of $M_{\zeta}\otimes \Lambda_{\lambda}$ generated by $\mathbf{B}(M_\zeta)\otimes \mathbf{B}(\Lambda_{\lambda})$. Then the statement follows from the same proof of \cite[Theorem 27.3.2]{Lusztig-1993}. 
\end{proof}

The basis $\mathbf{B}(M_\zeta \otimes \L_\l)$ is the {\it canonical basis} of the tensor product $M_{\zeta}\otimes \Lambda_{\lambda}$.

\begin{lemma}\label{a otimes Id is a based map}
Let $\lambda_1,\lambda_2\in X^+,w\in W$ and $a: M_{-w\l_1} \to {^{\omega}V_w(\lambda_1)}$ be the projection map. Then $(a\otimes \mathrm{Id})(\mathbf{B}(M_{-w\lambda_1}\otimes \Lambda_{\lambda_2})) \subset \mathbf{B}({^{\omega}V_{w}(\lambda_1)}\otimes \Lambda_{\lambda_2})\sqcup \{0\}$, and the set $\mathbf{B}(M_{-w\lambda_1}\otimes \Lambda_{\lambda_2})\cap \mathrm{ker}\,(a\otimes \mathrm{Id})$ is a basis of $\mathrm{ker}\,(a\otimes \mathrm{Id})$.
\end{lemma}
\begin{proof}
By Proposition \ref{action map is a based map}, $a(\mathbf{B}(M_{-w\lambda_1}))\subset \mathbf{B}({^{\omega}V_w(\lambda_1)})\sqcup\{0\}$. Let $b_1 \in \mathbf{B}(M_{-w\lambda_1})$ and $b_2 \in \mathbf{B}(\L_{\l_2})$. We prove that
$$(a\otimes \mathrm{Id})(b_1\diamondsuit b_2)=\begin{cases}a(b_1) \diamondsuit b_2,\ &\textrm{if}\ a(b_1)\not=0;\\
0,\ &\textrm{if}\ a(b_1)=0.
\end{cases}$$
Let $\Psi$ and $\Psi'$ be the $\Psi$-involutions on ${^{\omega}\Lambda_{\lambda_1}}\otimes \Lambda_{\lambda_2}$ and $M_{-w\lambda_1}\otimes \Lambda_{\lambda_2}$ respectively. By definitions, we have $\Psi(a\otimes \mathrm{Id})=(a\otimes \mathrm{Id})\Psi':M_{-w \l_1} \otimes \L_{\l_2}\rightarrow {^{\omega}\Lambda_{\lambda_1}}\otimes \Lambda_{\lambda_2}$. By Lemma \ref{canonical basis of Verma otime highest weight}, we have $\Psi'(b_1\diamondsuit b_2)=b_1\diamondsuit b_2\in \mathbb{Z}[v^{-1}][\mathbf{B}(M_{-w\lambda_1})\otimes \mathbf{B}(\Lambda_{\lambda_2})]$ and $b_1\diamondsuit b_2-b_1\otimes b_2\in v^{-1}\mathbb{Z}[v^{-1}][\mathbf{B}(M_{-w\lambda_1})\otimes \mathbf{B}(\Lambda_{\lambda_2})]$.
By Proposition \ref{action map is a based map}, 
\begin{align*}
&\Psi(a\otimes \mathrm{Id})(b_1\diamondsuit b_2)
=(a\otimes \mathrm{Id})(b_1\diamondsuit b_2)\in \mathbb{Z}[v^{-1}][\mathbf{B}({^{\omega}\Lambda_{\lambda_1}})\otimes \mathbf{B}(\Lambda_{\lambda_2})],\\
&(a\otimes \mathrm{Id})(b_1\diamondsuit b_2)-a(b_1)\otimes b_2\in v^{-1}\mathbb{Z}[v^{-1}][\mathbf{B}({^{\omega}\Lambda_{\lambda_1}})\otimes \mathbf{B}(\Lambda_{\lambda_2})].
\end{align*}

If $a(b_1) \neq 0$, then $a(b_1) \in \mathbf{B}({^{\omega}\Lambda_{\lambda_1}})$, and $(a\otimes \mathrm{Id})(b_1\diamondsuit b_2^-\eta_{\lambda_2})$ satisfies the characterizing conditions in Theorem \ref{24.3.3}. Hence $(a\otimes \mathrm{Id})(b_1\diamondsuit b_2)=a(b_1)\diamondsuit b_2$. 

If $a(b_1)=0$, then $(a\otimes \mathrm{Id})(b_1\diamondsuit b_2)\in v^{-1}\mathbb{Z}[v^{-1}][\mathbf{B}({^{\omega}\Lambda_{\lambda_1}})\otimes \mathbf{B}(\Lambda_{\lambda_2})]$ is $\Psi$-invariant. By Theorem \ref{24.3.3}, $\mathbf{B}({^{\omega}\Lambda_{\lambda_1}})\otimes \mathbf{B}(\Lambda_{\lambda_2})\subset \mathbb{Z}[v^{-1}][\mathbf{B}({^{\omega}\Lambda_{\lambda_1}}\otimes\Lambda_{\lambda_2})]$, and $0$ is the unique $\Psi$-invariant element in $v^{-1}\mathbb{Z}[v^{-1}][\mathbf{B}({^{\omega}\Lambda_{\lambda_1}}\otimes\Lambda_{\lambda_2})]$. Hence $(a\otimes \mathrm{Id})(b_1\diamondsuit b_2)=0$.

Then the statement follows from Proposition \ref{action map is a based map} and Proposition \ref{canonical basis of demazure otimes highest weight}. 
\end{proof}

\subsection{Canonical basis of $(\mathbf{f}\theta_{\lambda}\mathbf{f})_{\zeta\odot\lambda}$}

Let $\tilde{\mathbf{B}}$ be the canonical basis of $\tilde{\mathbf{f}}$.

Let $\lambda\in X^+$. Recall that $\mathbf{f}\theta_{\lambda}\mathbf{f}$ is a subspace of $\tilde{\mathbf{f}}$ defined in \S \ref{ftheta_lambdaf}. We define
$$\mathbf{B}(\mathbf{f}\theta_{\lambda}\mathbf{f}):=\tilde{\mathbf{B}}\cap (\mathbf{f}\theta_{\lambda}\mathbf{f})\subset \tilde{\mathbf{B}}.$$
By \cite[Proposition 4.3]{Fang-Lan-2025}, the set $\mathbf{B}(\mathbf{f}\theta_{\lambda}\mathbf{f})$ is a basis of $\mathbf{f}\theta_{\lambda}\mathbf{f}$. We call it the {\it canonical basis} of $\mathbf{f}\theta_{\lambda}\mathbf{f}$. 

Let $\mathbf{B}(\lambda)=\bigcap_{i\in I}(\mathbf{B}\setminus \mathbf{f}\theta_i^{\langle i,\lambda\rangle+1})$. By \cite[Theorem 14.4.11]{Lusztig-1993}, we have
\begin{equation}\label{14.4.11}
\mathbf{B}(\lambda)=\{b\in \mathbf{B};b^-\eta_{\lambda}\neq0\}. 
\end{equation}
By \cite[\S 5.1, (14)]{Fang-Lan-2025}, the right multiplication by $\theta_{\lambda}$ gives a bijection
\begin{equation}\label{Fang-Lan-2025-5.3*}
\mathbf{B}(\lambda)\rightarrow \bigcap_{i\in I}(\tilde{\mathbf{B}}\setminus \tilde{\mathbf{f}}\theta_i)\cap \mathbf{B}(\mathbf{f}\theta_{\lambda}\mathbf{f}), b\mapsto b\theta_{\lambda}.
\end{equation}

Let $\zeta\in X$ and $\lambda\in X^+$. We define 
$$\mathbf{B}((\mathbf{f}\theta_{\lambda}\mathbf{f})_{\zeta\odot\lambda}):=\{b^-v_{\zeta\odot\lambda};b\in \mathbf{B}(\mathbf{f}\theta_{\lambda}\mathbf{f})\}.$$
Then it is a basis of $(\mathbf{f}\theta_{\lambda}\mathbf{f})_{\zeta\odot\lambda}$, and we call it the {\it canonical basis} of $(\mathbf{f}\theta_{\lambda}\mathbf{f})_{\zeta\odot\lambda}$. 

\subsection{Involution and bilinear form}\label{sec:technical}

Let $\zeta\in X$ and $\lambda\in X^+$. By Lemma \ref{isomorphism varphi}, there is a $\mathbf{U}$-module isomorphism $\varphi_{\zeta,\lambda}:(\mathbf{f}\theta_{\lambda}\mathbf{f})_{\zeta\odot\lambda}\rightarrow M_{\zeta}\otimes \Lambda_{\lambda}$. In this subsection, we study the relations between $\varphi_{\zeta,\lambda}$ and the involutions, the symmetric bilinear forms on two sides. 

The underlying vector space of the Verma module $M_{\zeta}$ is $\mathbf{f}$, thus the bar-involution on $\mathbf{f}$ can be carried over to $M_{\zeta}$, that is, $\overline{x^-v_{\zeta}}=(\overline{x})^-v_{\zeta}$ for any $x\in \mathbf{f}$. We identify $M_{\zeta}$ with $\mathbf{U}/(\sum_{i\in I}\mathbf{U}E_i+\sum_{\mu\in Y}\mathbf{U}(K_{\mu}-v^{\langle \mu,\zeta\rangle}))$. Under this identification, the bar-involution of $M_{\zeta}$ is induced by the bar-involution of $\mathbf{U}$, and so $\overline{um}=\bar{u}\bar{m}$ for any $u\in \mathbf{U}$ and $m\in M_{\zeta}$. Similarly, the bar-involution of $M_{\zeta\odot\lambda}$ satisfies the same property. It restricts to a bar-involution on $(\mathbf{f}\theta_{\lambda}\mathbf{f})_{\zeta\odot\lambda}$. Recall that the tensor product $M_{\zeta}\otimes \Lambda_{\lambda}$ has a $\Psi$-involution defined by \eqref{Theta and Psi}.

\begin{lemma}\label{varphi and involution}
Let $\zeta\in X,\lambda\in X^+$ and $m\in (\mathbf{f}\theta_{\lambda}\mathbf{f})_{\zeta\odot\lambda}$. Then $\Psi\varphi_{\zeta,\lambda}(m)=\varphi_{\zeta,\lambda}(\overline{m})$.
\end{lemma}
\begin{proof}
We simply denote $\varphi_{\zeta,\lambda}=\varphi$. Note that both $\Psi\varphi$ and $\varphi(\bar{\cdot})$ are $\mathbb{Q}$-linear exchanging $v$ and $v^{-1}$.

Notice that $(\mathbf{f}\theta_{\lambda}\mathbf{f})_{\zeta\odot\lambda}$ is spanned by $x^-\theta_{\lambda}^-y^-v_{\zeta\odot\lambda}$ for any homogeneous $x,y\in \mathbf{f}$. We prove by induction on $\mathrm{tr}\,|x|\in \mathbb{N}$ that $\Psi\varphi(x^-\theta_{\lambda}^-y^-v_{\zeta\odot\lambda})=\varphi(\overline{x^-\theta_{\lambda}^-y^-v_{\zeta\odot\lambda}})$. If $\mathrm{tr}\,|x|=0$, then $x\in \mathbb{Q}(v)$. By \eqref{im varphi} and \eqref{Theta and Psi}, we have 
\begin{align*}
\Psi\varphi(\theta_{\lambda}^-y^-v_{\zeta\odot\lambda})=&\Psi(y^-v_{\zeta}\otimes \eta_{\lambda})=\Theta(\overline{y^-v_{\zeta}}\otimes \overline{\eta_{\lambda}})\\
=&(\overline{y})^-v_{\zeta}\otimes \eta_{\lambda}=\varphi(\theta_{\lambda}^-(\overline{y})^-v_{\zeta\odot\lambda})=\varphi(\overline{\theta_{\lambda}^-y^-v_{\zeta\odot\lambda}})   . 
\end{align*}
If $\mathrm{tr}\,|x|>0$, then it suffices to prove the statement for $x=\theta_ix'$ with $i\in I$ and homogeneous $x'\in \mathbf{f}$. Note that $\mathrm{tr}\,|x'|<\mathrm{tr}\,|x|$. By \cite[Lemma 24.1.2]{Lusztig-1993}, $\Psi(um)=\overline{u}\Psi(m)$ for any $u\in \mathbf{U}$ and $m\in M_{\zeta}\otimes \Lambda_{\lambda}$. By Lemma \ref{isomorphism varphi}, the inductive hypothesis, we have
\begin{align*}
\Psi\varphi(x^-\theta_{\lambda}^-y^-v_{\zeta\odot\lambda})=&\Psi\varphi(F_i{x'}^-\theta_{\lambda}^-y^-v_{\zeta\odot\lambda})=\Psi(F_i\varphi({x'}^-\theta_{\lambda}^-y^-v_{\zeta\odot\lambda}))\\
=&\overline{F_i}\Psi\varphi({x'}^-\theta_{\lambda}^-y^-v_{\zeta\odot\lambda})=F_i\varphi(\overline{{x'}^-\theta_{\lambda}^-y^-v_{\zeta\odot\lambda}})\\
=&\varphi(F_i\overline{{x'}^-\theta_{\lambda}^-y^-v_{\zeta\odot\lambda}})=\varphi(\overline{F_i{x'}^-\theta_{\lambda}^-y^-v_{\zeta\odot\lambda}})=\varphi(\overline{x^-\theta_{\lambda}^-y^-v_{\zeta\odot\lambda}}),
\end{align*}
as desired.
\end{proof}

We define a symmetric bilinear form 
$(\,,\,)_{\otimes}:(M_{\zeta}\otimes \Lambda_{\lambda})\times (M_{\zeta}\otimes \Lambda_{\lambda})\rightarrow \mathbb{Q}(v)$
by
\begin{equation}\label{definition the bilinear form on tensor product}
(x_1^-v_{\zeta}\otimes m_2,{x'_1}^-v_{\zeta}\otimes m'_2)_{\otimes}=(x_1,x'_1)_{\mathbf{f}}(m_2,m'_2)_{\Lambda_{\lambda}}\textrm{for}\ x_1,x'_1\in \mathbf{f},m_2,m'_2\in \Lambda_{\lambda}.
\end{equation}
For any $i\in I$, let ${_{i}r}:\tilde{\mathbf{f}}\rightarrow \tilde{\mathbf{f}}$ be the linear map defined in \cite[\S 1.2.13]{Lusztig-1993}. For any homogeneous $x,y\in \mathbf{f}$, we have 
\begin{equation}\label{irxthetalambday}
{_{i}r}(x\theta_{\lambda}y)={_{i}r}(x)\theta_{\lambda}y+v^{\langle i,|x|-\lambda\rangle}x\theta_{\lambda}\,{_{i}r}(y)\in \mathbf{f}\theta_{\lambda}\mathbf{f}.
\end{equation}
Hence ${_{i}r}:\tilde{\mathbf{f}}\rightarrow \tilde{\mathbf{f}}$ can restrict to ${_{i}r}:\mathbf{f}\theta_{\lambda}\mathbf{f}\rightarrow \mathbf{f}\theta_{\lambda}\mathbf{f}$. Identifying the $\mathbb{Q}(v)$ vector spaces $\mathbf{f}\theta_{\lambda}\mathbf{f}\cong (\mathbf{f}\theta_{\lambda}\mathbf{f})_{\zeta\odot\lambda}$, we have a linear map ${_{i}r}:(\mathbf{f}\theta_{\lambda}\mathbf{f})_{\zeta\odot\lambda}\rightarrow (\mathbf{f}\theta_{\lambda}\mathbf{f})_{\zeta\odot\lambda}$ such that ${_ir}(x^-v_{\zeta\odot\lambda})=({_ir(x)})^-v_{\zeta\odot\lambda}$ for any $x\in \mathbf{f}\theta_{\lambda}\mathbf{f}$. Let $\epsilon_i:M_{\zeta}\otimes \Lambda_{\lambda}\rightarrow M_{\zeta}\otimes \Lambda_{\lambda}$ be the linear map defined in \cite[\S 18.1.3]{Lusztig-1993}, that is, for any $x_1\in \mathbf{f}$ and $m_2\in \Lambda_{\lambda}$,
\begin{equation}\label{definition of epsiloni}
\epsilon_i(x_1^-v_{\zeta}\otimes m_2)=({_{i}r}(x_1))^-v_{\zeta}\otimes K_{-i}m_2+(v-v^{-1})x_1^-v_{\zeta}\otimes K_{-i}E_i m_2.
\end{equation}

\begin{lemma}\label{varphi and inner product}
Let $\zeta\in X,\lambda\in X^+$ and $i\in I$. Then
\begin{enumerate}
\item $\epsilon_i\varphi_{\zeta,\lambda}=\varphi_{\zeta,\lambda}\,{_{i}r}:(\mathbf{f}\theta_{\lambda}\mathbf{f})_{\zeta\odot\lambda}\rightarrow M_{\zeta}\otimes \Lambda_{\lambda}$;
\item $(F_i(x_1^-v_{\zeta}\otimes m_2),{x'_1}^-v_{\zeta}\otimes m'_2)_{\otimes}=(1-v^{-2})^{-1}(x_1^-v_{\zeta}\otimes m_2,\epsilon_i({x'_1}^-v_{\zeta}\otimes m'_2))_{\otimes}$ for any $x_1,x'_1\in \mathbf{f}$ and $m_2,m'_2\in \Lambda_{\lambda}$;
\item $(\varphi_{\zeta,\lambda}(z^-v_{\zeta\odot\lambda}),\varphi_{\zeta,\lambda}({z'}^-v_{\zeta\odot\lambda}))_{\otimes}=(\theta_{\lambda},\theta_{\lambda})_{\tilde{\mathbf{f}}}^{-1}(z,z')_{\tilde{\mathbf{f}}}$ for any $z,z'\in \mathbf{f}\theta_{\lambda}\mathbf{f}$.
\end{enumerate}
\end{lemma}
\begin{proof}
We simply denote $\varphi_{\zeta,\lambda}=\varphi$.

(1) Notice that $(\mathbf{f}\theta_{\lambda}\mathbf{f})_{\zeta\odot\lambda}$ is spanned by $x^-\theta_{\lambda}^-y^-v_{\zeta\odot\lambda}$ for any homogeneous $x,y\in \mathbf{f}$. We prove by induction on $\mathrm{tr}\,|x|\in \mathbb{N}$ that $\epsilon_i\varphi(x^-\theta_{\lambda}^-y^-v_{\zeta\odot\lambda})=\varphi\,{_{i}r}(x^-\theta_{\lambda}^-y^-v_{\zeta\odot\lambda})$. If $\mathrm{tr}\,|x|=0$, then $x\in \mathbb{Q}(v)$. By definitions, \eqref{im varphi}, \eqref{definition of epsiloni} and \eqref{irxthetalambday}, we have 
\begin{align*}
\epsilon_i\varphi(\theta_{\lambda}^-y^-v_{\zeta\odot\lambda})\!=&\epsilon_i(y^-v_{\zeta}\otimes \eta_{\lambda})=({_{i}r}(y))^-v_{\zeta}\otimes K_{-i}\eta_{\lambda}+(v-v^{-1})y^-v_{\zeta}\otimes K_{-i}E_i\eta_{\lambda}\\
=&v^{-\langle i,\lambda\rangle}({_{i}r}(y))^-v_{\zeta}\otimes \eta_{\lambda}\!=\!\varphi(v^{-\langle i,\lambda\rangle}\theta_{\lambda}^-({_{i}r}(y))^-v_{\zeta\odot\lambda})\!=\!\varphi\,{_{i}r}(\theta_{\lambda}^-y^-v_{\zeta\odot\lambda}).
\end{align*}
If $\mathrm{tr}\,|x|>0$, then it suffices to prove the statement for $x=\theta_jx'$ with $i\in I$ and homogeneous $x'\in \mathbf{f}$. Note that $\mathrm{tr}\,|x'|<\mathrm{tr}\,|x|$. By Lemma \ref{isomorphism varphi}, \cite[\S 18.1.3 \& \S 15.1.2 \& Lemma 15.1.4]{Lusztig-1993} and the inductive hypothesis, we have 
\begin{align*}
&\epsilon_i\varphi(x^-\theta_{\lambda}^-y^-v_{\zeta\odot\lambda})=\epsilon_i\varphi(F_j{x'}^-\theta_{\lambda}^-y^-v_{\zeta\odot\lambda})=\epsilon_i(F_j\varphi({x'}^-\theta_{\lambda}^-y^-v_{\zeta\odot\lambda}))\\
=&v^{i\cdot j}F_j\epsilon_i\varphi({x'}^-\theta_{\lambda}^-y^-v_{\zeta\odot\lambda})+\delta_{i,j}\varphi({x'}^-\theta_{\lambda}^-y^-v_{\zeta\odot\lambda})\\
=&v^{i\cdot j}F_j\varphi\,{_{i}r}({x'}^-\theta_{\lambda}^-y^-v_{\zeta\odot\lambda})+\delta_{i,j}\varphi({x'}^-\theta_{\lambda}^-y^-v_{\zeta\odot\lambda})\\
=&v^{i\cdot j}\varphi(F_j\,{_{i}r}({x'}^-\theta_{\lambda}^-y^-v_{\zeta\odot\lambda}))+\delta_{i,j}\varphi({x'}^-\theta_{\lambda}^-y^-v_{\zeta\odot\lambda})\\
=&\varphi(v^{i\cdot j}\theta_j\,{_{i}r}({x'}^-\theta_{\lambda}^-y^-v_{\zeta\odot\lambda})+\delta_{i,j}{x'}^-\theta_{\lambda}^-y^-v_{\zeta\odot\lambda})=\varphi\,{_{i}r}({x'}^-\theta_{\lambda}^-y^-v_{\zeta\odot\lambda}),
\end{align*}
as desired.

(2) By definitions and \cite[Proposition 19.1.2 \& \S 1.2.13]{Lusztig-1993}, we have
\begin{align*}
&(F_i(x_1^-v_{\zeta}\otimes m_2),{x'_1}^-v_{\zeta}\otimes m'_2)_{\otimes}=\!(x_1^-v_{\zeta}\otimes F_im_2+F_ix_1^-v_{\zeta}\otimes K_{-i}m_2,{x'_1}^v_{\zeta}\otimes m'_2)_{\otimes}\\
=&(x_1,x'_1)_{\mathbf{f}}(F_im_2,m'_2)_{\Lambda_{\lambda}}+(\theta_ix_1,x'_1)_{\mathbf{f}}(K_{-i}m_2,m'_2)_{\Lambda_{\lambda}}\\
=&(x_1,x'_1)_{\mathbf{f}}(m_2,vK_{-i}E_im'_2)_{\Lambda_{\lambda}}+(1-v^{-2})^{-1}(x_1,{_{i}}r(x'_1))_{\mathbf{f}}(m_2,K_{-i}m'_2)_{\Lambda_{\lambda}}\\
=&(1-v^{-2})^{-1}(x_1^-v_{\zeta}\otimes m_2,(v-v^{-1}){x'_1}^-v_{\zeta}\otimes K_{-i}E_im'_2+({_{i}}r(x'_1))^-v_{\zeta}\otimes K_{-i}m'_2)_{\otimes}\\
=&(1-v^{-2})^{-1}(x_1^-v_{\zeta}\otimes m_2,\epsilon_i({x'_1}^-v_{\zeta}\otimes m'_2))_{\otimes}.
\end{align*}

(3) Notice that $\mathbf{f}\theta_{\lambda}\mathbf{f}$ is spanned by $x\theta_{\lambda}y$ for any homogeneous $x,y\in \mathbf{f}$. We prove by induction on $\mathrm{tr}\,|x|\in \mathbb{N}$ the statement for any $z=x\theta_{\lambda}y$ and $z'=x'\theta_{\lambda}y'$. If $\mathrm{tr}\,|x|=0$, then $x\in \mathbb{Q}(v)$. On the one hand, write $r(x')=\sum x'_1\otimes x'_2$ with homogeneous $x'_1,x'_2\in \mathbf{f}$, by \eqref{im varphi} and \cite[Proposition 19.1.2]{Lusztig-1993}, we have
\begin{align*}
&(\varphi(\theta_{\lambda}^-y^-v_{\zeta\odot\lambda}),\varphi({x'}^-\theta_{\lambda}^-{y'}^-v_{\zeta\odot\lambda}))_{\otimes}=(y^-v_{\zeta}\otimes \eta_{\lambda},\sum v^{|x'_2|\cdot |\theta_{\lambda}|}{x'_2}^-{y'}^-v_{\zeta}\otimes {x'_1}^-\eta_{\lambda})_{\otimes}\\
=&\sum v^{|x'_2|\cdot |\theta_{\lambda}|}(y,x'_2y')_{\mathbf{f}}(\eta_{\lambda},{x'_1}^-\eta_{\lambda})_{\Lambda_{\lambda}}=v^{|x'|\cdot |\theta_{\lambda}|}(y,x'y')_{\mathbf{f}}.
\end{align*}
On the other hand, by \cite[Proposition 1.2.3]{Lusztig-1993} and the fact that $(\tilde{\mathbf{f}}_{\nu},\tilde{\mathbf{f}}_{\nu'})_{\tilde{\mathbf{f}}}=0$ for any $\nu\not=\nu'$ in $\mathbb{N}[\tilde{I}]$ (see the proof of \cite[Proposition 1.2.3]{Lusztig-1993}), we have 
\begin{align*}
(\theta_{\lambda}y,x'\theta_{\lambda}y')_{\tilde{\mathbf{f}}}=&(\theta_{\lambda}\otimes y,r(x'\theta_{\lambda}y'))_{\tilde{\mathbf{f}}\otimes\tilde{\mathbf{f}}}=(\theta_{\lambda}\otimes y,v^{|x'|\cdot|\theta_{\lambda}|}\theta_{\lambda}\otimes x'y')_{\tilde{\mathbf{f}}\otimes\tilde{\mathbf{f}}}\\
=&v^{|x'|\cdot|\theta_{\lambda}|}(\theta_{\lambda},\theta_{\lambda})_{\tilde{\mathbf{f}}}(y,x'y')_{\tilde{\mathbf{f}}}.
\end{align*}
It is clear that $(y,x'y')_{\mathbf{f}}=(y,x'y')_{\tilde{\mathbf{f}}}$, and so $$(\varphi(\theta_{\lambda}^-y^-v_{\zeta\odot\lambda}),\varphi({x'}^-\theta_{\lambda}^-{y'}^-v_{\zeta\odot\lambda}))_{\otimes}=(\theta_{\lambda},\theta_{\lambda})_{\tilde{\mathbf{f}}}^{-1}(\theta_{\lambda}y,x'\theta_{\lambda}y')_{\tilde{\mathbf{f}}}.$$
If $\mathrm{tr}\,|x|>0$, then it suffices to prove the statement for any $x=\theta_ix''$ with $i\in I$ homogeneous $x''\in \mathbf{f}$. Note that $\mathrm{tr}\,|x''|<\mathrm{tr}\,|x|$. By Lemma \ref{isomorphism varphi}, (2), (1), the inductive hypothesis and \cite[\S  1.2.13]{Lusztig-1993}, we have 
\begin{align*}
&(\varphi(x^-\theta_{\lambda}^-y^-v_{\zeta\odot\lambda}),\varphi({x'}^-\theta_{\lambda}^-{y'}^-v_{\zeta\odot\lambda}))_{\otimes}=(\varphi(F_i{x''}^-\theta_{\lambda}^-y^-v_{\zeta\odot\lambda}),\varphi({x'}^-\theta_{\lambda}^-{y'}^-v_{\zeta\odot\lambda}))_{\otimes}\\
=&(F_i\varphi({x''}^-\theta_{\lambda}^-y^-v_{\zeta\odot\lambda}),\varphi({x'}^-\theta_{\lambda}^-{y'}^-v_{\zeta\odot\lambda}))_{\otimes}\\
=&(1-v^{-2})^{-1}(\varphi({x''}^-\theta_{\lambda}^-y^-v_{\zeta\odot\lambda}),\epsilon_i\varphi({x'}^-\theta_{\lambda}^-{y'}^-v_{\zeta\odot\lambda}))_{\otimes}\\
=&(1-v^{-2})^{-1}(\varphi({x''}^-\theta_{\lambda}^-y^-v_{\zeta\odot\lambda}),\varphi\,{_{i}r}({x'}^-\theta_{\lambda}^-{y'}^-v_{\zeta\odot\lambda}))_{\otimes}\\
=&(1-v^{-2})^{-1}(\theta_{\lambda},\theta_{\lambda})_{\tilde{\mathbf{f}}}^{-1}(x''\theta_{\lambda}y,\,{_{i}r}(x'\theta_{\lambda}y'))_{\tilde{\mathbf{f}}}
=(\theta_{\lambda},\theta_{\lambda})_{\tilde{\mathbf{f}}}^{-1}(\theta_ix''\theta_{\lambda}y,x'\theta_{\lambda}y')_{\tilde{\mathbf{f}}}\\
=&(\theta_{\lambda},\theta_{\lambda})_{\tilde{\mathbf{f}}}^{-1}(x\theta_{\lambda}y,x'\theta_{\lambda}y')_{\tilde{\mathbf{f}}},
\end{align*}
as desired.
\end{proof}

\subsection{Realization of canonical basis of tensor product}

We prove that the thickening realization of the tensor product established in \S \ref{ftheta_lambdaf} is compatible with the canonical basis.

\begin{theorem}\label{positivity of transition matrix}
Let $\zeta\in X,\lambda\in X^+$. Then 
\begin{enumerate}
\item there are bijections
$$\xymatrix@C=1.5cm{\mathbf{B}((\mathbf{f}\theta_{\lambda}\mathbf{f})_{\zeta \odot\lambda}) \ar@<.5ex>[r]^{\varphi_{\zeta,\lambda}} &\mathbf{B}(M_{\zeta}\otimes \Lambda_{\lambda}); \ar@<.5ex>[l]^{\psi_{\zeta,\lambda}}}$$
\item let $\tilde{b}\in \mathbf{B}(\mathbf{f}\theta_{\lambda}\mathbf{f})$ with $r(\tilde{b})=\sum_{\tilde{b}_1,\tilde{b}_2\in \tilde{\mathbf{B}}}c_{\tilde{b}_1,\tilde{b}_2}^{\tilde{b}}\tilde{b}_1\otimes\tilde{b}_2$, then
$$\varphi_{\zeta,\lambda}(\tilde{b}^-v_{\zeta\odot\lambda})=\sum_{b_1\in \mathbf{B},b_2^-\eta_{\lambda}\in \mathbf{B}(\Lambda_{\lambda})}c_{b_2\theta_{\lambda_2},b_1}^{\tilde{b}}b_1^-v_{\zeta}\otimes b_2^-\eta_{\lambda}.$$
\end{enumerate}
\end{theorem}

\begin{remark}\label{structure constant remark}
A similar result to part (1) that $\psi_{\zeta,\lambda}$ is a bijection was established by Li in \cite[Theorem 7.17]{Li-2014} using perverse sheaves. Our approach is different and leads to the comparison of the entries of the transition matrix with the structure constants of the comultiplication in part (2). 
\end{remark}

\begin{proof}
We simply denote $\varphi_{\zeta,\lambda}=\varphi$. For any $\tilde{b}\in \mathbf{B}(\mathbf{f}\theta_{\lambda}\mathbf{f})\subset \tilde{\mathbf{B}}$, we write $r(\tilde{b})=\sum_{\tilde{b}_1,\tilde{b}_2\in \tilde{\mathbf{B}}}c_{\tilde{b}_1,\tilde{b}_2}^{\tilde{b}} \tilde{b}_1\otimes \tilde{b}_2$. Then $c_{\tilde{b}_1,\tilde{b}_2}^{\tilde{b}}\in \mathbb{N}[v,v^{-1}]$ by the positivity property of $\tilde{\mathbf{B}}$ with respect to the comultiplication in Theorem \ref{14.4.13}. By the definition of $\varphi$, we have $$\varphi(\tilde{b}^-v_{\zeta\odot\lambda})=\sum c_{\tilde{b}_1,\tilde{b}_2}^{\tilde{b}}\tilde{b}_2^-v_{\zeta}\otimes \phi_{\lambda}^{-1}(\tilde{b}_1^-\eta_{0\odot\lambda}),$$
where the summation is taken over $\tilde{b}_1\in \mathbf{B}(\mathbf{f}\theta_{\lambda}\mathbf{f})$ and $\tilde{b}_2\in \mathbf{B}$ with $\tilde{b}_1^-\eta_{0\odot\lambda}\not=0$. Since $\tilde{b}_1^-\eta_{0\odot\lambda}\not=0$, by \eqref{14.4.11}, we have $\tilde{b}_1\in \tilde{\mathbf{B}}(0\odot\lambda)\subset \bigcap_{i\in I}(\tilde{\mathbf{B}}\setminus \tilde{\mathbf{f}}\theta_i)$. Then by \eqref{Fang-Lan-2025-5.3*}, there exists a unique $b_1\in\mathbf{B}(\lambda)$ such that $\tilde{b}_1=b_1\theta_{\lambda}$. Hence
\begin{align}\label{varphib integral}
\varphi(\tilde{b}^-v_{\zeta\odot\lambda})=\sum c_{b_1\theta_{\lambda},b_2}^{\tilde{b}}b_2^-v_{\zeta}\otimes b_1^-\eta_{\lambda}\in \mathbb{N}[v,v^{-1}][\mathbf{B}(M_{\zeta})\otimes \mathbf{B}(\Lambda_{\lambda})].   
\end{align}

Let $\mathbf{A}=\mathbb{Q}[[v^{-1}]]\cap \mathbb{Q}(v)$ and $L(M_{\zeta}\otimes \Lambda_{\lambda})$ be the $\mathbf{A}$-submodule of $M_{\zeta}\otimes \Lambda_{\lambda}$ generated by $\mathbf{B}(M_{\zeta})\otimes \mathbf{B}(\Lambda_{\lambda})$. By \cite[Theorem 14.2.3 \& Proposition 19.3.3]{Lusztig-1993}, the basis $\mathbf{B}(M_{\zeta})\otimes \mathbf{B}(\Lambda_{\lambda})$ of $M_{\zeta}\otimes \Lambda_{\lambda}$ is almost orthonormal (in the sense of \cite[\S  14.2.1]{Lusztig-1993}) with respect to the bilinear form $(\,,\,)_{\otimes}$ defined in \eqref{definition the bilinear form on tensor product}. By \cite[Theorem 14.2.3]{Lusztig-1993} and Lemma \ref{varphi and inner product}, we have $(\varphi(\tilde{b}^-v_{\zeta\odot\lambda}),\varphi(\tilde{b}^-v_{\zeta\odot\lambda}))_{\otimes}=(\theta_{\lambda},\theta_{\lambda})_{\tilde{\mathbf{f}}}^{-1}(\tilde{b},\tilde{b})_{\tilde{\mathbf{f}}}\in 1+v^{-1}\mathbf{A}$. By \cite[Lemma 14.2.2]{Lusztig-1993}, there exists ${b'_1}^-v_{\zeta}\otimes {b'_2}^-\eta_{\lambda}\in \mathbf{B}(M_{\zeta})\otimes \mathbf{B}(\Lambda_{\lambda})$ such that 
$$\varphi(\tilde{b}^-v_{\zeta\odot\lambda})=\pm {b'_1}^-v_{\zeta}\otimes {b'_2}^-\eta_{\lambda}\ \mathrm{mod}\,v^{-1}L(M_{\zeta}\otimes \Lambda_{\lambda}).$$
By \eqref{varphib integral} and $\mathbb{N}[v,v^{-1}]\cap\mathbf{A}=\mathbb{N}[v^{-1}]$, we have 
\begin{align*}
\varphi(\tilde{b}^-v_{\zeta\odot\lambda})={b'_1}^-v_{\zeta}\otimes {b'_2}^-\eta_{\lambda}\ \mathrm{mod}\,v^{-1}\mathbb{N}[v^{-1}][\mathbf{B}(M_{\zeta})\otimes \mathbf{B}(\Lambda_{\lambda})].
\end{align*}
By \cite[Theorem 14.2.3]{Lusztig-1993} and Lemma \ref{varphi and involution}, we have $$\overline{\tilde{b}^-v_{\zeta\odot\lambda}}=(\overline{\tilde{b}})^-v_{\zeta\odot\lambda}=\tilde{b}^-v_{\zeta\odot\lambda},\ \Psi\varphi(\tilde{b}^-v_{\zeta\odot\lambda})=\varphi(\overline{\tilde{b}^-v_{\zeta\odot\lambda}})=\varphi(\tilde{b}^-v_{\zeta\odot\lambda}).$$ Hence $\varphi(\tilde{b}^-v_{\zeta\odot\lambda})$ satisfies the characterizing conditions in Lemma \ref{canonical basis of Verma otime highest weight}, and so $\varphi(\tilde{b}^-v_{\zeta\odot\lambda})={b'_1}^-v_{\zeta}\diamondsuit {b'_2}^-\eta_{\lambda}\in \mathbf{B}(M_{\zeta}\otimes \Lambda_{\lambda})$. Hence $\varphi(\mathbf{B}((\mathbf{f}\theta_{\lambda}\mathbf{f})_{\zeta\odot\lambda}))\subset \mathbf{B}(M_{\zeta}\otimes \Lambda_{\lambda})$. Now part (1) follows from Lemma \ref{isomorphism varphi}, and part (2) follows from \eqref{varphib integral}.
\end{proof}

Combining with Lemma \ref{a otimes Id is a based map}, we obtain

\begin{theorem}\label{thm:structure-trans}
Let $\lambda_1,\lambda_2\in X^+$ and $w\in W$. Then
\begin{enumerate}
\item there is a bijection
$$\mathbf{B}((\mathbf{f}\theta_{\lambda_2}\mathbf{f})_{-w\lambda_1\odot\lambda_2})\setminus\mathrm{ker}\,(a\otimes \mathrm{Id})\varphi_{-w\lambda_1,\lambda_2}\xrightarrow{(a\otimes \mathrm{Id})\varphi_{-w\lambda_1,\lambda_2}}\mathbf{B}({^{\omega}V_w(\lambda_1)}\otimes \Lambda_{\lambda_2});$$
\item let $\tilde{b}\in \mathbf{B}(\mathbf{f}\theta_{\lambda_2}\mathbf{f})$ with $r(\tilde{b})=\sum_{\tilde{b}_1,\tilde{b}_2\in \tilde{\mathbf{B}}}c_{\tilde{b}_1,\tilde{b}_2}^{\tilde{b}}\tilde{b}_1\otimes\tilde{b}_2$, then 
$$(a\otimes \mathrm{Id})\varphi_{-w\lambda_1,\lambda_2}(\tilde{b}^-v_{-w\lambda_1\odot\lambda_2})=\sum_{b_1^+\xi_{-\lambda_1}\in \mathbf{B}({^{\omega}V_w(\lambda_1)}),b_2^-\eta_{\lambda_2}\in \mathbf{B}(\Lambda_{\lambda_2})}c_{b_2\theta_{\lambda_2},b'_1}^{\tilde{b}}b_1^+\xi_{-\lambda_1}\otimes b_2^-\eta_{\lambda_2},$$
where $b'_1\in \mathbf{B}$ is the unique element such that $b_1^+\xi_{-\lambda_1}={b'_1}^- \xi_{-w \l_1}$.
\end{enumerate}
\end{theorem}

By the positivity property of $\tilde{\mathbf{B}}$ with respect to the comultiplication in Theorem \ref{14.4.13} and \cite[Lemma 14.2.2]{Lusztig-1993}, $\mathbf{B}({^{\omega}V_w(\lambda_1)}\otimes \Lambda_{\lambda_2})\subset \mathbb{N}[v^{-1}][\mathbf{B}({^{\omega}V_w(\lambda_1)})\otimes \mathbf{B}(\Lambda_{\lambda_2})]$. Moreover, by $\mathbf{B}({^{\omega}\Lambda_{\lambda}})=\bigcup_{w\in W}\mathbf{B}({^{\omega}V_{w}(\lambda)})$ and Proposition \ref{canonical basis of demazure otimes highest weight}, we have the following positivity property on the transition matrix from the basis $\mathbf{B}({^{\omega}\Lambda_{\lambda_1}})\otimes \mathbf{B}(\Lambda_{\lambda_2})$ to the canonical basis $\mathbf{B}({^{\omega}\Lambda_{\lambda_1}}\otimes \Lambda_{\lambda_2})$.  

\begin{theorem}\label{thm:trans-pos}
Let $\lambda_1,\lambda_2\in X^+$. Then 
$$\mathbf{B}({^{\omega}\Lambda_{\lambda_1}}\otimes \Lambda_{\lambda_2})\subset \mathbb{N}[v^{-1}][\mathbf{B}({^{\omega}\Lambda_{\lambda_1}})\otimes \mathbf{B}(\Lambda_{\lambda_2})].$$
\end{theorem}

\subsection{Thickening map}

Let $\lambda_1,\lambda_2\in X^+$ and $w\in W$. We define the {\it thickening map}
$$\mathrm{th}_{\lambda_1,\lambda_2,w}:\dot{\mathbf{B}}\rightarrow \tilde{\mathbf{B}}\sqcup\{0\},\dot{b}\mapsto \begin{cases}\tilde{b},\ &\textrm{if}\ \dot{b}(\xi_{-\lambda_1}\otimes \eta_{\lambda_2})\in \mathbf{B}({^{\omega}V_w(\lambda_1)\otimes \Lambda_{\lambda_2}});\\0, &\textrm{otherwise},\end{cases}$$
where in the first case, $\tilde{b}$ is the unique element in $\mathbf{B}(\mathbf{f}\theta_{\lambda_2}\mathbf{f})$ satisfying 
\begin{equation}\label{definition of thickening map}
\dot{b}(\xi_{-\lambda_1}\otimes \eta_{\lambda_2})=(a\otimes \mathrm{Id})\varphi_{-w\lambda_1,\lambda_2}(\tilde{b}^-v_{-w\lambda_1\odot\lambda_2}).   
\end{equation}
The existence and uniqueness of such $\tilde{b}$ follow from Theorem \ref{thm:structure-trans}.

For any fixed $\dot{b}\in\dot{\mathbf{B}}$, by Theorem \ref{25.2.1}, we have $\dot{b}(\xi_{-\lambda_1}\otimes \eta_{\lambda_2})\in \mathbf{B}({^{\omega}\Lambda_{\lambda_1}}\otimes \Lambda_{\lambda_2})$ for any sufficiently regular $\lambda_1,\lambda_2\in X^+$, that is, $\langle i,\lambda_1\rangle,\langle i,\lambda_2\rangle\in \mathbb{N}$ are sufficiently large for any $i\in I$. Then by $\mathbf{B}({^{\omega}\Lambda_{\lambda}})=\bigcup_{w\in W}\mathbf{B}({^{\omega}V_{w}(\lambda)})$ and Proposition \ref{canonical basis of demazure otimes highest weight}, $\dot{b}(\xi_{-\lambda_1}\otimes \eta_{\lambda_2})\in \mathbf{B}({}^{\omega}V_w(\lambda_1)\otimes\Lambda_{\lambda_2})$ for some $w\in W$. In this way, every element of the canonical basis $\dot{\mathbf{B}}$ can be realized (via the thickening map) as an element of the canonical basis $\tilde{\mathbf{B}}$ after choosing appropriate parameters $\lambda_1,\lambda_2,w$.

\subsection{Example}
We present the example in rank $1$.

Let $(I,\cdot)$ be the symmetric Cartan datum of type $A_1$ and $(Y,X,\langle,\rangle,...)$ be the simply connected root datum of type $(I,\cdot)$, that is, $I=\{i\}, i\cdot i=2$ and $Y=X=\mathbb{Z}$. Then the canonical basis of $\mathbf{f}$ is 
$$\mathbf{B}=\{\theta_i^{(k)}; k\in \mathbb{N}\}.$$ 
For dominant weights $m,n\in \mathbb{N}$, the canonical bases of ${^{\omega}\Lambda_m}$ and $\Lambda_{n}$ are 
$$\mathbf{B}({^{\omega}\Lambda_m})=\{E_i^{(k)}\xi_{-m}; 0\leqslant k\leqslant m\},\ \mathbf{B}(\Lambda_n)=\{F_i^{(l)}\eta_n; 0\leqslant l\leqslant n\}$$
respectively. The canonical basis of ${^{\omega}\Lambda_m}\otimes \Lambda_n$ is 
\begin{align*}
\mathbf{B}({^{\omega}\Lambda_m}\otimes \Lambda_n)=\{E_i^{(k)}\xi_{-m}\diamondsuit F_i^{(l)}\eta_n; 0\leqslant k\leqslant m,0\leqslant l\leqslant n\},
\end{align*}
where $E_i^{(k)}\xi_{-m}\diamondsuit F_i^{(l)}\eta_n$ is given by 
$$\begin{cases}
\sum_{s\in \mathbb{N}} v^{s(k-m-s)}\frac{[n-l+s]!}{[s]![n-l]!}E_i^{(k-s)}\xi_{-m}\otimes F_i^{(l-s)}\eta_n, &\textrm{if}\ n-m\leqslant l-k;\\
\sum_{s\in \mathbb{N}} v^{s(l-n-s)}\frac{[m-k+s]!}{[s]![m-k]!}E_i^{(k-s)}\xi_{-m}\otimes F_i^{(l-s)}\eta_n, &\textrm{if}\ n-m\geqslant l-k
\end{cases}$$
with the identification of two expressions for $n-m=l-k$ (see \cite[\S 6]{Lusztig-1992}). The canonical basis of $\dot{\mathbf{U}}$ is 
$$\dot{\mathbf{B}}=\{\theta_i^{(k)}\diamondsuit_t\,\theta_i^{(l)};t\in \mathbb{Z},k,l\in \mathbb{N}\},\ \theta_i^{(k)}\diamondsuit_t\,\theta_i^{(l)}=\begin{cases}
E_i^{(k)}1_{t-2l}F_i^{(l)},\ &\textrm{if}\ t\leqslant l-k;\\
F_i^{(l)}1_{t+2k}E_i^{(k)},\ &\textrm{if}\ t\geqslant l-k
\end{cases}$$
with the identification $E_i^{(k)}1_{t-2l}F_i^{(l)}=F_i^{(l)}1_{t+2k}E_i^{(k)}$ for $t=l-k$ (see \cite[\S 25.3.1]{Lusztig-1993}).

The thickening Cartan datum $(\tilde{I},\cdot)$ is of type $A_2$. The canonical basis of $\tilde{\mathbf{f}}$ is
$$\tilde{\mathbf{B}}=\{\theta_i^{(p)}\theta_{i'}^{(q)}\theta_i^{(r)},\theta_{i'}^{(r)}\theta_i^{(q)}\theta_{i'}^{(p)}; p,q,r\in \mathbb{N},q\geqslant p+r\}$$
with the identification $\theta_i^{(p)}\theta_{i'}^{(q)}\theta_i^{(r)}=\theta_{i'}^{(r)}\theta_i^{(q)}\theta_{i'}^{(p)}$ for $q=p+r$ (see \cite[\S 14.5.4]{Lusztig-1993}). Let $w=-1$. Then we have $(a\otimes \mathrm{Id})\varphi_{m,n}:(\mathbf{f}\theta_n\mathbf{f})_{m\odot n}\rightarrow {^{\omega}\Lambda_m}\otimes \Lambda_n$ such that 
$$E_i^{(k)}\xi_{-m}\diamondsuit F_i^{(l)}\eta_n=\begin{cases}
(a\otimes \mathrm{Id})\varphi_{m,n}(F_{i'}^{(n-l)}F_i^{(m-k+l)}F_{i'}^{(l)}v_{m\odot n}),\ &\textrm{if}\ n-m\leqslant l-k;\\
(a\otimes \mathrm{Id})\varphi_{m,n}(F_i^{(l)}F_{i'}^{(n)}F_i^{(m-k)})v_{m\odot n},\ &\textrm{if}\ n-m\geqslant l-k
\end{cases}$$
for any $0\leqslant k\leqslant m$ and $0\leqslant l\leqslant n$ (see \cite[Proposition 4.9]{Fang-Lan-2025}).

The thickening map $\mathrm{th}_{m,n,-1}:\dot{\mathbf{B}}\rightarrow \tilde{\mathbf{B}}\sqcup\{0\}$ is given by $\mathrm{th}_{m,n,-1}(\theta_i^{(k)}\diamondsuit_t\,\theta_i^{(l)})=0$ unless $t=n-m,0\leqslant k\leqslant m,0\leqslant l\leqslant n$, and 
$$\mathrm{th}_{m,n,-1}(\theta_i^{(k)}\diamondsuit_{n-m}\,\theta_i^{(l)})=\begin{cases}
\theta_{i'}^{(n-l)}\theta_i^{(m-k+l)}\theta_{i'}^{(l)},\ &\textrm{if}\ n-m\leqslant l-k;\\
\theta_i^{(l)}\theta_{i'}^{(n)}\theta_i^{(m-k)},\ &\textrm{if}\ n-m\geqslant l-k.
\end{cases}$$

\section{Structure constant of multiplication}\label{Structure constant of multiplication}

By using the thickening realizations, we compare the structure constants of the multiplication in the modified quantum group and the multiplication in the negative part of a larger quantum group with respect to their canonical bases.

\subsection{Spherical parabolic subalgebra}
Let $J \subset I$, and $\mathbf{U}_J^+,\mathbf{U}_J^-$ be the subalgebras of $\mathbf{U}$ generated by $E_j,F_j$ for any $j\in J$ respectively. By \cite[\S 23.2.1]{Lusztig-1993}, $\dot{\mathbf{U}}$ is a free $\mathbf{U}^+\otimes (\mathbf{U}^-)^{\mathrm{opp}}$-module with a basis $\{1_{\zeta};\zeta\in X\}$. Let $\dot{\mathbf{U}}_J$ be the $\mathbf{U}_J^+\otimes (\mathbf{U}^-)^{\mathrm{opp}}$-submodule generated by $\{1_{\zeta};\zeta\in X\}$, and ${^{\omega}\dot{\mathbf{U}}_J}$ be the $\mathbf{U}^+\otimes (\mathbf{U}_J^-)^{\mathrm{opp}}$-submodule generated by $\{1_{\zeta};\zeta\in X\}$. By \cite[\S 23.1.3]{Lusztig-1993}, both $\dot{\mathbf{U}}_J$ and ${^{\omega}\dot{\mathbf{U}}_J}$ are subalgebras of $\dot{\mathbf{U}}$. We call them the {\it parabolic subalgebra} and the {\it opposite parabolic subalgebra} of $\dot{\mathbf{U}}$ associated with $J$ respectively.  

A subset $J \subset I$ is called spherical if the subgroup $W_J$ of $W$ generated by $s_j$ for any $j \in J$ is finite. A spherical parabolic subalgebra of $\dot{\mathbf{U}}$ is either $\dot{\mathbf{U}}_J$ or ${}^{\omega}\dot{\mathbf{U}}_J$ for a spherical subset $J \subset I$.

For any $w \in W$, the set $J(w) := \{j \in I \mid s_j w < w\}$ is spherical. Conversely, given a spherical subset $J \subset I$, one can choose $w \in W$ sufficiently large (for example, the longest element of $W_J$ times a suitable element) such that $J = J(w)$.

\begin{lemma}\label{lem:Demazure-UJ}
Let $\l \in X^+,w \in W$ and $J=J(w)$. Then the projection map $a: M_{-w \l} \to {^{\omega}V_{w}(\lambda)}$ commutes with the action of $\dot{\mathbf{U}}_J$.
\end{lemma}
\begin{proof}
The projection map is already a $\mathbf{U}(\mathfrak{b}^-)$-module homomorphism. It remains to prove that it commutes with the action of $E_j$ for any $j\in J$. Let $r_j:\mathbf{f}\rightarrow \mathbf{f}$ and ${_jr}:\mathbf{f}\rightarrow \mathbf{f}$ be the linear maps defined in \cite[\S 1.2.13]{Lusztig-1993}. For any homogeneous $x\in \mathbf{f}$, by \cite[Proposition 3.1.6]{Lusztig-1993}, we have 
$$E_jx^--x^-E_j=\frac{K_j({_jr(x)})^--(r_j(x))^-K_{-j}}{v-v^{-1}}\in \mathbf{U}.$$
Note that $E_jv_{-w\lambda}=0$ in $M_{-w\lambda}$ and $E_j\xi_{-w\lambda}=0$ in ${^{\omega}V_w(\lambda)}$. So both $a(E_jx^-v_{-w\lambda})$ and $E_ja(x^-v_{-w\lambda})$ are equal to
$(v-v^{-1})^{-1}(K_j({_jr(x)})^-\xi_{-w\lambda}-(r_j(x))^-K_{-j}\xi_{-w\lambda})$.
\end{proof}

\subsection{A vanishing lemma}
Let $J\subset I$, and $\mathbf{f}_J$ be the subalgebra of $\mathbf{f}$ generated by $\theta_j$ for any $j\in J$. We regard $\mathbf{f}_J$ as the algebra defined in \S \ref{The algebra f} associated with the sub-Cartan datum $(J,\cdot)$ of $(I,\cdot)$, and let $\mathbf{B}_J$ be the canonical basis of $\mathbf{f}_J$. Then $\mathbf{B}_J\subset \mathbf{B}$. Similar to \cite[\S 23.2.1]{Lusztig-1993}, $\dot{\mathbf{U}}_J$ has a basis $\{b_1^+1_{\zeta}b_2^-;\zeta\in X,b_1\in \mathbf{B}_J,b_2\in \mathbf{B}\}$. Then by the proof of Theorem \ref{25.2.1}, $\{b_1\diamondsuit_{\zeta}b_2\in \dot{\mathbf{B}};\zeta\in X,b_1\in \mathbf{B}_J,b_2\in \mathbf{B}\}$ is a basis of $\dot{\mathbf{U}}_J$. Then by definition, it is equal to $\dot{\mathbf{B}}\cap\dot{\mathbf{U}}_J$. Similarly, ${^{\omega}\dot{\mathbf{U}}_J}$ has a basis $\dot{\mathbf{B}}\cap{^{\omega}\dot{\mathbf{U}}_J}=\{b_1\diamondsuit_{\zeta}b_2\in \dot{\mathbf{B}};\zeta\in X,b_1\in \mathbf{B},b_2\in \mathbf{B}_J\}$. 

Let $\dot{\tilde{\mathbf{U}}}$ be the modified quantum group corresponding to $\tilde{\mathbf{U}}$ associated with the thickening root datum $(\tilde{Y},\tilde{X},\langle\,,\,\rangle,\ldots)$, and $\dot{\tilde{\mathbf{B}}}$ be its canonical basis. Then there are natural embeddings $\dot{\mathbf{U}}\subset \dot{\tilde{\mathbf{U}}}$ and $\dot{\mathbf{B}}\subset \dot{\tilde{\mathbf{B}}}$.

\begin{lemma}\label{canonical basis of U dot acts on Verma}
Let $\lambda_1,\lambda_2\in X^+,w\in W$ and $J=J(w)$. Then for any $b_1\in \mathbf{B}_J$ with $b_1 \neq 1$ and $b_2\in \mathbf{B}(\mathbf{f}\theta_{\lambda_2}\mathbf{f})$, we have 
$$(a\otimes \mathrm{Id})\varphi_{-w\lambda_1,\lambda_2}\bigl((b_1\diamondsuit_{-w\lambda_1\odot\lambda_2}b_2)v_{-w\lambda_1\odot\lambda_2}\bigr)=0.$$
\end{lemma}
\begin{proof}
We define two auxiliary dominant weights $\tilde{\lambda}_1,\tilde{\lambda}_2\in \tilde{X}^+$ by
\begin{align*}
\langle i,\tilde{\lambda}_1\rangle=\begin{cases}
0, &\textrm{if}\ s_iw<w;\\
\langle i,w\lambda_1\rangle, &\textrm{if}\ s_iw>w,
\end{cases}\ &\langle i,\tilde{\lambda}_2\rangle=\begin{cases}
\langle i,-w\lambda_1\rangle, &\textrm{if}\ s_iw<w;\\
0, &\textrm{if}\ s_iw>w,\end{cases}\ \textrm{for}\ i\in I;\\
\langle i',\tilde{\lambda}_1\rangle=0,\ &\langle i',\tilde{\lambda}_2\rangle=\langle i,\lambda_2\rangle\ \textrm{for}\ i'\in I'.
\end{align*}
A direct verification using \eqref{definition of odot} shows that $-w\lambda_1\odot\lambda_2=\tilde{\lambda}_2-\tilde{\lambda}_1$. 

Consider the following $\tilde{\mathbf{U}}$-module homomorphisms
$$\xymatrix@C=2.5cm{\dot{\tilde{\mathbf{U}}}1_{-w\lambda_1\odot\lambda_2} \ar[d]_-{\dot{u}\mapsto \dot{u}v_{-w\lambda_1\odot\lambda_2}} \ar@{->>}[r]^-{\dot{u}\mapsto \dot{u}(\xi_{-\tilde{\lambda}_1}\otimes v_{\tilde{\lambda}_2})} &{^{\omega}\Lambda_{\tilde{\lambda}_1}}\otimes M_{\tilde{\lambda}_2} \ar@{->>}[r]^-{\mathrm{Id}\otimes \pi_{\tilde{\lambda}_2}} \ar@{-->}[ld]^f &{^{\omega}\Lambda_{\tilde{\lambda}_1}}\otimes \Lambda_{\tilde{\lambda}_2}\\
M_{-w\lambda_1\odot\lambda_2}.}$$
By \cite[\S 23.3.1 \& \S 23.3.5]{Lusztig-1993}, $\dot{\tilde{\mathbf{U}}}1_{-w\lambda_1\odot\lambda_2}\rightarrow {^{\omega}\Lambda_{\tilde{\lambda}_1}}\otimes M_{\tilde{\lambda}_2}, \dot{u}\mapsto \dot{u}(\xi_{-\tilde{\lambda}_1}\otimes v_{\tilde \l_2})$ is surjective with kernel $\sum_{\tilde{i}\in \tilde{I},n>\langle \tilde{i}, \tilde{\lambda}_1\rangle}\dot{\tilde{\mathbf{U}}}E_{\tilde{i}}^{(n)}1_{-w\lambda_1\odot\lambda_2}$. It is clear that the kernel acts on $v_{-w\lambda_1\odot\lambda_2}$ by $0$, and so $\dot{\tilde{\mathbf{U}}}1_{-w\lambda_1\odot\lambda_2}\rightarrow M_{-w\lambda_1\odot\lambda_2},\dot{u}\mapsto \dot{u} v_{-w\lambda_1\odot\lambda_2}$ factor though a $\tilde{\mathbf{U}}$-module homomorphism $f:{^{\omega}\Lambda_{\tilde{\lambda}_1}}\otimes M_{\tilde{\lambda}_2}\rightarrow M_{-w\lambda_1\odot\lambda_2}$ such that $f(\xi_{-\tilde{\lambda}_1}\otimes v_{\tilde{\lambda}_2})=v_{-w\lambda_1\odot\lambda_2}$. In particular, we have 
\begin{equation}\label{factor through f}
(b_1\diamondsuit_{-w\lambda_1\odot\lambda_2}b_2)v_{-w\lambda_1\odot\lambda_2}=f((b_1\diamondsuit_{-w\lambda_1\odot\lambda_2}b_2)(\xi_{-\tilde{\lambda}_1}\otimes v_{\tilde{\lambda}_2})).    
\end{equation}

Since $b_1\in \mathbf{B}_J$ and $b_2\in \mathbf{B}(\mathbf{f}\theta_{\lambda_2}\mathbf{f})$, by the construction of $b_1\diamondsuit_{-w\lambda_1\odot\lambda_2}b_2\in \dot{\tilde{\mathbf{B}}}$ in Theorem \ref{25.2.1} and the fact that $(\mathbf{f}\theta_{\lambda_2}\mathbf{f})_{\tilde{\lambda}}$ is a $\mathbf{U}$-submodule of $M_{\tilde{\lambda}}$ for any $\tilde{\lambda}\in \tilde{X}^+$, we know that $b_1\diamondsuit_{-w\lambda_1\odot\lambda_2}b_2$ is in the $\mathcal{A}$-submodule of $\dot{\tilde{\mathbf{U}}}$ generated by ${b'_1}^+{b'_2}^-1_{-w\lambda_1\odot\lambda_2}$ for any $b'_1\in \mathbf{B}_J$ and $b'_2\in \mathbf{B}(\mathbf{f}\theta_{\lambda_2}\mathbf{f})$. We can write
$$(b_1\diamondsuit_{-w\lambda_1\odot\lambda_2}b_2)(\xi_{-\tilde{\lambda}_1}\otimes v_{\tilde \l_2})=\sum_{b'_1 \in \mathbf{B}_J, b'_2 \in \mathbf{B}(\mathbf{f}\theta_{\lambda_2}\mathbf{f})} c_{b'_1,b'_2} {b'_1}^+\xi_{-\tilde{\lambda}_1}\otimes {b'_2}^- v_{\tilde \l_2}.$$ 
Since $\langle j,\tilde{\lambda}_1\rangle=0$ for any $j\in J$, we have ${b'_1}^+\xi_{-\tilde{\lambda}_1}=0$ for any $b'_1 \in \mathbf{B}_J$ with $b'_1 \neq 1$. Hence the linear combination can be rewritten as 
\begin{equation}\label{linear combination}
(b_1\diamondsuit_{-w\lambda_1\odot\lambda_2}b_2)(\xi_{-\tilde{\lambda}_1}\otimes v_{\tilde \l_2})=\sum_{b \in \mathbf{B}(\mathbf{f}\theta_{\lambda_2}\mathbf{f})}c_{b}\,\xi_{-\tilde{\lambda}_1}\otimes b^- v_{\tilde \l_2}.
\end{equation}
Similar, since $b_1\in \mathbf{B}_J$ and $b_1\neq 1$, we have $b_1^+\xi_{-\tilde{\lambda}_1}=0$. By Theorem \ref{25.2.1}, applying $\mathrm{Id}\otimes \pi_{\tilde{\lambda}_2}:{^{\omega}\Lambda_{\tilde{\lambda}_1}}\otimes M_{\tilde{\lambda}_2}\rightarrow {^{\omega}\Lambda_{\tilde{\lambda}_1}}\otimes \Lambda_{\tilde{\lambda}_2}$ to \eqref{linear combination}, we obtain
$$0=(b_1\diamondsuit_{-w\lambda_1\odot\lambda_2}b_2)(\xi_{-\tilde{\lambda}_1}\otimes \eta_{\tilde{\lambda}_2})=\sum_{b\in \mathbf{B}(\mathbf{f}\theta_{\lambda_2}\mathbf{f})}c_b\,\xi_{-\tilde{\lambda}_1}\otimes b^-\eta_{\tilde{\lambda}_2}.$$
By \cite[Theorem 14.4.11]{Lusztig-1993}, \{$b^-\eta_{\tilde{\lambda}_2};b\in \mathbf{B}(\mathbf{f}\theta_{\lambda_2}\mathbf{f}),b^-\eta_{\tilde{\lambda}_2}\neq 0\}\subset \mathbf{B}(\Lambda_{\tilde{\lambda}_2})$ is linearly independent, and so $c_b=0$ for any $b\in \mathbf{B}(\mathbf{f}\theta_{\lambda_2}\mathbf{f})$ with $b^- \eta_{\tilde \l_2} \neq 0$. By \eqref{factor through f}, we have
\begin{align*}
&(b_1\diamondsuit_{-w\lambda_1\odot\lambda_2}b_2) v_{-w\lambda_1\odot\lambda_2} 
=\sum_{b\in \mathbf{B}(\mathbf{f}\theta_{\lambda_2}\mathbf{f});b^-\eta_{\tilde{\lambda}_2}=0} f(c_b\,\xi_{-\tilde{\lambda}_1}\otimes b^- v_{\tilde \l_2})\\
=&\sum_{b\in \mathbf{B}(\mathbf{f}\theta_{\lambda_2}\mathbf{f});b^-\eta_{\tilde{\lambda}_2}=0} c_b\,b^- f(\xi_{-\tilde{\lambda}_1}\otimes v_{\tilde \l_2})
=\sum_{b\in \mathbf{B}(\mathbf{f}\theta_{\lambda_2}\mathbf{f});b^-\eta_{\tilde{\lambda}_2}=0} c_b\,b^- v_{-w\lambda_1\odot\lambda_2}
\end{align*}
It remains to prove that it belongs to $\mathrm{ker}\,(a\otimes \mathrm{Id})\varphi_{-w\lambda_1,\lambda_2}$.

Let $b \in \mathbf{B}(\mathbf{f}\theta_{\lambda_2}\mathbf{f})$ with $b^-\eta_{\tilde{\lambda}_2}=0$. By \eqref{14.4.11}, $b\in \bigcup_{\tilde{i}\in \tilde{I}}(\tilde{\mathbf{B}}\cap \tilde{\mathbf{f}}\theta_{\tilde{i}}^{\langle \tilde{i},\tilde{\lambda}_2\rangle+1})$. Note that $|b|\in \sum_{i\in I}\langle i,\lambda_2\rangle i'+\mathbb{N}[I]$. Hence $b\not\in \bigcup_{i'\in I'}\tilde{\mathbf{f}}\theta_{i'}^{\langle i',\tilde{\lambda}_2\rangle+1}=\bigcup_{i'\in I'}\tilde{\mathbf{f}}\theta_{i'}^{\langle i,\lambda_2\rangle+1}$ and $b\in \tilde{\mathbf{f}}\theta_i^{\langle i,\tilde{\lambda}_2\rangle+1}$ for some $i\in I$. We write $b=x\theta_i^{\langle i,\tilde{\lambda}_2\rangle+1}$ with $x\in \tilde{\mathbf{f}}$, and write $r(x)=\sum x_1\otimes x_2,r(\theta_i^{\langle i,\tilde{\lambda}_2\rangle+1})=y_1\otimes y_2$ with homogeneous $x_1,x_2\in \tilde{\mathbf{f}},y_1,y_2\in \mathbf{f}$. By the definition of $\varphi_{-w\lambda_1,\lambda_2}$, we have 
$$(a\otimes \mathrm{Id})\varphi_{-w\lambda_1,\lambda_2}(b^-v_{-w\lambda_1\odot\lambda_2})
=\sum v^{|x_2|\cdot |y_1|}x_2^-y_2^-\xi_{-w\lambda_1}\otimes  \phi_{\lambda_2}^{-1}(x_1^-y_1^-\eta_{0\odot\lambda_2}).$$
Since $\langle k,0\odot\lambda_2\rangle=0$ for any $k\in I$, we have $y_1^-\eta_{0\odot\lambda_2}=0$ whenever $|y_1|\neq 0$. Notice that $x_2^-F_i^{\langle i,\tilde{\lambda}_2\rangle+1}\xi_{-w\lambda_1}=0$, and so $(a\otimes \mathrm{Id})\varphi_{-w\lambda_1,\lambda_2}(b)=0$. 
\end{proof}

\subsection{Structure constant}

Let $\zeta\in X$ and $b_1,b_2\in \mathbf{B}$. Then we have $$b_1\diamondsuit_{\zeta}b_2\in 1_{\zeta+|b_1|-|b_2|}\dot{\mathbf{U}}1_{\zeta}$$ 
(see \cite[Remark 25.2.4]{Lusztig-1993}). We denote $|b_1\diamondsuit_{\zeta} b_2|=|b_1|-|b_2|\in \mathbb{Z}[I]$.

Let $\sigma:\mathbf{f}\rightarrow \mathbf{f}^{\mathrm{opp}}$ be the algebra isomorphism defined by $\sigma(\theta_i)=\theta_i$ for any $i\in I$. By \cite[Theorem 14.4.3]{Lusztig-1993}, $\sigma(\mathbf{B})=\mathbf{B}$. Similarly, we have $\sigma:\tilde{\mathbf{f}}\rightarrow \tilde{\mathbf{f}}^{\mathrm{opp}}$ such that $\sigma(\tilde{\mathbf{B}})=\tilde{\mathbf{B}}$. Let $\sigma:\mathbf{U}\rightarrow \mathbf{U}^{\mathrm{opp}}$ be the algebra isomorphism defined by $\sigma(E_i)=E_i,\sigma(F_i)=F_i,\sigma(K_{\mu})=K_{-\mu}$ for any $i\in I$ and $\mu\in Y$. Then $\sigma(x^+)=(\sigma(x))^+$ and $\sigma(x^-)=(\sigma(x))^-$ for any $x\in \mathbf{f}$. It induces an algebra isomorphism $\dot{\mathbf{U}}\rightarrow \dot{\mathbf{U}}^{\mathrm{opp}}$, still denoted by $\sigma$ (see \cite[\S 23.1.6]{Lusztig-1993}). By \cite[Theorem 4.3.2]{Kashiwara-1994}, $\dot{\sigma}(\dot{\mathbf{B}})=\dot{\mathbf{B}}$.

Let $(\tilde{\tilde{I}},\cdot)$ be the thickening Cartan datum of $(\tilde{I},\cdot)$, and $\tilde{\tilde{\mathbf{f}}}$ be the corresponding algebra defined in \S \ref{The algebra f} with the canonical basis $\tilde{\tilde{\mathbf{B}}}$. For any $b_1,b_2\in \tilde{\tilde{\mathbf{B}}}$, we write $b_1b_2=\sum_{b\in \tilde{\tilde{\mathbf{B}}}}m_{b_1,b_2}^b b$. 

\begin{theorem}\label{thm:structure-multiplication}
Let $J \subset I$ be a spherical subset. Then for any $\dot{b}\in \dot{\mathbf{B}}\cap \dot{\mathbf{U}}_J$ and $\dot{b}'\in \dot{\mathbf{B}}$, we have
$$\dot{b}\dot{b}'=\sum_{\dot{b}''\in \dot{\mathbf{B}}}m_{\sigma(\mathrm{th}_{\lambda_1,\lambda_2,w}(\dot{b}')),\widetilde{\mathrm{th}}_{\tilde{\lambda}_1,\tilde{\lambda}_2,\tilde{w}}(\sigma(\dot{\tilde{b}}))}^{\widetilde{\mathrm{th}}_{\tilde{\lambda}_1,\tilde{\lambda}_2,\tilde{w}}(\sigma(1\diamondsuit_{-w\lambda_1\odot\lambda_2}\,\mathrm{th}_{\lambda_1,\lambda_2,w}(\dot{b}'')))}\dot{b}'',$$
for sufficiently regular $\l_1, \l_2 \in X^+$, $\tilde{\lambda}_1,\tilde{\lambda}_2\in \tilde{X}^+$ and sufficiently large $w\in W,\tilde{w}\in \tilde{W}$ with $\dot{b}' \in \dot{\mathbf{U}}1_{\lambda_2-\lambda_1}$, $J=J(w)$ and $\dot{\tilde{b}}=b_1\diamondsuit_{-w\lambda_1\odot\lambda_2-|\mathrm{th}_{\lambda_1,\lambda_2,w}(\dot{b}')|}b_2\in 1_{\tilde{\lambda}_1-\tilde{\lambda}_2}\dot{\tilde{\mathbf{U}}}$, where $b_1\in \mathbf{B}_J$ and $b_2\in \mathbf{B}$ are determined by $\dot{b}=b_1\diamondsuit_{\lambda_2-\lambda_1+|\dot{b}'|}b_2$.
\end{theorem} 
\begin{proof}
We simply denote $\mathrm{th}_{\lambda_1,\lambda_2,w}=\mathrm{th},\varphi_{-w\lambda_1,\lambda_2}=\varphi$ and $\widetilde{\mathrm{th}}_{\tilde{\lambda}_1,\tilde{\lambda}_2,\tilde{w}}=\widetilde{\mathrm{th}}$.

By Theorem \ref{25.2.1}, we have $\dot{b}' \in \dot{\mathbf{U}}1_{\zeta}$ for some $\zeta\in X$. We can write $\dot{b}\dot{b}'=\sum_{\dot{b}''\in \dot{\mathbf{B}}\cap \dot{\mathbf{U}}1_{\zeta}}\dot{m}_{\dot{b},\dot{b}'}^{\dot{b}''}\dot{b}''$, and choose sufficiently regular $\l_1, \l_2 \in X^+$ with $\zeta=\l_2-\l_1$ such that $\dot{b}'(\xi_{-\lambda_1}\otimes \eta_{\lambda_2})\neq 0$ and $\dot{b}''(\xi_{-\lambda_1}\otimes \eta_{\lambda_2})\neq 0$ for any $\dot{b}'' \in \dot{\mathbf{B}}$ with $\dot{m}_{\dot{b},\dot{b}'}^{\dot{b}''}\neq 0$. Then 
\begin{equation}\label{structure constant m dot}
\dot{b}\dot{b}'(\xi_{-\lambda_1}\otimes \eta_{\lambda_2})=\sum_{\dot{b}''\in \dot{\mathbf{B}}\cap \dot{\mathbf{U}}1_{\lambda_2-\lambda_1}}\dot{m}_{\dot{b},\dot{b}'}^{\dot{b}''}\dot{b}''(\xi_{-\lambda_1}\otimes \eta_{\lambda_2}).    
\end{equation} 
By Theorem \ref{25.2.1}, $\{\dot{b}''(\xi_{-\lambda_1}\otimes \eta_{\lambda_2});\dot{b}'' \in \dot{\mathbf{B}},\dot{m}_{\dot{b},\dot{b}'}^{\dot{b}''}\neq 0\}\subset \mathbf{B}(^{\omega} \L_{\l_1} \otimes \L_{\l_2})$ is linearly independent. Let $w \in W$ with $J=J(w)$ and $\dot{b}'(\xi_{-\lambda_1}\otimes \eta_{\lambda_2})\in \mathbf{B}({^{\omega}V_w(\lambda_1)}\otimes \Lambda_{\lambda_2})$. Then $\dot{b}\dot{b}'(\xi_{-\lambda_1}\otimes \eta_{\lambda_2})\in {^{\omega}V_w(\lambda_1)}\otimes \Lambda_{\lambda_2}$. By Proposition \ref{canonical basis of demazure otimes highest weight}, we have $\dot{b}''(\xi_{-\lambda_1}\otimes \eta_{\lambda_2})\in \mathbf{B}({^{\omega}V_w(\lambda_1)}\otimes \Lambda_{\lambda_2})$ for any $\dot{b}'' \in \dot{\mathbf{B}}$ with $\dot{m}_{\dot{b},\dot{b}'}^{\dot{b}''}\neq 0$. 

Notice that $(\mathrm{th}(\dot{b}'))^-v_{-w\lambda_1\odot\lambda_2}\in \mathbf{B}((\mathbf{f}\theta_{\lambda_2}\mathbf{f})_{-w\lambda_1\odot\lambda_2})$ and $(\mathbf{f}\theta_{\lambda_2}\mathbf{f})_{-w\lambda_1\odot\lambda_2}$ is a $\mathbf{U}$-submodule of $M_{-w\lambda_1\odot\lambda_2}$, we can write 
\begin{equation}\label{structure constant a}
\dot{b}(\mathrm{th}(\dot{b}'))^-v_{-w\lambda_1\odot\lambda_2}=\sum_{\tilde{b}''\in \mathbf{B}(\mathbf{f}\theta_{\lambda_2}\mathbf{f})} a_{\dot{b},\mathrm{th}(\dot{b}')}^{\tilde{b}''}\tilde{b}^{''-}v_{-w\lambda_1\odot\lambda_2}.  
\end{equation}
By \eqref{definition of thickening map}, Lemma \ref{lem:Demazure-UJ} and Lemma \ref{isomorphism varphi}, we have 
\begin{equation}\label{structure constant a'}
\begin{aligned}
\dot{b}\dot{b}'(\xi_{-\lambda_1}\otimes \eta_{\lambda_2})=&\dot{b}(a\otimes \mathrm{Id})\varphi((\mathrm{th}(\dot{b}'))^-v_{-w\lambda_1\odot\lambda_2})\\
=&(a\otimes \mathrm{Id})\varphi(\dot{b}(\mathrm{th}(\dot{b}'))^-v_{-w\lambda_1\odot\lambda_2})\\
=&\sum_{\tilde{b}''\in \mathbf{B}(\mathbf{f}\theta_{\lambda_2}\mathbf{f})}a_{\dot{b},\mathrm{th}(\dot{b}')}^{\tilde{b}''}(a\otimes \mathrm{Id})\varphi(\tilde{b}^{''-}v_{-w\lambda_1\odot\lambda_2}).   
\end{aligned}
\end{equation}
Comparing \eqref{structure constant m dot} with \eqref{structure constant a'}, we obtain the relation between the structure constants of the multiplication in $\dot{\mathbf{U}}$ and the action of $\dot{\mathbf{U}}$ on $(\mathbf{f}\theta_{\lambda_2}\mathbf{f})_{-w\lambda_1\odot\lambda_2}$. More precisely,

(a) $\dot{m}_{\dot{b},\dot{b}'}^{\dot{b}''}=a_{\dot{b},\mathrm{th}(\dot{b}')}^{\mathrm{th}(\dot{b}'')}$ for any $\dot{b}'' \in \dot{\mathbf{B}}$ with $\dot{m}_{\dot{b},\dot{b}'}^{\dot{b}''}\not=0$.

In order to determine the structure constant $a_{\dot{b},\mathrm{th}(\dot{b}')}^{\mathrm{th}(\dot{b}'')}$, we first consider the special case that $\dot{b}=1\diamondsuit_{\zeta+|\dot{b}'|}b$ with $b\in \mathbf{B}$. In this case, it is clear that $b\,\mathrm{th}(\dot{b}')\in \mathbf{f}\theta_{\lambda_2}\mathbf{f}$. We have $b\,\mathrm{th}(\dot{b}')=\sum_{\tilde{b}''\in \mathbf{B}(\mathbf{f}\theta_{\lambda_2}\mathbf{f})}m_{b,\mathrm{th}(\dot{b}')}^{\tilde{b}''}\tilde{b}''$. Note that
\begin{equation}\label{weight relation}
\langle i,\zeta+|\dot{b}'|\rangle=\langle i,-w\lambda_1\odot\lambda_2-|\mathrm{th}(\dot{b}')|\rangle\ \textrm{for any}\ i\in I.   
\end{equation}
Hence 
\begin{align*}
(1\diamondsuit_{\zeta+|\dot{b}'|}b)(\mathrm{th}(\dot{b}'))^-v_{-w\lambda_1\odot\lambda_2}=&b^-(\mathrm{th}(\dot{b}'))^-v_{-w\lambda_1\odot\lambda_2}
=(b\,\mathrm{th}(\dot{b}'))^-v_{-w\lambda_1\odot\lambda_2}\\
=&\sum_{\tilde{b}''\in \mathbf{B}(\mathbf{f}\theta_{\lambda_2}\mathbf{f})}m_{b,\mathrm{th}(\dot{b}')}^{\tilde{b''}}\tilde{b}^{''-}v_{-w\lambda_1\odot\lambda_2}.
\end{align*}
Comparing it with \eqref{structure constant a}, we obtain $a_{1\diamondsuit_{\zeta+|\dot{b}'|}b,\mathrm{th}(\dot{b}')}^{\tilde{b}''}=m_{b,\mathrm{th}(\dot{b}')}^{\tilde{b''}}$ for any $\tilde{b}''\in \mathbf{B}(\mathbf{f}\theta_{\lambda_2}\mathbf{f})$. This relates the structure constants of the action of $1\diamondsuit_{\zeta+|\dot{b}'|}b$ on $(\mathbf{f}\theta_{\lambda_2}\mathbf{f})_{-w\lambda_1\odot\lambda_2}$ with the multiplication in $\tilde{\mathbf{f}}$. By (a), $\dot{m}_{1\diamondsuit_{\zeta+|\dot{b}'|}b,\dot{b}'}^{\dot{b}''}=m_{b,\mathrm{th}(\dot{b}')}^{\mathrm{th}(\dot{b}'')}$ for any $\dot{b}'' \in \dot{\mathbf{B}}$ with $\dot{m}_{\dot{b}',\dot{b}}^{\dot{b}''}\not=0$, and so $(1\diamondsuit_{\zeta+|\dot{b}'|}b)\dot{b}'=\sum_{\dot{b}''\in \dot{\mathbf{B}}\cap \dot{\mathbf{U}}1_{\lambda_2-\lambda_1}}m_{b,\mathrm{th}(\dot{b}')}^{\mathrm{th}(\dot{b}'')}\dot{b}''$. Applying the algebra isomorphism $\sigma:\dot{\mathbf{U}}\rightarrow \dot{\mathbf{U}}^{\mathrm{opp}}$, we obtain 

(b) for any $\dot{b}\in \dot{\mathbf{B}}\cap 1_{-\zeta}\dot{\mathbf{U}},b\in \mathbf{B}$, sufficiently regular $\lambda_1,\lambda_2\in X^+$ and sufficiently large $w\in W$ with $\zeta=\lambda_2-\lambda_1$ and $\sigma(\dot{b})(\xi_{-\lambda_1}\otimes \eta_{\lambda_2})\in \mathbf{B}({^{\omega}V_w(\lambda_1)\otimes \Lambda_{\lambda_2}})$, we have
$$\dot{b}(1\diamondsuit_{-\zeta-|\dot{b}|+|b|}b)=\sum_{\dot{b}''\in \dot{\mathbf{B}}\cap 1_{\lambda_1-\lambda_2}\dot{\mathbf{U}}}m_{\sigma(b),\mathrm{th}(\sigma(\dot{b}))}^{\mathrm{th}(\sigma(\dot{b}''))}\dot{b}''.$$ 

Now we consider the general case. Applying (b) in the larger modified quantum group $\dot{\tilde{\mathbf{U}}}$ for $\dot{\tilde{b}}\in \dot{\tilde{\mathbf{B}}}\cap 1_{-w\lambda_1\odot\lambda_2-|\mathrm{th}(\dot{b}')|+|\dot{b}|}\dot{\tilde{\mathbf{U}}},\mathrm{th}(\dot{b}')\in \tilde{\mathbf{B}}$, sufficiently regular $\tilde{\lambda}_1,\tilde{\lambda}_2\in X^+$ and sufficiently large $\tilde{w}\in \tilde{W}$ with $-w\lambda_1\odot\lambda_2-|\mathrm{th}(\dot{b}')|+|\dot{b}|=\tilde{\lambda}_1-\tilde{\lambda}_2$ and $\sigma(\dot{\tilde{b}})(\xi_{-\tilde{\lambda}_1}\otimes \eta_{\tilde{\lambda}_2})\in \mathbf{B}({^{\omega}V_{\tilde{w}}(\tilde{\lambda}_1)\otimes \Lambda_{\tilde{\lambda}_2}})$, we have  
$$\dot{\tilde{b}}(1\diamondsuit_{-w\lambda_1\odot\lambda_2}\,\mathrm{th}(\dot{b}'))=\sum_{\dot{\tilde{b}}''\in \dot{\tilde{\mathbf{B}}}\cap 1_{\tilde{\lambda}_1-\tilde{\lambda}_2}\dot{\tilde{\mathbf{U}}}} m_{\sigma(\mathrm{th}(\dot{b}')),\widetilde{\mathrm{th}}(\sigma(\dot{\tilde{b}}))}^{\widetilde{\mathrm{th}}(\sigma(\dot{\tilde{b}}''))}\dot{\tilde{b}}''.$$ 
By the constructions of $\dot{b}=b_1\diamondsuit_{\zeta+|\dot{b}'|}b_2\in \dot{\mathbf{B}},\dot{\tilde{b}}=b_1\diamondsuit_{-w\lambda_1\odot\lambda_2-|\mathrm{th}(\dot{b}')|}b_2\in \dot{\tilde{\mathbf{B}}}$ in Theorem \ref{25.2.1} and \eqref{weight relation}, we have  
\begin{equation}\label{the product acts on Verma}
\begin{aligned}
\dot{b}(\mathrm{th}(\dot{b}'))^-v_{-w\lambda_1\odot\lambda_2}=&\dot{\tilde{b}}(\mathrm{th}(\dot{b}'))^-v_{-w\lambda_1\odot\lambda_2}=\dot{\tilde{b}}(1\diamondsuit_{-w\lambda_1\odot\lambda_2}\,\mathrm{th}(\dot{b}'))v_{-w\lambda_1\odot\lambda_2}\\
=&\sum_{\dot{\tilde{b}}''\in \dot{\tilde{\mathbf{B}}}\cap 1_{\tilde{\lambda}_1-\tilde{\lambda}_2}\dot{\tilde{\mathbf{U}}}} m_{\sigma(\mathrm{th}(\dot{b}')),\widetilde{\mathrm{th}}(\sigma(\dot{\tilde{b}}))}^{\widetilde{\mathrm{th}}(\sigma(\dot{\tilde{b}}''))}\dot{\tilde{b}}''v_{-w\lambda_1\odot\lambda_2}. 
\end{aligned}
\end{equation}
Since $\dot{b}\in \dot{\mathbf{B}}\cap\dot{\mathbf{U}}_J$ and $\mathrm{th}(\dot{b}')\in \mathbf{B}(\mathbf{f}\theta_{\lambda_2}\mathbf{f})$, we have $b_1\in \mathbf{B}_J$, and the product
$$\dot{\tilde{b}}(1\diamondsuit_{-w\lambda_1\odot\lambda_2}\,\mathrm{th}(\dot{b}'))=\dot{\tilde{b}}(\mathrm{th}(\dot{b}'))^-1_{-w\lambda_1\odot\lambda_2}$$
belongs to the $\mathcal{A}$-submodule of $\dot{\tilde{\mathbf{U}}}$ generated by $\tilde{b}_1^+\tilde{b}_2^-1_{-w\lambda_1\odot\lambda_2}$ for any $\tilde{b}_1\in \mathbf{B}_J$ and $\tilde{b}_2\in \mathbf{B}(\mathbf{f}\theta_{\lambda_2}\mathbf{f})$. By the proof of Theorem \ref{25.2.1}, for any $\dot{\tilde{b}}''\in \dot{\tilde{\mathbf{B}}}$ with $m_{\sigma(\mathrm{th}(\dot{b}')),\widetilde{\mathrm{th}}(\sigma(\dot{\tilde{b}}))}^{\widetilde{\mathrm{th}}(\sigma(\dot{\tilde{b}}''))}\not=0$, we have $\dot{\tilde{b}}''=\tilde{b}''_1\diamondsuit_{-w\lambda_1\odot\lambda_2}\tilde{b}''_2$ for some $\tilde{b}''_1\in \mathbf{B}_J,\tilde{b}''_2\in \mathbf{B}(\mathbf{f}\theta_{\lambda_2}\mathbf{f})$. We write \eqref{the product acts on Verma} into the linear combination of $\mathbf{B}((\mathbf{f}\theta_{\lambda_2}\mathbf{f})_{-w\lambda_1\odot\lambda_2})$, then compare the coefficients of $(\mathrm{th}(\dot{b}''))^-v_{-w\lambda_1\odot\lambda_2}$ in \eqref{structure constant a} and \eqref{the product acts on Verma} for any $\dot{b}'' \in \dot{\mathbf{B}}$ with $\dot{m}_{\dot{b},\dot{b}'}^{\dot{b}''}\not=0$. Note that $(\mathrm{th}(\dot{b}''))^-v_{-w\lambda_1\odot\lambda_2}\notin \mathrm{ker}\,(a\otimes \mathrm{Id})\varphi_{-w\lambda_1,\lambda_2}$. 
By Lemma \ref{a otimes Id is a based map} and Lemma \ref{canonical basis of U dot acts on Verma}, the coefficient of $(\mathrm{th}(\dot{b}''))^-v_{-w\lambda_1\odot\lambda_2}$ in \eqref{the product acts on Verma} is $m_{\sigma(\mathrm{th}(\dot{b}')),\widetilde{\mathrm{th}}(\sigma(\dot{\tilde{b}}))}^{\widetilde{\mathrm{th}}(\sigma(1\diamondsuit_{-w\lambda_1\odot\lambda_2}\,\mathrm{th}(\dot{b}'')))}$, and so $\dot{m}_{\dot{b},\dot{b}'}^{\dot{b}''}=a_{\dot{b},\mathrm{th}(\dot{b}')}^{\mathrm{th}(\dot{b}'')}=m_{\sigma(\mathrm{th}(\dot{b}')),\widetilde{\mathrm{th}}(\sigma(\dot{\tilde{b}}))}^{\widetilde{\mathrm{th}}(\sigma(1\diamondsuit_{-w\lambda_1\odot\lambda_2}\,\mathrm{th}(\dot{b}'')))}$. 
\end{proof}

\section{Positivity of canonical basis of modified quantum group}\label{Positivity of canonical basis of modified quantum group}

\subsection{The involution $\omega$ and canonical basis}
Recall the algebra automorphism $\omega:\mathbf{U}\rightarrow \mathbf{U}$ defined by $\o(E_i)=F_i$, $\o(F_i)=E_i,\o(K_\mu)=K_{-\mu}$ for any $i \in I$ and $\mu \in Y$. It induces an automorphism of $\dot{\mathbf{U}}$ (see \cite[\S 23.1.6]{Lusztig-1993}), also denoted by $\o$. By \cite[Corollary 26.3.2]{Lusztig-1993}, $\omega(\dot{\mathbf{B}})\subset \pm\dot{\mathbf{B}}$. Lusztig conjectured that $\omega(\dot{\mathbf{B}})=\dot{\mathbf{B}}$. When $(I,\cdot)$ is of finite type, it is proved by Lusztig in \cite[Proposition 3.16]{Lusztig-2023}. In this subsection, we verify this conjecture in general. 

Let $\lambda_1,\lambda_2\in X^+$. We define the linear map
$$\Omega_{-\lambda_1,\lambda_2}:{^{\omega}\Lambda_{\lambda_1}}\otimes \Lambda_{\lambda_2}\rightarrow {^{\omega}({^{\omega}\Lambda_{\lambda_2}}\otimes \Lambda_{\lambda_1})}$$
by $u_1\xi_{-\lambda_1}\otimes u_2\eta_{\lambda_2} \mapsto {}^{\omega}(\omega(u_2)\xi_{-\lambda_2}\otimes \omega(u_1)\eta_{\lambda_1})$ for any $u_1,u_2\in \mathbf{U}$. Similarly, we have the linear map $\Omega_{-\lambda_2, \lambda_1}:{^{\omega}\Lambda_{\lambda_2}}\otimes \Lambda_{\lambda_1}\rightarrow {^{\omega}({^{\omega}\Lambda_{\lambda_1}}\otimes \Lambda_{\lambda_2})}$, and it induces the linear map ${^{\omega}\Omega_{-\lambda_2,\lambda_1}}:{^{\omega}({^{\omega}\Lambda_{\lambda_2}}\otimes \Lambda_{\lambda_1})}\rightarrow {^{\omega}\Lambda_{\lambda_1}}\otimes \Lambda_{\lambda_2}$ by ${^{\omega}\Omega_{-\lambda_2,\lambda_1}}({^{\omega}m})={^{\omega}(\Omega_{-\lambda_2,\lambda_1}(m))}$ for any $m\in {^{\omega}\Lambda_{\lambda_2}}\otimes \Lambda_{\lambda_1}$. Then it is easy to check that ${^{\omega}\Omega_{-\lambda_2,\lambda_1}}$ is inverse to $\Omega_{-\lambda_1,\lambda_2}$. Hence $\Omega_{-\lambda_1,\lambda_2}$ is a linear isomorphism.

\begin{lemma}\label{isomorphism Omega}
Let $\lambda_1,\lambda_2\in X^+$. Then $\Omega_{-\lambda_1,\lambda_2}:{^{\omega}\Lambda_{\lambda_1}}\otimes \Lambda_{\lambda_2}\rightarrow {^{\omega}({^{\omega}\Lambda_{\lambda_2}}\otimes \Lambda_{\lambda_1})}$ is a $\mathbf{U}$-module isomorphism such that
$$\Omega_{-\lambda_1,\lambda_2}(b_1^+\xi_{-\lambda_1}\diamondsuit b_2^-\eta_{\lambda_2})={}^{\omega}(b_2^+\xi_{-\lambda_2}\diamondsuit b_1^-\eta_{\lambda_1})\ \textrm{for any}\ b_1\in \mathbf{B}(\lambda_1),b_2\in \mathbf{B}(\lambda_2).$$
\end{lemma}
\begin{proof}
We simply denote $\Omega_{-\lambda_1,\lambda_2}=\Omega$. Consider the diagram
$$\xymatrix{\mathbf{U} \ar[r]^-{\Delta} \ar[d]_-{\omega} &\mathbf{U}\otimes \mathbf{U} \ar[d]^-{\omega\otimes \omega} \\
\mathbf{U} \ar[r]^-{\tau\Delta}  &\mathbf{U}\otimes \mathbf{U},}$$
where $\tau(u_1\otimes u_2)=u_2\otimes u_1$. By the proof of \cite[Proposition 25.1.3]{Lusztig-1993}, the diagram commutes, that is, for any $u\in \mathbf{U}$ with $\Delta(u)=\sum u_1\otimes u_2\in \mathbf{U}\otimes \mathbf{U}$, we have $\Delta\omega(u)=\sum \omega(u_2)\otimes \omega(u_1)$. Then for any $u'_1,u'_2\in \mathbf{U}$, we have 
\begin{align*}
\Omega(u(u'_1\xi_{-\lambda_1}\otimes u'_2\eta_{\lambda_2}))=&\Omega(\sum u_1u'_1\xi_{-\lambda_1}\otimes u_2u'_2\eta_{\lambda_2})\\
=&\sum {}^{\omega}(\omega(u_2)\omega(u'_2)\xi_{-\lambda_2}\otimes \omega(u_1)\omega(u'_1)\eta_{\lambda_1}),\\
u\Omega(u'_1\xi_{-\lambda_1}\otimes u'_2\eta_{\lambda_2})=&u({^{\omega}(\omega(u'_2)\xi_{-\lambda_2}\otimes \omega(u'_1)\eta_{\lambda_1})})\\
=&{^{\omega}((\omega(u))(\omega(u'_2)\xi_{-\lambda_2}\otimes \omega(u'_1)\eta_{\lambda_1}))}\\=
&\sum {}^{\omega}(\omega(u_2)\omega(u'_2)\xi_{-\lambda_2}\otimes \omega(u_1)\omega(u'_1)\eta_{\lambda_1}).
\end{align*}
Hence $\Omega$ is a $\mathbf{U}$-module isomorphism.

Let $\Psi$ and $\Psi'$ denote the involutions on ${^{\omega}\Lambda_{\lambda_1}}\otimes \Lambda_{\lambda_2}$ and ${^{\omega}\Lambda_{\lambda_2}}\otimes \Lambda_{\lambda_1}$ defined by \eqref{Phi-involution} respectively. By definitions, for any $u_1,u_2\in \mathbf{U}$, we have
\begin{align*}
\Psi'({^{\omega}(\Omega(u_1\xi_{-\lambda_1}\otimes u_2\eta_{\lambda_2})}))=&\Psi'(\omega(u_2)\xi_{-\lambda_2}\otimes \omega(u_1)\eta_{\lambda_1})\\
=&\sum_{\nu\in \mathbb{N}[I]}(-v)^{\mathrm{tr}\,\nu}\sum_{b\in \mathbf{B}_{\nu}}b^-\overline{\omega(u_2)\xi_{-\lambda_2}}\otimes b^{*+}\overline{\omega(u_1)\eta_{\lambda_1}}\\
=&\sum_{\nu\in \mathbb{N}[I]}(-v)^{\mathrm{tr}\,\nu}\sum_{b\in \mathbf{B}_{\nu}}b^-\omega(\overline{u_2})\xi_{-\lambda_2}\otimes b^{*+}\omega(\overline{u_1})\eta_{\lambda_1},\\
{^{\omega}(\Omega\Psi(u_1\xi_{-\lambda_1}\otimes u_2\eta_{\lambda_2}))}=&{^{\omega}(\Omega(\sum_{\nu\in \mathbb{N}[I]}(-v)^{\mathrm{tr}\,\nu}\sum_{b\in \mathbf{B}_{\nu}}b^-\overline{u_1\xi_{-\lambda_1}}\otimes b^{*+}\overline{u_2\eta_{\lambda_2}}))}\\
=&\sum_{\nu\in \mathbb{N}[I]}(-v)^{\mathrm{tr}\,\nu}\sum_{b\in \mathbf{B}_{\nu}}b^{*-}\omega(\overline{u_2})\xi_{-\lambda_2}\otimes b^+\omega(\overline{u_1})\eta_{\lambda_1}.
\end{align*}
Notice that both $\{b; b \in \mathbf{B}_{\nu}\}$ and $\{b^*; b \in \mathbf{B}_{\nu}\}$ are bases of $\mathbf{f}_{\nu}$ which are dual to each other. So $\sum_{b\in \mathbf{B}_{\nu}}b^-\omega(\overline{u_2})\xi_{-\lambda_2}\otimes b^{*+}\omega(\overline{u_1})\eta_{\lambda_1}=\sum_{b\in \mathbf{B}_{\nu}}b^{*-}\omega(\overline{u_2})\xi_{-\lambda_2}\otimes b^+\omega(\overline{u_1})\eta_{\lambda_1}$, and $\Psi'({^{\omega}(\Omega(u_1\xi_{-\lambda_1}\otimes u_2\eta_{\lambda_2})}))={^{\omega}(\Omega\Psi(u_1\xi_{-\lambda_1}\otimes u_2\eta_{\lambda_2}))}$. 

Let $b_1\in \mathbf{B}(\lambda_1)$ and $b_2\in \mathbf{B}(\lambda_2)$. By Theorem \ref{24.3.3}, we have $\Psi(b_1^+\xi_{-\lambda_1}\diamondsuit b_2^-\eta_{\lambda_2})=b_1^+\xi_{-\lambda_1}\diamondsuit b_2^-\eta_{\lambda_2}\in \mathbb{Z}[v^{-1}][\mathbf{B}({^{\omega}\Lambda_{\lambda_1}})\otimes \mathbf{B}(\Lambda_{\lambda_2})]$ and $b_1^+\xi_{-\lambda_1}\diamondsuit b_2^-\eta_{\lambda_2}-b_1^+\xi_{-\lambda_1}\otimes b_2^-\eta_{\lambda_2}\in v^{-1}\mathbb{Z}[v^{-1}][\mathbf{B}({^{\omega}\Lambda_{\lambda_1}})\otimes \mathbf{B}(\Lambda_{\lambda_2})]$. Then 
\begin{align*}
&\Psi'({^{\omega}(\Omega(b_1^+\xi_{-\lambda_1}\diamondsuit b_2^-\eta_{\lambda_2})}))={^{\omega}(\Omega(b_1^+\xi_{-\lambda_1}\diamondsuit b_2^-\eta_{\lambda_2})})\in \mathbb{Z}[v^{-1}][\mathbf{B}({^{\omega}\Lambda_{\lambda_2}})\otimes \mathbf{B}(\Lambda_{\lambda_1})];\\
&{^{\omega}(\Omega(b_1^+\xi_{-\lambda_1}\diamondsuit b_2^-\eta_{\lambda_2}))}-b_2^+\xi_{-\lambda_2}\otimes b_1^-\eta_{\lambda_1}\in v^{-1}\mathbb{Z}[v^{-1}][\mathbf{B}({^{\omega}\Lambda_{\lambda_2}})\otimes \mathbf{B}(\Lambda_{\lambda_1})],
\end{align*}
Hence ${^{\omega}(\Omega(b_1^+\xi_{-\lambda_1}\diamondsuit b_2^-\eta_{\lambda_2}))}$ satisfies the characterizing conditions in Theorem \ref{24.3.3}, and so ${^{\omega}(\Omega(b_1^+\xi_{-\lambda_1}\diamondsuit b_2^-\eta_{\lambda_2}))}=b_2^+\xi_{-\lambda_2}\diamondsuit b_1^-\eta_{\lambda_1}$, that is, we have $\Omega(b_1^+\xi_{-\lambda_1}\diamondsuit b_2^-\eta_{\lambda_2})={}^{\omega}(b_2^+\xi_{-\lambda_2}\diamondsuit b_1^-\eta_{\lambda_1})$.
\end{proof}

Now we show that the involution $\omega$ preserves the canonical basis. 

\begin{proposition}\label{prop:omega}
Let $\zeta\in X$ and $b_1,b_2\in \mathbf{B}$. Then $\omega(b_1\diamondsuit_{\zeta} b_2)=b_2\diamondsuit_{-\zeta}\, b_1$.
\end{proposition}
\begin{proof}
By Theorem Theorem \ref{25.2.1}, we have $\omega(b_1\diamondsuit_{\zeta} b_2)\in \omega({_{\mathcal{A}}\dot{\mathbf{U}}1_{\zeta}})={_{\mathcal{A}}\dot{\mathbf{U}}1_{-\zeta}}$. For any $\lambda_1,\lambda_2\in X^+$ with $\zeta=\lambda_2-\lambda_1$ and $b_1\in \mathbf{B}(\lambda_1),b_2\in \mathbf{B}(\lambda_2)$, we have 
\begin{align*}
{}^{\omega}(\omega(b_1 \diamondsuit_{\zeta} b_2)(\xi_{-\lambda_2}\otimes \eta_{\lambda_1}))=&(b_1\diamondsuit_{\zeta} b_2)({^{\omega}(\xi_{-\lambda_2}\otimes \eta_{\lambda_1}}))\\
=&(b_1\diamondsuit_{\zeta} b_2)\Omega_{-\lambda_1,\lambda_2}(\xi_{-\lambda_1}\otimes \eta_{\lambda_2})\\
=&\Omega_{-\lambda_1,\lambda_2}((b_1\diamondsuit_{\zeta} b_2)(\xi_{-\lambda_1}\otimes \eta_{\lambda_2}))\\
=&\Omega_{-\lambda_1,\lambda_2}(b_1^+\xi_{-\lambda_1}\diamondsuit b_2^-\eta_{\lambda_2})\\ 
=& {}^{\omega}(b_2^+\xi_{-\lambda_2}\diamondsuit b_1^-\eta_{\lambda_1}).
\end{align*}
So $\omega(b_1 \diamondsuit_{\zeta} b_2)(\xi_{-\lambda_2}\otimes \eta_{\lambda_1})=b_2^+\xi_{-\lambda_2}\diamondsuit b_1^-\eta_{\lambda_1}$. Hence $\omega(b_1\diamondsuit_{\zeta} b_2)$ satisfies the characterizing conditions in Theorem \ref{25.2.1} (1), and so $\omega(b_1\diamondsuit_{\zeta} b_2)=b_2\diamondsuit_{-\zeta}\, b_1$.
\end{proof}

\subsection{Positivity of multiplication}

Let $J\subset I$ be a spherical subset. Then the spherical parabolic subalgebras $\dot{\mathbf{U}}_J$ and ${^{\omega}\dot{\mathbf{U}}_J}$ have bases $\dot{\mathbf{B}}\cap \dot{\mathbf{U}}_J$ and $\dot{\mathbf{B}}\cap {^{\omega}\dot{\mathbf{U}}_J}$ respectively. A canonical basis element of $\dot{\mathbf{B}}$ is called {\it spherical parabolic}, if it is contained in $\dot{\mathbf{B}}\cap\dot{\mathbf{U}}_J$ or $\dot{\mathbf{B}}\cap{^{\omega}\dot{\mathbf{U}}_J}$ for some spherical subset $J\subset I$.
\begin{theorem}\label{thm:mult}
For any $\dot{b},\dot{b}'\in \dot{\mathbf{B}}$, if one of them is spherical parabolic, then 
$$\dot{b}\dot{b}'\in \mathbb{N}[v,v^{-1}][\dot{\mathbf{B}}].$$
\end{theorem}

\begin{proof}
Let $J$ be a spherical subset of $I$. 

(a) If $\dot{b}\in \dot{\mathbf{U}}_J$, then the desired statement follows from Theorem \ref{thm:structure-multiplication} and the positivity property of $\tilde{\tilde{\mathbf{B}}}$ with respect to the multiplication in Theorem \ref{14.4.13}. 

(b) If $\dot{b'}\in {^{\omega}\dot{\mathbf{U}}_J}$, then $\omega(\dot{b}) \in \dot{\mathbf{U}}_J$. By Proposition \ref{prop:omega} and case (a), we have $\omega(\dot{b}), \omega(\dot{b}') \in \dot{\mathbf{B}}$ and $\omega(\dot{b}) \omega(\dot{b}')\in \mathbb{N}[v,v^{-1}][\dot{\mathbf{B}}]$. Applying $\omega$, by Proposition \ref{prop:omega}, we obtain $\dot{b}\dot{b}'\in \mathbb{N}[v,v^{-1}][\dot{\mathbf{B}}]$.

(c) If $\dot{b}'\in \dot{\mathbf{U}}_J$ or ${^{\omega}\dot{\mathbf{U}}_J}$, then $\sigma(\dot{b}')\in \dot{\mathbf{U}}_J$ or ${^{\omega}\dot{\mathbf{U}}_J}$. By \cite[Theorem 4.3.2]{Kashiwara-1994} and case (a),(b), we have  $\sigma(\dot{b}),\sigma(\dot{b}')\in \dot{\mathbf{B}}$ and  $\sigma(\dot{b}')\sigma(\dot{b})\in \mathbb{N}[v,v^{-1}[\dot{\mathbf{B}}]$. Applying $\sigma$, by \cite[Theorem 4.3.2]{Kashiwara-1994}, we obtain $\dot{b}\dot{b}'\in \mathbb{N}[v,v^{-1}][\dot{\mathbf{B}}]$.
\end{proof}

\section{Positivity of canonical basis of tensor product}\label{Positivity of canonical basis of tensor product}

By using the positivity property of the canonical basis of the modified quantum group, we establish the positivity property of the canonical basis of tensor product.

\subsection{Canonical basis of tensor product sequel}\label{Canonical basis of tensor product sequel}

Let $\lambda_1,\ldots, \lambda_{m+n}\in X^+$ with $m,n\in \mathbb{N}$. We denote 
\begin{align*}
\blambda&=(-\lambda_1,\ldots,-\lambda_m,\lambda_{m+1},\ldots,\lambda_{m+n}),\\ \Lambda(\blambda)&={^{\omega}\Lambda_{\lambda_1}}\otimes \cdots \otimes {^{\omega}\Lambda_{\lambda_m}}\otimes \Lambda_{\lambda_{m+1}}\otimes\cdots\otimes \Lambda_{\lambda_{m+n}}.
\end{align*}

The {\it canonical basis} $\mathbf{B}(\Lambda(\blambda))$ of the tensor product $\Lambda(\blambda)$ was established by Bao and Wang in \cite[Theorem 2.9]{Bao-Wang-2016}. 

By the associativity of tensor products of based $\mathbf{U}$-modules (see \cite[\S 27.3.6]{Lusztig-1993}), for any subsequences $\blambda'$ and $\blambda''$ such that their concatenation is $\blambda$, we have 
\begin{equation}\label{associativity}
(\Lambda(\blambda),\mathbf{B}(\Lambda(\blambda)))=(\Lambda(\blambda')\otimes \Lambda(\blambda''),\mathbf{B}(\Lambda(\blambda')\otimes \Lambda(\blambda''))).    
\end{equation}

\subsection{Realization of tensor product of simple highest weight modules}

In this subsection, we give the thickening realization of the tensor product of simple highest weight modules as a based $\mathbf{U}$-submodule of a simple highest weight modules of a larger quantum group.

Let $\lambda_1,\lambda_2\in X^+$. Recall that the thickening product $\lambda_1\odot\lambda_2\in \tilde{X}$ defined in \eqref{definition of odot} is dominant, and we have the Verma module $M_{\lambda_1\odot\lambda_2}$, the simple highest weight module $\Lambda_{\lambda_1\odot\lambda_2}$ of $\tilde{\mathbf{U}}$ and the natural projection $\pi_{\lambda_1\odot\lambda_2}:M_{\lambda_1\odot\lambda_2}\rightarrow \Lambda_{\lambda_1\odot\lambda_2}$. We regard $M_{\lambda_1\odot\lambda_2}$ and $\Lambda_{\lambda_1\odot\lambda_2}$ as $\mathbf{U}$-modules via the embedding $\mathbf{U}\hookrightarrow \tilde{\mathbf{U}}$, and define the $\mathbf{U}$-module $\Lambda_{\lambda_1,\lambda_2}$ to be the image of $(\mathbf{f}\theta_{\lambda_2}\mathbf{f})_{\lambda_1\odot\lambda_2}\subset M_{\lambda_1\odot\lambda_2}$ under the natural projection $\pi_{\lambda_1\odot\lambda_2}:M_{\lambda_1\odot\lambda_2}\rightarrow \Lambda_{\lambda_1\odot\lambda_2}$.

By \cite[Corollary 4.5]{Fang-Lan-2025}, $\mathbf{B}(\Lambda_{\lambda_1,\lambda_2})=:\{b^-\eta_{\lambda_1\odot\lambda_2};b\in \mathbf{B}(\mathbf{f}\theta_{\lambda_2}\mathbf{f})\}\setminus \{0\}$ is a basis of $\Lambda_{\lambda_1,\lambda_2}$, and $(\Lambda_{\lambda_1,\lambda_2},\mathbf{B}(\Lambda_{\lambda_1,\lambda_2}))$ is a based $\mathbf{U}$-submodule of $(\Lambda_{\lambda_1\odot\lambda_2},\mathbf{B}(\Lambda_{\lambda_1\odot\lambda_2}))$.

The following result, first established in \cite[Corollary 7.21]{Li-2014} and \cite[Theorem 3.8 \& Theorem 4.7]{Fang-Lan-2025}, realizes the tensor product of two simple highest weight modules as a based $\mathbf{U}$-submodule of a simple highest weight module of $\tilde{\mathbf{U}}$. Our proof here is different and simpler.

\begin{lemma}\label{morphism underlinepsi of case 2}
Let $\lambda_1,\lambda_2\in X^+$. Then there are based $\mathbf{U}$-module isomorphisms
$$\xymatrix@C=2cm{(\Lambda_{\lambda_1,\lambda_2}, \mathbf{B}(\Lambda_{\lambda_1,\lambda_2})) \ar@<.5ex>[r]^-{\underline{\varphi_{\lambda_1,\lambda_2}}} &(\Lambda_{\lambda_1}\otimes \Lambda_{\lambda_2}, \mathbf{B}(\Lambda_{\lambda_1}\otimes \Lambda_{\lambda_2})).\ar@<.5ex>[l]^-{\underline{\psi_{\lambda_1,\lambda_2}}}}$$
\end{lemma}
\begin{proof}
By \cite[Lemma 3.2 \& Definition 3.3]{Fang-Lan-2025}, we can identify 
\begin{align*}
\Lambda_{\lambda_1,\lambda_2}\cong (\mathbf{f}\theta_{\lambda_2}\mathbf{f})_{\lambda_1\odot\lambda_2}/((\mathbf{f}\theta_{\lambda_2}\mathbf{f})_{\lambda_1\odot\lambda_2}\cap \sum_{i\in I}\tilde{\mathbf{f}}\theta_i^{\langle i,\lambda_1\rangle+1}).
\end{align*}
Let $\pi_{\lambda_1,\lambda_2}:(\mathbf{f}\theta_{\lambda_2}\mathbf{f})_{\lambda_1\odot\lambda_2}\rightarrow \Lambda_{\lambda_1,\lambda_2}$ be the natural projection. By \cite[Corollary 4.5]{Fang-Lan-2025}, $\pi_{\lambda_1,\lambda_2}$ is a based $\mathbf{U}$-module homomorphism. By Lemma \ref{isomorphism varphi}, there are $\mathbf{U}$-module isomorphisms $\varphi_{\lambda_1,\lambda_2}$ and $\psi_{\lambda_1,\lambda_2}$. We simply denote $\pi_{\lambda_1,\lambda_2}=\pi,\varphi_{\lambda_1,\lambda_2}=\varphi,\psi_{\lambda_1,\lambda_2}=\psi$. The proof is based on the following diagram 
$$\xymatrix{\mathrm{ker}\,\pi \ar@{^{(}->}[r] \ar@<.5ex>[d]^{\varphi} &(\mathbf{f}\theta_{\lambda_2}\mathbf{f})_{\lambda_1\odot\lambda_2} \ar@<.5ex>[d]^{\varphi} \ar@{->>}[r]^-{\pi} &\Lambda_{\lambda_1,\lambda_2} \ar@<.5ex>[d]^{\underline{\varphi}}\\
\mathrm{ker}\,(\pi_{\lambda_1}\otimes \mathrm{Id}) \ar@{^{(}->}[r] \ar@<.5ex>[u]^{\psi} &M_{\lambda_1}\otimes \Lambda_{\lambda_2} \ar@{->>}[r]^-{\pi_{\lambda_1}\otimes \mathrm{Id}}  \ar@<.5ex>[u]^{\psi} &\Lambda_{\lambda_1}\otimes \Lambda_{\lambda_2}. \ar@<.5ex>[u]^{\underline{\psi}}}$$

We show by induction on $\mathrm{tr}\,|y|\in \mathbb{N}$ that $\pi\psi(x^-v_{\lambda_1}\otimes y^-\eta_{\lambda_2})=0$ for any $x\in \sum_{i\in I}\mathbf{f}\theta_i^{\langle i,\lambda_1\rangle+1}$ and homogeneous $y\in \mathbf{f}$. If $\mathrm{tr}\,|y|=0$, then $y\in \mathbb{Q}(v)$. By \eqref{im psi},we have $\pi\psi(x^-v_{\lambda_1}\otimes\eta_{\lambda_2})=\pi(\theta_{\lambda_2}^-x^-v_{\lambda_1\odot\lambda_2})=0$. If $\mathrm{tr}\,|y|>0$, then it is enough to prove the statement for any $y=\theta_iy'$, with $i\in I$ and homogeneous $y'\in \mathbf{f}$. Note that $\mathrm{tr}\,|y'|<\mathrm{tr}\,|y|$. Since $\pi\psi$ commutes with the actions of $F_i$, by
\begin{align*}
F_i(x^-v_{\lambda_1}\otimes y'^-\eta_{\lambda_2})=&x^-v_{\lambda_1}\otimes F_iy'^-\eta_{\lambda_2}+F_ix^-v_{\lambda_1}\otimes K_{-i}y'^-\eta_{\lambda_2}\\
=&x^-v_{\lambda_1}\otimes y^-\eta_{\lambda_2}+(\theta_ix)^-v_{\lambda_1}\otimes v^{\langle -i,\lambda_2-|y'|\rangle}y'^-\eta_{\lambda_2},
\end{align*}
$\theta_ix\in \sum_{i\in I}\mathbf{f}\theta_i^{\langle i,\lambda_1\rangle+1}$ and the inductive hypothesis, we have $\pi\psi(x^-v_{\lambda_1}\otimes y^-\eta_{\lambda_2})=0$, as desired. Hence $\pi\psi(\mathrm{ker}\,(\pi_{\lambda_1}\otimes \mathrm{Id}))=0$,
and then $\pi\psi$ induces the $\mathbf{U}$-module homomorphism $\underline{\psi}$. By the proof of \cite[Theorem 3.8]{Fang-Lan-2025}, the $\mathbf{U}$-module homomorphism $(\pi_{\lambda_1}\otimes \mathrm{Id})\varphi$ induces the $\mathbf{U}$-module homomorphism $\underline{\varphi}$. From the commutative diagram, $\underline{\varphi}$ and $\underline{\psi}$ are the inverses of each other, and both of them are $\mathbf{U}$-module isomorphisms.

By \cite[Theorem 14.4.11]{Lusztig-1993}, $\pi_{\lambda_1}:M_{\lambda_1}\rightarrow \Lambda_{\lambda_1}$ is a based $\mathbf{U}$-module homomorphism, and then so is $\pi_{\lambda_1}\otimes \mathrm{Id}:M_{\lambda_1}\otimes \Lambda_{\lambda_2}\rightarrow \Lambda_{\lambda_1}\otimes \Lambda_{\lambda_2}$. Combining with Theorem \ref{positivity of transition matrix}, both $\underline{\varphi}$ and $\underline{\psi}$ are based $\mathbf{U}$-module isomorphisms.
\end{proof}

Recall the thickening data $(\tilde{I},\cdot)$ and $(\tilde{Y},\tilde{X},\langle\,,\,\rangle,\ldots)$ of $(I,\cdot)$ and $\langle Y,X,\langle\,,\,\rangle,\ldots)$ defined in \S \ref{The thickening product}. For $n\geqslant 1$, we inductively define $(\tilde{I}^n,\cdot)$ and $(\tilde{Y}^n,\tilde{X}^n,\langle\,,\,\rangle,\ldots)$ to be the thickening data of $(\tilde{I}^{n-1},\cdot)$ and $(\tilde{Y}^{n-1},\tilde{X}^{n-1},\langle\,,\,\rangle,\ldots)$, where $(\tilde{I}^0,\cdot)=(I,\cdot)$ and $(\tilde{Y}^0,\tilde{X}^0,\langle\,,\,\rangle,\ldots)=(\tilde{Y},\tilde{X},\langle\,,\,\rangle,\ldots)$. Then $(\tilde{I}^n,\cdot)$ is a symmetric Cartan datum and $(\tilde{Y}^n,\tilde{X}^n,\langle\,,\,\rangle,\ldots)$ is a $\tilde{Y}^n$-regular root datum of type $(\tilde{I}^n,\cdot)$. Let $\tilde{\mathbf{f}}^n$ and $\tilde{\mathbf{U}}^n$ be the algebras defined in \S \ref{The algebra f} and \S \ref{Quantized enveloping algebra} associated with $(\tilde{I}^n,\cdot)$ and $(\tilde{Y}^n,\tilde{X}^n,\langle\,,\,\rangle,\ldots)$ respectively. Then there are natural subalgebra embeddings $\mathbf{f}\hookrightarrow\tilde{\mathbf{f}}\hookrightarrow\cdots\hookrightarrow \tilde{\mathbf{f}}^n$ and $\mathbf{U}\hookrightarrow\tilde{\mathbf{U}}\hookrightarrow\cdots\hookrightarrow \tilde{\mathbf{U}}^n$.

For any $\lambda_1,\ldots,\lambda_n\in X^+$, we inductively define $\lambda_1\odot\cdots\odot\lambda_n\in (\tilde{X}^{n-1})^+$ to be the thickening product of $\lambda_1\odot\cdots\odot\lambda_{n-1}\in (\tilde{X}^{n-2})^+$ and $0\odot\cdots\odot 0\odot\lambda_{n}\in (\tilde{X}^{n-2})^+$.

\begin{proposition}\label{thickening realization of tensor product of simple based modules}
Let $n\geqslant 2,\lambda_1,\ldots,\lambda_n\in X^+$ and $\mathbf{B}(\Lambda(\blambda))$ be the canonical basis of the tensor product $\Lambda(\blambda)=\Lambda_{\lambda_1}\otimes\cdots\otimes \Lambda_{\lambda_n}$. Then there exist a based $\mathbf{U}$-submodule $(\Lambda_{\blambda},\mathbf{B}(\Lambda_{\blambda}))$ of the based $\tilde{\mathbf{U}}^{n-1}$-module $(\Lambda_{\lambda_1\odot\cdots\odot\lambda_n},\mathbf{B}(\Lambda_{\lambda_1\odot\cdots\odot\lambda_n}))$, and based $\mathbf{U}$-module isomorphisms 
$$\xymatrix@C=2cm{(\Lambda_{\blambda},\mathbf{B}(\Lambda_{\blambda})) \ar@<.5ex>[r]^-{\underline{\varphi_{\blambda}}} &(\Lambda(\blambda),\mathbf{B}(\Lambda(\blambda))). \ar@<.5ex>[l]^-{\underline{\psi_{\blambda}}}}$$
\end{proposition}
\begin{proof}
We prove by induction on $n$. If $n=2$, it is Lemma \ref{morphism underlinepsi of case 2}. If $n\geqslant 3$, then we set $\blambda'=(\lambda_1,\ldots,\lambda_{n-1})$, $\blambda''=(0,\ldots,0, \lambda_{n})$ and let $\tilde \l'=\lambda_1\odot\cdots\odot\lambda_{n-1},\tilde \l''=0 \odot \cdots \odot 0 \odot \l_n\in (\tilde{X}^{n-2})^+$. By the inductive hypothesis, there exist a based $\mathbf{U}$-submodule $(\Lambda_{\blambda'},\mathbf{B}(\Lambda_{\blambda'}))$ of the based $\tilde{\mathbf{U}}^{n-2}$-module $(\Lambda_{\tilde \l'},\mathbf{B}(\Lambda_{\tilde \l'}))$, and based $\mathbf{U}$-module isomorphisms
$\xymatrix@C=1cm{(\Lambda_{\blambda'},\mathbf{B}(\Lambda_{\blambda'})) \ar@<.5ex>[r]^-{\underline{\varphi_{\blambda'}}} &(\Lambda(\blambda'),\mathbf{B}(\Lambda(\blambda'))) \ar@<.5ex>[l]^-{\underline{\psi_{\blambda'}}}}$. Similarly, there exist a based $\mathbf{U}$-submodule $(\Lambda_{\blambda''},\mathbf{B}(\Lambda_{\blambda''}))$ of the based $\tilde{\mathbf{U}}^{n-2}$-module $(\Lambda_{\tilde \l''},\mathbf{B}(\Lambda_{\tilde \l''}))$, and based $\mathbf{U}$-module isomorphisms $
\xymatrix@C=1cm{(\Lambda_{\blambda''},\mathbf{B}(\Lambda_{\blambda''})) \ar@<.5ex>[r]^-{\underline{\varphi_{\blambda''}}} &(\Lambda(\blambda''),\mathbf{B}(\Lambda(\blambda''))) \ar@<.5ex>[l]^-{\underline{\psi_{\blambda''}}}}$.

Note that the simple $\mathbf{U}$-module $\Lambda_0$ with the highest weight $0$ is $\mathbb{Q}(v)$ such that $F_i,E_i$ acts by $0$ and $K_{\mu}$ acts by $1$ for any $i\in I$ and $\mu\in Y$, and so $(\Lambda(\blambda''),\mathbf{B}(\Lambda(\blambda'')))=(\Lambda_{\lambda_n},\mathbf{B}(\Lambda_{\lambda_n}))$. Taking the tensor product of the based modules, we obtain a based $\mathbf{U}$-submodule $(\Lambda_{\blambda'}\otimes \Lambda_{\blambda''},\mathbf{B}(\Lambda_{\blambda'}\otimes \Lambda_{\blambda''}))$ of the based $\tilde{\mathbf{U}}^{n-2}$-module $(\Lambda_{\tilde \l'}\otimes \Lambda_{\tilde \l''},\mathbf{B}(\Lambda_{\tilde \l'}\otimes \Lambda_{\tilde \l''}))$, and based $\mathbf{U}$-module isomorphisms $$\xymatrix@C=2cm{(\Lambda_{\blambda'}\otimes \Lambda_{\blambda''},\mathbf{B}(\Lambda_{\blambda'}\otimes \Lambda_{\blambda''})) \ar@<.5ex>[r]^-{\underline{\varphi_{\blambda'}}\otimes \underline{\varphi_{\blambda''}}} &(\Lambda(\blambda),\mathbf{B}(\Lambda(\blambda))). \ar@<.5ex>[l]^-{\underline{\psi_{\blambda'}}\otimes \underline{\psi_{\blambda''}}}}$$

Applying Lemma \ref{morphism underlinepsi of case 2} for the tensor product of $\tilde{\mathbf{U}}^{n-2}$-modules, there exist based $\tilde{\mathbf{U}}^{n-2}$-modules isomorphisms
\begin{equation*}
\xymatrix@C=2cm{(\Lambda_{\tilde \lambda', \tilde \lambda''},\mathbf{B}(\Lambda_{\tilde \lambda', \tilde \lambda''})) \ar@<.5ex>[r]^-{\underline{\varphi_{\tilde \lambda', \tilde \lambda''}}} &(\Lambda_{\tilde \l'}\otimes \Lambda_{\tilde \l''},\mathbf{B}(\Lambda_{\tilde \l'}\otimes \Lambda_{\tilde \l''})), \ar@<.5ex>[l]^-{\underline{\psi_{\tilde \lambda', \tilde \lambda''}}}}
\end{equation*}
where $(\Lambda_{\tilde \lambda', \tilde \lambda''},\mathbf{B}(\Lambda_{\tilde \lambda', \tilde \lambda''})$ is a based $\tilde{\mathbf{U}}^{n-2}$-submodule of the based $\tilde{\mathbf{U}}^{n-1}$-module $(\Lambda_{\lambda_1\odot\cdots\odot\lambda_n},\mathbf{B}(\Lambda_{\lambda_1\odot\cdots\odot\lambda_n}))$. Let $(\Lambda_{\blambda},\mathbf{B}(\Lambda_{\blambda}))=\underline{\psi_{\tilde \l', \tilde \l''}}((\Lambda_{\blambda'}\otimes \Lambda_{\blambda''},\mathbf{B}(\Lambda_{\blambda'}\otimes \Lambda_{\blambda''})))$. Then it is a based $\mathbf{U}$-submodule of $(\Lambda_{\lambda_1\odot\cdots\odot\lambda_n},\mathbf{B}(\Lambda_{\lambda_1\odot\cdots\odot\lambda_n}))$, and 
\begin{align*}
\underline{\varphi_{\blambda}}=(\underline{\varphi_{\blambda'}}\otimes \underline{\varphi_{\blambda''}})(\underline{\varphi_{\tilde \l',\tilde \l''}}|_{\Lambda_{\blambda}}), \quad
\underline{\psi_{\blambda}} =\underline{\psi_{\tilde \l',\tilde \l''}}(\underline{\psi_{\blambda'}}\otimes \underline{\psi_{\blambda''}})
\end{align*}
are the desired based $\mathbf{U}$-module isomorphisms.
\end{proof}

\subsection{Positivity of canonical basis of tensor product}

\begin{theorem}\label{Positivity in general case}
Let $\lambda_1,\ldots,\lambda_{m},\lambda_{m+1},\ldots,\lambda_{m+n}\in X^+$ with $m,n\in \mathbb{N}$. The canonical basis $\mathbf{B}(\Lambda(\blambda))$ of the tensor product $\Lambda(\blambda)$ has the following positivity property: 
\begin{enumerate}
\item (Positivity of transition matrix) 
$$\mathbf{B}(\Lambda(\blambda))\subset \mathbb{N}[v^{-1}][\mathbf{B}({^{\omega}\Lambda_{\lambda_1}})\otimes\cdots\otimes \mathbf{B}({^{\omega}\Lambda_{\lambda_m}})\otimes \mathbf{B}(\Lambda_{\lambda_{m+1}})\otimes\cdots\otimes \mathbf{B}(\Lambda_{\lambda_{m+n}})].$$
\item (Positivity of action) Let $\dot{b}\in \dot{\mathbf{B}}$. If $\dot{b}$ is spherical parabolic or $\Lambda(\blambda)$ is the tensor product of simple highest (resp. lowest) weight modules, that is, $m=0$ (resp. $n=0$), then 
\begin{align*}
\dot{b}(\mathbf{B}(\Lambda(\blambda)))\subset \mathbb{N}[v,v^{-1}][\mathbf{B}(\Lambda(\blambda))].
\end{align*}
\end{enumerate}
\end{theorem}
\begin{proof}
Without loss of generality, we may assume that $m=n$ by taking the tensor products ${{^\omega}\Lambda_{0}}\otimes\ldots\otimes {{^\omega}\Lambda_{0}}\otimes \Lambda(\blambda)$ or $\Lambda(\blambda)\otimes \Lambda_0\otimes\ldots\otimes \Lambda_0$ if necessary. We denote $\blambda'=(-\lambda_1,\ldots,-\lambda_n)$ and $\blambda''=(\lambda_{n+1},\ldots,\lambda_{2n})$ so that $\blambda=\blambda'\blambda''$. If $n\geqslant 2$, by Proposition \ref{thickening realization of tensor product of simple based modules}, there exist a based $\mathbf{U}$-submodule $(\Lambda_{\blambda''},\mathbf{B}(\Lambda_{\blambda''}))$ of the based $\tilde{\mathbf{U}}^{n-1}$-module $(\Lambda_{\lambda_{n+1}\odot\cdots\odot\lambda_{2n}},\mathbf{B}(\Lambda_{\lambda_{n+1}\odot\cdots\odot\lambda_{2n}}))$, and a based $\mathbf{U}$-module isomorphism 
\begin{equation}\label{larger highest}
(\Lambda_{\blambda''},\mathbf{B}(\Lambda_{\blambda''}))\cong (\Lambda(\blambda''),\mathbf{B}(\Lambda(\blambda''))).   
\end{equation}
Similarly, there exist a based $\mathbf{U}$-submodule $({^{\omega}\Lambda_{\blambda'}},\mathbf{B}({^{\omega}\Lambda_{\blambda'}}))$ of the based $\tilde{\mathbf{U}}^{n-1}$-module $({^{\omega}\Lambda_{\lambda_{1}\odot\cdots\odot\lambda_{n}}},\!\mathbf{B}({^{\omega}\Lambda_{\lambda_{1}\odot\cdots\odot\lambda_{n}}}))$, and a based $\mathbf{U}$-module isomorphism 
\begin{equation}\label{larger lowest}
({^{\omega}\Lambda_{\blambda'}},\mathbf{B}({^{\omega}\Lambda_{\blambda'}}))\cong (\Lambda(\blambda'),\mathbf{B}(\Lambda(\blambda'))).
\end{equation}

Taking the tensor product of based modules, we obtain a based $\mathbf{U}$-submodule $({^{\omega}\Lambda_{\blambda'}}\otimes \Lambda_{\blambda''},\mathbf{B}({^{\omega}\Lambda_{\blambda'}}\otimes \Lambda_{\blambda''}))$ of the based $\tilde{\mathbf{U}}^{n-1}$-module 
\begin{equation}\label{larger lowest tensor larger highest}
({^{\omega}\Lambda_{\lambda_{1}\odot\ldots\odot \lambda_{n}}}\otimes \Lambda_{\lambda_{n+1}\odot\cdots\odot \lambda_{2 n}},\mathbf{B}({^{\omega}\Lambda_{\lambda_{1}\odot\cdots\odot \lambda_{n}}}\otimes \Lambda_{\lambda_{n+1}\odot\ldots\odot \lambda_{2 n}})),
\end{equation}
and based $\mathbf{U}$-module isomorphisms
\begin{equation}\label{thickening realization of based modules in general case}
({^{\omega}\Lambda_{\blambda'}}\otimes \Lambda_{\blambda''},\mathbf{B}({^{\omega}\Lambda_{\blambda'}}\otimes \Lambda_{\blambda''}))\cong (\Lambda(\blambda),\mathbf{B}(\Lambda(\blambda)))
\end{equation}

(1) If $n=1$, the desired statement is Theorem \ref{thm:trans-pos}. If $n\geqslant 2$, applying Theorem \ref{thm:trans-pos} for the tensor product of $\tilde{\mathbf{U}}^{n-1}$-modules in \eqref{larger lowest tensor larger highest}, we have 
$$\mathbf{B}({^{\omega}\Lambda_{\lambda_{1}\odot\ldots\odot \lambda_{n}}}\otimes \Lambda_{\lambda_{n+1}\odot\cdots\odot \lambda_{2 n}})\subset \mathbb{N}[v^{-1}][\mathbf{B}({^{\omega}\Lambda_{\lambda_{1}\odot\ldots\odot \lambda_{n}}})\otimes \mathbf{B}(\Lambda_{\lambda_{n+1}\odot\cdots\odot \lambda_{2 n}})],$$ 
and then $\mathbf{B}({^{\omega}\Lambda_{\blambda'}}\otimes \Lambda_{\blambda''})\subset \mathbb{N}[v^{-1}][\mathbf{B}({^{\omega}\Lambda_{\blambda'}})\otimes \mathbf{B}(\Lambda_{\blambda''})]$. By the based $\mathbf{U}$-module isomorphisms in \eqref{larger highest}, \eqref{larger lowest} and \eqref{thickening realization of based modules in general case}, $\mathbf{B}(\Lambda(\blambda))\subset \mathbb{N}[v^{-1}][\mathbf{B}(\Lambda(\blambda'))\otimes \mathbf{B}(\Lambda(\blambda''))]$. 

By Proposition \ref{thickening realization of tensor product of simple based modules}, and applying \cite[Corollary 4.8]{Fang-Lan-2025} for the tensor product of $\tilde{\mathbf{U}}^{n-1}$-modules inductively, we have $\mathbf{B}(\Lambda(\blambda''))\subset \mathbb{N}[v^{-1}][\mathbf{B}(\Lambda_{\lambda_{n+1}})\otimes \cdots \otimes\mathbf{B}(\Lambda_{\lambda_{2n}})]$. Similarly, $\mathbf{B}(\Lambda(\blambda'))\subset \mathbb{N}[v^{-1}][\mathbf{B}({^{\omega}\Lambda_{\lambda_1}})\otimes\cdots\otimes \mathbf{B}({^{\omega}\Lambda_{\lambda_n}})]$. Therefore, we have $$\mathbf{B}(\Lambda(\blambda))\subset \mathbb{N}[v^{-1}][\mathbf{B}({^{\omega}\Lambda_{\lambda_1}})\otimes\cdots\otimes \mathbf{B}({^{\omega}\Lambda_{\lambda_m}})\otimes \mathbf{B}(\Lambda_{\lambda_{m+1}})\otimes\cdots\otimes \mathbf{B}(\Lambda_{\lambda_{m+n}})].$$

(2) We first prove the case that $\dot{b}\in \dot{\mathbf{B}}$ is spherical parabolic. If $n=1$, for any $b_1^+\xi_{-\lambda_1}\diamondsuit b_2^-\eta_{\lambda_2}\in \mathbf{B}({^{\omega}\Lambda_{\lambda_1}}\otimes \Lambda_{\lambda_2})$, by Theorem \ref{thm:mult} and Theorem \ref{25.2.1}, we have $\dot{b}(b_1\diamondsuit_{\lambda_2-\lambda_1}b_2)\in \mathbb{N}[v,v^{-1}][\dot{\mathbf{B}}]$ and 
$$\dot{b}(b_1^+\xi_{-\lambda_1}\diamondsuit b_2^-\eta_{\lambda_2})=\dot{b}(b_1\diamondsuit_{\lambda_2-\lambda_1}b_2)(\xi_{-\lambda_1}\otimes \eta_{\lambda_2})\in \mathbb{N}[v,v^{-1}][\mathbf{B}({^{\omega}\Lambda_{\lambda_1}}\otimes \Lambda_{\lambda_2})].$$
Hence 

(a) $\dot{b}(\mathbf{B}({^{\omega}\Lambda_{\lambda_1}}\otimes \Lambda_{\lambda_2}))\subset \mathbb{N}[v,v^{-1}][\mathbf{B}({^{\omega}\Lambda_{\lambda_1}}\otimes \Lambda_{\lambda_2})]$.

If $n\geqslant 2$, let $\dot{\tilde{\mathbf{U}}}^{n-1}$ be the modified quantum group corresponding to $\tilde{\mathbf{U}}^{n-1}$ associated with $(\tilde{Y}^{n-1},\tilde{X}^{n-1},\langle\,,\,\rangle,\ldots)$, and $\dot{\tilde{\mathbf{B}}}^{n-1}$ be its canonical basis. Then there are natural embeddings $\dot{\mathbf{U}}\subset \dot{\tilde{\mathbf{U}}}^{n-1}$ and $\dot{\mathbf{B}}\subset \dot{\tilde{\mathbf{B}}}^{n-1}$. Applying (a) for $\dot{b}\in \dot{\tilde{\mathbf{B}}}^{n-1}$ and the tensor product of $\tilde{\mathbf{U}}^{n-1}$-modules in \eqref{larger lowest tensor larger highest}, we obtain 
$$\dot{b}(\mathbf{B}({^{\omega}\Lambda_{\lambda_{1}\odot\cdots\odot \lambda_{n}}}\otimes \Lambda_{\lambda_{n+1}\odot\ldots\odot \lambda_{2 n}}))\subset \mathbb{N}[v,v^{-1}][\mathbf{B}({^{\omega}\Lambda_{\lambda_{1}\odot\cdots\odot \lambda_{n}}}\otimes \Lambda_{\lambda_{n+1}\odot\ldots\odot \lambda_{2 n}})],$$
and so $\dot{b}(\mathbf{B}({^{\omega}\Lambda_{\blambda'}}\otimes \Lambda_{\blambda''}))\subset \mathbb{N}[v,v^{-1}][\mathbf{B}({^{\omega}\Lambda_{\blambda'}}\otimes \Lambda_{\blambda''})]$. By the based $\mathbf{U}$-module isomorphism in \eqref{thickening realization of based modules in general case}, we have $\dot{b}(\mathbf{B}(\Lambda(\blambda)))\subset \mathbb{N}[v,v^{-1}][\mathbf{B}(\Lambda(\blambda))]$.

Next we prove the case that $\Lambda(\blambda)$ is the tensor product of simple highest weight modules. The case that $\Lambda(\blambda)$ is the tensor product of simple lowest weight modules can be proved in a similar way. Let $\dot{b}\in \dot{\mathbf{B}}$ and $b^-\eta_{\lambda_1}\in \mathbf{B}(\Lambda_{\lambda_1})$. Notice that $1\diamondsuit_{\lambda}b\in \dot{\mathbf{B}}$ is spherical parabolic. By Theorem \ref{thm:mult} and Theorem \ref{25.2.1}, we have $\dot{b}(1\diamondsuit_{\lambda}b)\in \mathbb{N}[v,v^{-1}][\dot{\mathbf{B}}]$ and
$$\dot{b}(b^-\eta_{\lambda_1})=\dot{b}b^-1_{\lambda_1}\eta_{\lambda_1}=\dot{b}(1\diamondsuit_{\lambda}b)\eta_{\lambda_1}\in \mathbb{N}[v,v^{-1}][\mathbf{B}(\Lambda_{\lambda_1})].$$
Hence 

(b) $\dot{b}(\mathbf{B}(\Lambda_{\lambda_1}))\subset \mathbb{N}[v,v^{-1}][\mathbf{B}(\Lambda_{\lambda_1})].$

If $n\geqslant 2$, applying (b) for $\dot{b}\in \dot{\tilde{\mathbf{B}}}^{n-1}$ and the $\tilde{\mathbf{U}}^{n-1}$-modules $\Lambda_{\lambda_1\odot\cdots\odot\lambda_n}$, we obtain 
$$\dot{b}(\mathbf{B}(\Lambda_{\lambda_1\odot\cdots\odot\lambda_n}))\subset \mathbb{N}[v,v^{-1}][\mathbf{B}(\Lambda_{\lambda_1\odot\cdots\odot\lambda_n})],$$
and so $\dot{b}(\mathbf{B}(\Lambda_{\lambda_1,\ldots,\lambda_n}))\subset \mathbb{N}[v,v^{-1}][\mathbf{B}(\Lambda_{\lambda_1,\ldots,\lambda_n})]$. By the based $\mathbf{U}$-module isomorphisms in Proposition \ref{thickening realization of tensor product of simple based modules}, we have $\dot{b}(\mathbf{B}(\Lambda(\blambda)))\subset \mathbb{N}[v,v^{-1}][\mathbf{B}(\Lambda(\blambda))]$.
\end{proof}

\printbibliography

@book{Lusztig-1993,
	author = {Lusztig, George},
	isbn = {0-8176-3712-5},
	mrclass = {17B37 (16W30 17-02 17B35 81R50)},
	mrnumber = {1227098},
	mrreviewer = {Jie Du},
	pages = {xii+341},
	publisher = {Birkh\"{a}user Boston, Inc., Boston, MA},
	series = {Progress in Mathematics},
	title = {Introduction to quantum groups},
	url = {https://mathscinet.ams.org/mathscinet-getitem?mr=1227098},
	volume = {110},
	year = {1993}}

@article{Lusztig-1991,
	author = {Lusztig, G.},
	doi = {10.2307/2939279},
	fjournal = {Journal of the American Mathematical Society},
	issn = {0894-0347},
	journal = {J. Amer. Math. Soc.},
	mrclass = {17B37 (17B67 20G05)},
	mrnumber = {1088333},
	mrreviewer = {H. H. Andersen},
	number = {2},
	pages = {365--421},
	title = {Quivers, perverse sheaves, and quantized enveloping algebras},
	url = {https://mathscinet.ams.org/mathscinet-getitem?mr=1088333},
	volume = {4},
	year = {1991}}

@incollection {Lusztig-1994,
    AUTHOR = {Lusztig, G.},
     TITLE = {Total positivity in reductive groups},
 BOOKTITLE = {Lie theory and geometry},
    SERIES = {Progr. Math.},
    VOLUME = {123},
     PAGES = {531--568},
 PUBLISHER = {Birkh\"{a}user Boston, Boston, MA},
      YEAR = {1994},
   MRCLASS = {20G20},
  MRNUMBER = {1327548},
       DOI = {10.1007/978-1-4612-0261-5\_20},
       URL = {https://doi.org/10.1007/978-1-4612-0261-5_20},
}

@article {Lusztig-2019,
    AUTHOR = {Lusztig, G.},
     TITLE = {Total positivity in reductive groups, {II}},
   JOURNAL = {Bull. Inst. Math. Acad. Sin. (N.S.)},
  FJOURNAL = {Bulletin of the Institute of Mathematics. Academia Sinica. New
              Series},
    VOLUME = {14},
      YEAR = {2019},
    NUMBER = {4},
     PAGES = {403--459},
      ISSN = {2304-7909},
   MRCLASS = {20G20},
  MRNUMBER = {4054343},
MRREVIEWER = {Anthony Henderson},
       DOI = {10.21915/bimas.2019402},
       URL = {https://doi.org/10.21915/bimas.2019402},
}

@article {Bao-He-2021,
    AUTHOR = {Bao, Huanchen and He, Xuhua},
     TITLE = {Flag manifolds over semifields},
   JOURNAL = {Algebra Number Theory},
  FJOURNAL = {Algebra \& Number Theory},
    VOLUME = {15},
      YEAR = {2021},
    NUMBER = {8},
     PAGES = {2037--2069},
      ISSN = {1937-0652},
   MRCLASS = {14M15 (15B48 20G44)},
  MRNUMBER = {4337460},
MRREVIEWER = {Francesco Esposito},
       DOI = {10.2140/ant.2021.15.2037},
       URL = {https://doi.org/10.2140/ant.2021.15.2037},
}

@misc{Bao-He-2022,
      title={Total positivity in twisted product of flag varieties}, 
      author={Bao, Huanchen and He, Xuhua},
      year={2022},
      eprint={2211.11168},
      archivePrefix={arXiv},
      primaryClass={math.QA},
      url={https://arxiv.org/abs/2211.11168}, 
}

@article {Bao-Wang-2016,
    AUTHOR = {Bao, Huanchen and Wang, Weiqiang},
     TITLE = {Canonical bases in tensor products revisited},
   JOURNAL = {Amer. J. Math.},
  FJOURNAL = {American Journal of Mathematics},
    VOLUME = {138},
      YEAR = {2016},
    NUMBER = {6},
     PAGES = {1731--1738},
      ISSN = {0002-9327,1080-6377},
   MRCLASS = {17B37},
  MRNUMBER = {3595499},
MRREVIEWER = {Alexandr\ Nikolaevich\ Zubkov},
       DOI = {10.1353/ajm.2016.0051},
       URL = {https://doi.org/10.1353/ajm.2016.0051},
}

@article {Lusztig-1992,
    AUTHOR = {Lusztig, G.},
     TITLE = {Canonical bases in tensor products},
   JOURNAL = {Proc. Nat. Acad. Sci. U.S.A.},
  FJOURNAL = {Proceedings of the National Academy of Sciences of the United
              States of America},
    VOLUME = {89},
      YEAR = {1992},
    NUMBER = {17},
     PAGES = {8177--8179},
      ISSN = {0027-8424},
   MRCLASS = {17B37 (17B10)},
  MRNUMBER = {1180036},
MRREVIEWER = {Chon\ Hu\ Cheng},
       DOI = {10.1073/pnas.89.17.8177},
       URL = {https://doi.org/10.1073/pnas.89.17.8177},
}

@misc{Fang-Lan-Xiao-2023,
      title={Lusztig sheaves and integrable highest weight modules}, 
      author={Jiepeng Fang and Yixin Lan and Jie Xiao},
      year={2023},
      eprint={2307.16131},
      archivePrefix={arXiv},
      primaryClass={math.RT},
      url={https://arxiv.org/abs/2307.16131}, 
}

@misc{Fang-Lan-2023,
      title={Lusztig sheaves and tensor products of integrable highest weight modules}, 
      author={Jiepeng Fang and Yixin Lan},
      year={2023},
      eprint={2310.18682},
      archivePrefix={arXiv},
      primaryClass={math.RT},
      url={https://arxiv.org/abs/2310.18682}, 
}

@article {Kashiwara-1993,
    AUTHOR = {Kashiwara, Masaki},
     TITLE = {The crystal base and {L}ittelmann's refined {D}emazure
              character formula},
   JOURNAL = {Duke Math. J.},
  FJOURNAL = {Duke Mathematical Journal},
    VOLUME = {71},
      YEAR = {1993},
    NUMBER = {3},
     PAGES = {839--858},
      ISSN = {0012-7094,1547-7398},
   MRCLASS = {17B37 (17B10 17B67)},
  MRNUMBER = {1240605},
MRREVIEWER = {Kailash\ C.\ Misra},
       DOI = {10.1215/S0012-7094-93-07131-1},
       URL = {https://doi.org/10.1215/S0012-7094-93-07131-1},
}

@article {Kashiwara-1994,
    AUTHOR = {Kashiwara, Masaki},
     TITLE = {Crystal bases of modified quantized enveloping algebra},
   JOURNAL = {Duke Math. J.},
  FJOURNAL = {Duke Mathematical Journal},
    VOLUME = {73},
      YEAR = {1994},
    NUMBER = {2},
     PAGES = {383--413},
      ISSN = {0012-7094,1547-7398},
   MRCLASS = {17B37 (17B10)},
  MRNUMBER = {1262212},
MRREVIEWER = {Kailash\ C.\ Misra},
       DOI = {10.1215/S0012-7094-94-07317-1},
       URL = {https://doi.org/10.1215/S0012-7094-94-07317-1},
}

@article {Li-2014,
    AUTHOR = {Li, Yiqiang},
     TITLE = {Tensor product varieties, perverse sheaves, and stability
              conditions},
   JOURNAL = {Selecta Math. (N.S.)},
  FJOURNAL = {Selecta Mathematica. New Series},
    VOLUME = {20},
      YEAR = {2014},
    NUMBER = {2},
     PAGES = {359--401},
      ISSN = {1022-1824,1420-9020},
   MRCLASS = {17B37 (14F05 14F43)},
  MRNUMBER = {3177922},
MRREVIEWER = {Volodymyr\ Mazorchuk},
       DOI = {10.1007/s00029-013-0121-y},
       URL = {https://doi.org/10.1007/s00029-013-0121-y},
}

@article {He-Nie-2024,
    AUTHOR = {He, Xuhua and Nie, Sian},
     TITLE = {Demazure product of the affine {W}eyl groups},
   JOURNAL = {Acta Math. Sinica (Chinese Ser.)},
  FJOURNAL = {Acta Mathematica Sinica. Chinese Series},
    VOLUME = {67},
      YEAR = {2024},
    NUMBER = {2},
     PAGES = {296--306},
      ISSN = {0583-1431},
   MRCLASS = {20F55 (14M15 14N35 20G25)},
  MRNUMBER = {4715011},
}

@article {Lusztig-1990,
    AUTHOR = {Lusztig, G.},
     TITLE = {Canonical bases arising from quantized enveloping algebras},
   JOURNAL = {J. Amer. Math. Soc.},
  FJOURNAL = {Journal of the American Mathematical Society},
    VOLUME = {3},
      YEAR = {1990},
    NUMBER = {2},
     PAGES = {447--498},
      ISSN = {0894-0347,1088-6834},
   MRCLASS = {17B35 (16A64)},
  MRNUMBER = {1035415},
MRREVIEWER = {Ya.\ S.\ So\u ibel\cprime man},
       DOI = {10.2307/1990961},
       URL = {https://doi.org/10.2307/1990961},
}

@article {Kashiwara-1991,
    AUTHOR = {Kashiwara, M.},
     TITLE = {On crystal bases of the {$Q$}-analogue of universal enveloping
              algebras},
   JOURNAL = {Duke Math. J.},
  FJOURNAL = {Duke Mathematical Journal},
    VOLUME = {63},
      YEAR = {1991},
    NUMBER = {2},
     PAGES = {465--516},
      ISSN = {0012-7094,1547-7398},
   MRCLASS = {17B37 (17B10 17B67)},
  MRNUMBER = {1115118},
MRREVIEWER = {Kailash\ C.\ Misra},
       DOI = {10.1215/S0012-7094-91-06321-0},
       URL = {https://doi.org/10.1215/S0012-7094-91-06321-0},
}

@article {Kang-Kashiwara-2012,
    AUTHOR = {Kang, Seok-Jin and Kashiwara, Masaki},
     TITLE = {Categorification of highest weight modules via
              {K}hovanov-{L}auda-{R}ouquier algebras},
   JOURNAL = {Invent. Math.},
  FJOURNAL = {Inventiones Mathematicae},
    VOLUME = {190},
      YEAR = {2012},
    NUMBER = {3},
     PAGES = {699--742},
      ISSN = {0020-9910,1432-1297},
   MRCLASS = {17B67 (17B37)},
  MRNUMBER = {2995184},
MRREVIEWER = {Volodymyr\ Mazorchuk},
       DOI = {10.1007/s00222-012-0388-1},
       URL = {https://doi.org/10.1007/s00222-012-0388-1},
}

@article {Zheng-2014,
    AUTHOR = {Zheng, Hao},
     TITLE = {Categorification of integrable representations of quantum
              groups},
   JOURNAL = {Acta Math. Sin. (Engl. Ser.)},
  FJOURNAL = {Acta Mathematica Sinica (English Series)},
    VOLUME = {30},
      YEAR = {2014},
    NUMBER = {6},
     PAGES = {899--932},
      ISSN = {1439-8516,1439-7617},
   MRCLASS = {17B37 (14F05 18D05)},
  MRNUMBER = {3200442},
MRREVIEWER = {Volodymyr\ Mazorchuk},
       DOI = {10.1007/s10114-014-3631-4},
       URL = {https://doi.org/10.1007/s10114-014-3631-4},
}

@article {Webster-2015,
    AUTHOR = {Webster, Ben},
     TITLE = {Canonical bases and higher representation theory},
   JOURNAL = {Compos. Math.},
  FJOURNAL = {Compositio Mathematica},
    VOLUME = {151},
      YEAR = {2015},
    NUMBER = {1},
     PAGES = {121--166},
      ISSN = {0010-437X,1570-5846},
   MRCLASS = {17B37 (17B10 18D05)},
  MRNUMBER = {3305310},
MRREVIEWER = {Kevin\ D.\ Coulembier},
       DOI = {10.1112/S0010437X1400760X},
       URL = {https://doi.org/10.1112/S0010437X1400760X},
}

@article {Khovanov-Lauda-2009,
    AUTHOR = {Khovanov, Mikhail and Lauda, Aaron D.},
     TITLE = {A diagrammatic approach to categorification of quantum groups.
              {I}},
   JOURNAL = {Represent. Theory},
  FJOURNAL = {Representation Theory. An Electronic Journal of the American
              Mathematical Society},
    VOLUME = {13},
      YEAR = {2009},
     PAGES = {309--347},
      ISSN = {1088-4165},
   MRCLASS = {17B37},
  MRNUMBER = {2525917},
MRREVIEWER = {Fan\ Xu},
       DOI = {10.1090/S1088-4165-09-00346-X},
       URL = {https://doi.org/10.1090/S1088-4165-09-00346-X},
}

@article {Rouquier-2012,
    AUTHOR = {Rouquier, Rapha\"el},
     TITLE = {Quiver {H}ecke algebras and 2-{L}ie algebras},
   JOURNAL = {Algebra Colloq.},
  FJOURNAL = {Algebra Colloquium},
    VOLUME = {19},
      YEAR = {2012},
    NUMBER = {2},
     PAGES = {359--410},
      ISSN = {1005-3867,0219-1733},
   MRCLASS = {20C08 (16G20 17B10 17B67 18D10)},
  MRNUMBER = {2908731},
MRREVIEWER = {Vanessa\ Miemietz},
       DOI = {10.1142/S1005386712000247},
       URL = {https://doi.org/10.1142/S1005386712000247},
}

@article {Varagnolo-Vasserot-2011,
    AUTHOR = {Varagnolo, M. and Vasserot, E.},
     TITLE = {Canonical bases and {KLR}-algebras},
   JOURNAL = {J. Reine Angew. Math.},
  FJOURNAL = {Journal f\"ur die Reine und Angewandte Mathematik. [Crelle's
              Journal]},
    VOLUME = {659},
      YEAR = {2011},
     PAGES = {67--100},
      ISSN = {0075-4102,1435-5345},
   MRCLASS = {17B37 (16T20)},
  MRNUMBER = {2837011},
MRREVIEWER = {Nicolas\ Jacon},
       DOI = {10.1515/CRELLE.2011.068},
       URL = {https://doi.org/10.1515/CRELLE.2011.068},
}

@article {Schiffmann-Vasserot-2000,
    AUTHOR = {Schiffmann, O. and Vasserot, E.},
     TITLE = {Geometric construction of the global base of the quantum
              modified algebra of {$\widehat{\mathfrak{gl}}_n$}},
   JOURNAL = {Transform. Groups},
  FJOURNAL = {Transformation Groups},
    VOLUME = {5},
      YEAR = {2000},
    NUMBER = {4},
     PAGES = {351--360},
      ISSN = {1083-4362,1531-586X},
   MRCLASS = {17B37 (14M15)},
  MRNUMBER = {1800532},
MRREVIEWER = {Andrei\ G.\ Bytsko},
       DOI = {10.1007/BF01234797},
       URL = {https://doi.org/10.1007/BF01234797},
}

@article {Beilinson-Lusztig-MacPherson-1990,
    AUTHOR = {Beilinson, A. A. and Lusztig, G. and MacPherson, R.},
     TITLE = {A geometric setting for the quantum deformation of {${\rm
              GL}_n$}},
   JOURNAL = {Duke Math. J.},
  FJOURNAL = {Duke Mathematical Journal},
    VOLUME = {61},
      YEAR = {1990},
    NUMBER = {2},
     PAGES = {655--677},
      ISSN = {0012-7094,1547-7398},
   MRCLASS = {17B37 (16S30 20G99)},
  MRNUMBER = {1074310},
MRREVIEWER = {Jie\ Du},
       DOI = {10.1215/S0012-7094-90-06124-1},
       URL = {https://doi.org/10.1215/S0012-7094-90-06124-1},
}

@article {Ginzburg-Vasserot-1993,
    AUTHOR = {Ginzburg, Victor and Vasserot, \'Eric},
     TITLE = {Langlands reciprocity for affine quantum groups of type
              {$A_n$}},
   JOURNAL = {Internat. Math. Res. Notices},
  FJOURNAL = {International Mathematics Research Notices},
      YEAR = {1993},
    NUMBER = {3},
     PAGES = {67--85},
      ISSN = {1073-7928,1687-0247},
   MRCLASS = {17B37 (17B67)},
  MRNUMBER = {1208827},
MRREVIEWER = {Jie\ Du},
       DOI = {10.1155/S1073792893000078},
       URL = {https://doi.org/10.1155/S1073792893000078},
}

@incollection {Lusztig-1999,
    AUTHOR = {Lusztig, G.},
     TITLE = {Aperiodicity in quantum affine {$\mathfrak{gl}_n$}},
      NOTE = {Sir Michael Atiyah: a great mathematician of the twentieth
              century},
   JOURNAL = {Asian J. Math.},
  FJOURNAL = {Asian Journal of Mathematics},
    VOLUME = {3},
      YEAR = {1999},
    NUMBER = {1},
     PAGES = {147--177},
      ISSN = {1093-6106,1945-0036},
   MRCLASS = {17B37 (20G42)},
  MRNUMBER = {1701926},
MRREVIEWER = {Toshiyuki\ Tanisaki},
       DOI = {10.4310/AJM.1999.v3.n1.a7},
       URL = {https://doi.org/10.4310/AJM.1999.v3.n1.a7},
}

@incollection {Lusztig-2000,
    AUTHOR = {Lusztig, George},
     TITLE = {Transfer maps for quantum affine {$\mathfrak{sl}_n$}},
 BOOKTITLE = {Representations and quantizations ({S}hanghai, 1998)},
     PAGES = {341--356},
 PUBLISHER = {China High. Educ. Press, Beijing},
      YEAR = {2000},
      ISBN = {7-04-009073-2},
   MRCLASS = {17B37 (16S30 17B67)},
  MRNUMBER = {1802182},
MRREVIEWER = {Naihuan\ Jing},
}

@article {WY-2007,
    AUTHOR = {Webster, Ben and Yakimov, Milen},
     TITLE = {A {D}eodhar-type stratification on the double flag variety},
   JOURNAL = {Transform. Groups},
  FJOURNAL = {Transformation Groups},
    VOLUME = {12},
      YEAR = {2007},
    NUMBER = {4},
     PAGES = {769--785},
      ISSN = {1083-4362},
   MRCLASS = {14M15},
  MRNUMBER = {2365444},
MRREVIEWER = {Ben Ntatin},
       DOI = {10.1007/s00031-007-0061-8},
       URL = {https://doi.org/10.1007/s00031-007-0061-8},
}

@article {Fu-Shoji-2018,
    AUTHOR = {Fu, Qiang and Shoji, Toshiaki},
     TITLE = {Positivity properties for canonical bases of modified quantum
              affine {$\mathfrak{sl}_n$}},
   JOURNAL = {Math. Res. Lett.},
  FJOURNAL = {Mathematical Research Letters},
    VOLUME = {25},
      YEAR = {2018},
    NUMBER = {2},
     PAGES = {535--559},
      ISSN = {1073-2780,1945-001X},
   MRCLASS = {17B37},
  MRNUMBER = {3826834},
MRREVIEWER = {Serge\ M.\ Skryabin},
       DOI = {10.4310/MRL.2018.v25.n2.a10},
       URL = {https://doi.org/10.4310/MRL.2018.v25.n2.a10},
}

@article {Nakajima-1998,
    AUTHOR = {Nakajima, Hiraku},
     TITLE = {Quiver varieties and {K}ac-{M}oody algebras},
   JOURNAL = {Duke Math. J.},
  FJOURNAL = {Duke Mathematical Journal},
    VOLUME = {91},
      YEAR = {1998},
    NUMBER = {3},
     PAGES = {515--560},
      ISSN = {0012-7094,1547-7398},
   MRCLASS = {17B67 (14D25 16G20 17B35 53C25 58F05)},
  MRNUMBER = {1604167},
MRREVIEWER = {Michael\ M.\ Kapranov},
       DOI = {10.1215/S0012-7094-98-09120-7},
       URL = {https://doi.org/10.1215/S0012-7094-98-09120-7},
}

@article {Nakajima-2001,
    AUTHOR = {Nakajima, Hiraku},
     TITLE = {Quiver varieties and tensor products},
   JOURNAL = {Invent. Math.},
  FJOURNAL = {Inventiones Mathematicae},
    VOLUME = {146},
      YEAR = {2001},
    NUMBER = {2},
     PAGES = {399--449},
      ISSN = {0020-9910,1432-1297},
   MRCLASS = {17B37 (17B67)},
  MRNUMBER = {1865400},
MRREVIEWER = {Olivier\ G.\ Schiffmann},
       DOI = {10.1007/PL00005810},
       URL = {https://doi.org/10.1007/PL00005810},
}

@misc{Rouquier-2008,
      title={2-{K}ac-{M}oody algebras}, 
      author={Raphael Rouquier},
      year={2008},
      eprint={0812.5023},
      archivePrefix={arXiv},
      primaryClass={math.RT},
      url={https://arxiv.org/abs/0812.5023}, 
}

@misc{HX,
      title={Total positivity in twisted flag varieties}, 
      author={He, Xuhua and Kaitao Xie},
      year={2026},
      eprint={2602.09350},
      archivePrefix={arXiv},
      primaryClass={math.RT},
      url={https://arxiv.org/abs/2602.09350}, 
}

@article {Li-2013,
    AUTHOR = {Li, Yiqiang},
     TITLE = {On projective modules in category {$\mathscr O_{\rm int}$} of
              quantum {$\mathfrak{sl}_2$}},
   JOURNAL = {Algebr. Represent. Theory},
  FJOURNAL = {Algebras and Representation Theory},
    VOLUME = {16},
      YEAR = {2013},
    NUMBER = {5},
     PAGES = {1315--1332},
      ISSN = {1386-923X,1572-9079},
   MRCLASS = {17B37 (17B10)},
  MRNUMBER = {3102956},
MRREVIEWER = {Pasha\ Zusmanovich},
       DOI = {10.1007/s10468-012-9358-y},
       URL = {https://doi.org/10.1007/s10468-012-9358-y},
}

@article {McGerty-2012,
    AUTHOR = {McGerty, Kevin},
     TITLE = {On the geometric realization of the inner product and
              canonical basis for quantum affine {$\mathfrak{sl}_n$}},
   JOURNAL = {Algebra Number Theory},
  FJOURNAL = {Algebra \& Number Theory},
    VOLUME = {6},
      YEAR = {2012},
    NUMBER = {6},
     PAGES = {1097--1131},
      ISSN = {1937-0652,1944-7833},
   MRCLASS = {20G42 (17B37 20G43)},
  MRNUMBER = {2968635},
MRREVIEWER = {Yiqiang\ Li},
       DOI = {10.2140/ant.2012.6.1097},
       URL = {https://doi.org/10.2140/ant.2012.6.1097},
}

@article {Fan-Li-2021,
    AUTHOR = {Fan, Zhaobing and Li, Yiqiang},
     TITLE = {Positivity of canonical bases under comultiplication},
   JOURNAL = {Int. Math. Res. Not. IMRN},
  FJOURNAL = {International Mathematics Research Notices. IMRN},
      YEAR = {2021},
    NUMBER = {9},
     PAGES = {6871--6931},
      ISSN = {1073-7928,1687-0247},
   MRCLASS = {16S30 (14M15 16T20)},
  MRNUMBER = {4251292},
MRREVIEWER = {Luz\ Adriana\ Mej\'ia Casta\~no},
       DOI = {10.1093/imrn/rnz047},
       URL = {https://doi.org/10.1093/imrn/rnz047},
}

@article {Fang-Lan-2025,
    AUTHOR = {Fang, Jiepeng and Lan, Yixin},
     TITLE = {Canonical bases of tensor products of integrable highest
              weight modules arising from framed constructions},
   JOURNAL = {Int. Math. Res. Not. IMRN},
  FJOURNAL = {International Mathematics Research Notices. IMRN},
      YEAR = {2025},
    NUMBER = {24},
     PAGES = {Paper No. rnaf360, 29},
      ISSN = {1073-7928,1687-0247},
   MRCLASS = {20G42},
  MRNUMBER = {5002704},
       DOI = {10.1093/imrn/rnaf360},
       URL = {https://doi.org/10.1093/imrn/rnaf360},
}

@article {Lusztig-2023,
    AUTHOR = {Lusztig, George},
     TITLE = {The quantum group {$\dot U$} and flag manifolds over the
              semifield {Z}},
   JOURNAL = {Bull. Inst. Math. Acad. Sin. (N.S.)},
  FJOURNAL = {Bulletin of the Institute of Mathematics. Academia Sinica. New
              Series},
    VOLUME = {18},
      YEAR = {2023},
    NUMBER = {3},
     PAGES = {235--267},
      ISSN = {2304-7909,2304-7895},
   MRCLASS = {20G42 (17B37)},
  MRNUMBER = {4661463},
MRREVIEWER = {Huafeng\ Zhang},
}

@article {Fu-2021,
    AUTHOR = {Fu, Qiang},
     TITLE = {The comultiplication of modified quantum affine
              {$\mathfrak{sl}_n$}},
   JOURNAL = {Commun. Contemp. Math.},
  FJOURNAL = {Communications in Contemporary Mathematics},
    VOLUME = {23},
      YEAR = {2021},
    NUMBER = {1},
     PAGES = {Paper No. 1950048, 13},
      ISSN = {0219-1997,1793-6683},
   MRCLASS = {17B37 (20C08 20G43)},
  MRNUMBER = {4169240},
MRREVIEWER = {Jorge\ A.\ Vargas},
       DOI = {10.1142/s0219199719500482},
       URL = {https://doi.org/10.1142/s0219199719500482},
}

\end{document}